\documentclass[reqno]{amsart}

\usepackage{amsmath,amssymb,amscd,amsthm,accents} \usepackage{graphicx}
\DeclareGraphicsExtensions{.eps} \usepackage{mathrsfs}
\usepackage[mathcal]{eucal}
\usepackage{enumitem}

\def\sideremark#1{\ifvmode\leavevmode\fi\vadjust{\vbox to0pt{\vss 
			\hbox to 0pt{\hskip\hsize\hskip1em          
				\vbox{\hsize3cm\tiny\raggedright\pretolerance10000%
					\noindent #1\hfill}\hss}\vbox to8pt{\vfil}\vss}}}%

\newtheorem{Thm}{Theorem}{\bfseries}{\itshape}
\newtheorem*{Thm*}{Theorem}{\bfseries}{\itshape}
\newtheorem{Cor}{Corollary}{\bfseries}{\itshape}
\newtheorem{Prop}[Cor]{Proposition}{\bfseries}{\itshape}
\newtheorem{Lem}[Cor]{Lemma}{\bfseries}{\itshape}
\newtheorem*{Lem*}{Lemma}{\bfseries}{\itshape}
\newtheorem{Fact}[Cor]{Fact}{\bfseries}{\itshape}
{\bfseries}{\itshape}
\newtheorem{Def}[Cor]{Definition}{\bfseries}{\rmfamily}
\newtheorem{Ex}[Cor]{Example}{\scshape}{\rmfamily}
\newtheorem{Rem}[Cor]{Remark}{\scshape}{\rmfamily}
{\bfseries}{\itshape}

\renewcommand\ge{\geqslant} \renewcommand\le{\leqslant}
\let\tildeaccent=\~ \let\hataccent=\^
\renewcommand\~[1]{\widetilde{#1}}

\def\<{\left<} \def\>{\right>} \def\({\left(} \def\){\right)}

\def\abs#1{\left\vert #1 \right\vert} \def\norm#1{\left\Vert #1
  \right\Vert} 

\let\parasymbol=\S \def\secref#1{\parasymbol\ref{#1}}

\def\pd#1#2{\frac{\partial#1}{\partial#2}}

 \def\arg{\operatorname{Arg}}

\let\polishL=l \def\Zoladek.{\.Zol\c adek}

 \def\const{\operatorname{const}}

\def\Re{\operatorname{Re}} \def\Im{\operatorname{Im}}
\def\Arg{\operatorname{Arg}} \def\dist{\operatorname{dist}}
\def\length{\operatorname{length}}
\def\diam{\operatorname{diam}} 
 \def\SL{\operatorname{SL}}
\def\etc.{\emph{etc}.}
 
\def\sign{\operatorname{sgn}}

\def\:{\colon} \def\R{{\mathbb R}} \def\C{{\mathbb C}} \def\Z{{\mathbb
    Z}} \def\N{{\mathbb N}} \def\Q{{\mathbb Q}} \def\P{{\mathbb P}}
\def\H{{\mathbb H}}
\def\D{{\mathbb D}}

\def\A{{\mathbb A}}

\let\PolishL=\L 
\def\L{{\mathbb L}}

 \def\e{\varepsilon} \def\S{\varSigma}

\def\poly{\operatorname{poly}}

 \def\d{\,\mathrm d}

 \def\Lojas.{\PolishL ojasiewicz}
 
\def\cP{{\mathcal P}} \def\cR{{\mathcal R}}

\def\cF{{\mathcal F}}  \def\cR{{\mathcal R}}
\def\cI{{\mathcal I}}  
\def\cT{{\mathcal T}} \def\cD{{\mathcal D}} \def\cC{{\mathcal C}}
\def\cS{{\mathcal S}} 
\def\cA{{\mathcal A}}
\def\cO{{\mathcal O}}
\def\cX{{\mathcal X}}
\def\cY{{\mathcal Y}}

 \def\trans{\pitchfork}

\def\rest#1{{\vert_{#1}}}

\def\supp{\operatorname{supp}}

\def\id{\operatorname{id}}
\def\spec{\operatorname{Spec}}

\def\O{\mathcal{O}}

\def\LT{\operatorname{LT}}

\def\vf{{\mathbf f}}
\def\vg{{\mathbf g}}
\def\vx{{\mathbf x}}
\def\vz{{\mathbf z}}
\def\vw{{\mathbf w}}
\def\vy{{\mathbf y}}
\def\vu{{\mathbf u}}
\def\vp{{\mathbf p}}
\def\vsigma{{\boldsymbol\sigma}}
\def\valpha{{\boldsymbol\alpha}}
\def\vgamma{{\boldsymbol\gamma}}

\def\vdelta{{\boldsymbol\delta}}

\def\vrho{{\boldsymbol\rho}}
\def\vsigma{{\boldsymbol\sigma}}
\def\vnu{{\boldsymbol\nu}}

\def\Exp{\operatorname{Exp}}
\def\Qua{\mathcal{Q}}
\def\an{{\text{an}}}

\def\he#1{{\{#1\}}}
\def\hrho{{\he\rho}}
\def\hvrho{{\he\vrho}}
\def\hsigma{{\he\sigma}}

\def\alg{\mathrm{alg}}
\def\trans{\mathrm{trans}}
\def\RE{\mathrm{RE}}

\def\sec{{\mathfrak{S}}}

\DeclareMathOperator{\lcm}{lcm}
\DeclareMathOperator{\pos}{pos}

\begin{document}

\title{Complex Cellular Structures}

\author{Gal Binyamini}
\thanks{G.B. is the incumbent of the Dr. A. Edward Friedmann career development chair in mathematics.}
\address{Weizmann Institute of Science, Rehovot, Israel}
\email{gal.binyamini@weizmann.ac.il}

\author{Dmitry Novikov}
\address{Weizmann Institute of Science, Rehovot, Israel}
\email{dmitry.novikov@weizmann.ac.il}
\thanks{This research was supported by the ISRAEL SCIENCE FOUNDATION
	(grant No. 1167/17) and by funding received from the MINERVA
	Stiftung with the funds from the BMBF of the Federal Republic of
	Germany. This project has received funding from the European Research Council (ERC) under the European Union’s Horizon 2020 research and innovation programme (grant agreement No 802107)}


\date{\today}

\begin{abstract}
  We introduce the notion of a \emph{complex cell}, a complexification
  of the cells/cylinders used in real tame geometry. For
  $\delta\in(0,1)$ and a complex cell $\cC$ we define its holomorphic
  extension $\cC\subset\cC^\delta$, which is again a complex cell. The
  hyperbolic geometry of $\cC$ within $\cC^\delta$ provides the class
  of complex cells with a rich geometric function theory absent in the
  real case. We use this to prove a complex analog of the cellular
  decomposition theorem of real tame geometry. In the algebraic case
  we show that the complexity of such decompositions depends
  polynomially on the degrees of the equations involved.

  Using this theory, we refine the Yomdin-Gromov algebraic lemma on
  $C^r$-smooth parametrizations of semialgebraic sets: we show that
  the number of $C^r$ charts can be taken to be polynomial in the
  smoothness order $r$ and in the complexity of the set. The algebraic
  lemma was initially invented in the work of Yomdin and Gromov to
  produce estimates for the topological entropy of $C^\infty$
  maps. For \emph{analytic} maps our refined version, combined with
  work of Burguet, Liao and Yang, establishes an optimal refinement of
  these estimates in the form of tight bounds on the tail entropy and
  volume growth. This settles a conjecture of Yomdin who proved the
  same result in dimension two in 1991. A self-contained proof of
  these estimates using the refined algebraic lemma is given in an
  appendix by Yomdin.

  The algebraic lemma has more recently been used in the study of
  rational points on algebraic and transcendental varieties. We use
  the theory of complex cells in these two directions. In the
  algebraic context we refine a result of Heath-Brown on interpolating
  rational points in algebraic varieties. In the transcendental
  context we prove an interpolation result for (unrestricted)
  logarithmic images of subanalytic sets.
\end{abstract}
\maketitle
\date{\today}

\setcounter{tocdepth}{1}
{\small \tableofcontents}

\section{Introduction}
\label{sec:intro}

In this section we aim to provide an intuitive motivation for the
formal notion of complex cells introduced
in~\secref{sec:complex-cells}. We discuss the smooth parametrization
problem, particularly the Yomdin-Gromov algebraic lemma and its
sharper forms established in this paper. We then discuss the hyperbolic
obstruction to a Gromov-Yomdin type lemma in the holomorphic category
and motivate the notion of a complex cell as a way of overcoming this
obstruction. Finally we briefly discuss the applications of our
parametrization results in smooth dynamics and in diophantine
geometry.

\subsection{Smooth parametrizations}
\label{sec:intro-smooth}

\subsubsection{The Yomdin-Gromov Algebraic Lemma}
\label{sec:intro-algebraic-lemma}

In \cite{yomdin:entropy,yomdin:gy} Yomdin gave a proof of Shub's
entropy conjecture for $C^\infty$-maps. A key step in this proof was
the construction of $C^r$-smooth parametrizations of semialgebraic
sets with the number of charts depending only on the combinatorial
data. More precisely, for a $C^r$-smooth function $f:U\to\R^n$ on a
domain $U\subset\R^m$ we denote by $\norm{f}=\norm{f}_U$ the maximum
norm on $U$ and
\begin{equation}\label{eq:r-norm-def}
  \norm{f}_r := \max_{|\valpha|\le r} \frac{\norm{D^\valpha f}}{\valpha!}.
\end{equation}
The original parametrization theorem of
\cite{yomdin:entropy,yomdin:gy} involved a technical condition on the
removal of a small piece from the semialgebraic set being
parametrized. The following refinement by Gromov \cite{gromov:gy},
see also \cite{burguet:alg-lemma}, is known as the Yomdin-Gromov
Algebraic Lemma.

\begin{Thm*}[\protect{\cite[3.3.~Algebraic Lemma]{gromov:gy}}]
  Let $X\subset[0,1]^n$ be a semialgebraic set of dimension $\mu$
  defined by conditions $p_j(\vx)=0$ or $p_j(\vx)<0$ where $p_j$ are
  polynomials and $\sum\deg p_j=\beta$. Let $r\in\N$. There exists a
  constant $C=C(n,\mu,r,\beta)$ and semialgebraic maps
  $\phi_1,\ldots,\phi_C:(0,1)^\mu\to X$ such that their images cover
  $X$ and $\norm{\phi_j}_r\le 1$ for $j=1,\ldots,C$.
\end{Thm*}

Understanding the behavior of the constant $C(n,\mu,r,\beta)$ has been
the key difficulty in establishing a conjectural sharpening of
Yomdin's results for analytic maps, as we discuss
in~\secref{sec:intro-dynamics}. More precisely, what is needed is an
estimate $C(n,\mu,r,\beta)=\poly_n(r,\beta)$
(see~\secref{sec:asymptotics} for this asymptotic notation). Such an
estimate is our first main result.

\begin{Thm}\label{thm:refined-alg-lemma}
  In the Yomdin-Gromov algebraic lemma one may take
  $C(n,\mu,r,\beta)=\poly_n(\beta)\cdot r^\mu$. Moreover the maps
  $\phi_j$ can be chosen to be semialgebraic of complexity
  $\poly_n(\beta,r)$.
\end{Thm}

\subsubsection{Pila-Wilkie's generalization to o-minimal structures}

Parametrizations by $C^r$-smooth functions have also been used to
great effect in the seemingly unrelated study of rational points on
algebraic and transcendental sets. The algebraic lemma was first
introduced to this subject by Pila and Wilkie \cite{pila-wilkie}, who
proved its far-reaching generalization for arbitrary o-minimal
structures. We refer the reader to \cite{vdd:book} for general
background on o-minimal geometry.

\begin{Thm*}[\protect{\cite[Theorem~2.3]{pila-wilkie}}]\label{thm:yg-pw}
  Let $X=\{X_p\subset[0,1]^n\}$ be a family of sets definable in an
  o-minimal structure, with $\dim X_p\le\mu$. There exists a constant
  $C=C(X,r)$ such that for any $p$ there exist definable maps
  $\phi_1,\ldots,\phi_C:(0,1)^\mu\to X_p$ such that their images cover
  $X_p$ and $\norm{\phi_j}_r\le 1$ for $j=1,\ldots,C$.
\end{Thm*}

In the diophantine direction, understanding the behavior of $C(X,r)$
with respect to the geometry of the family $X$ and the smoothness
order $r$ is also a question of great importance as we discuss
in~\secref{sec:intro-rational}. In the recent paper \cite{cpw:params}
Cluckers, Pila and Wilkie made significant progress in this direction
by proving that for subanalytic sets (and also sets in the larger
structure $\R_\an^{\text{pow}}$) one has $C(X,r)=\poly_X(r)$. For
subanalytic sets we obtain a similar result with a precise control
over the degree.

\begin{Thm}
  In the Pila-Wilkie algebraic lemma for $\R_\an$ one may take
  $C(X,r)=O_X(r^\mu)$ where $\mu:=\max_p(\dim X_p)$.
\end{Thm}

\subsubsection{Mild parametrizations}

It is natural to inquire whether one can in fact replace the
$C^r$-charts in the algebraic lemma with $C^\infty$ charts with
appropriate control over the derivatives. In \cite{pila:mild} Pila
introduced a notion of this type called \emph{mild parametrization}
and investigated its diophantine applications related to the Wilkie
conjecture.

\begin{Def}[Mild parametrization]\label{def:mild}
  A smooth map $f:U\to[0,1]^n$ on a domain $U\subset\R^m$ is said
  to be $(A,K)$-mild if
  \begin{equation}\label{eq:mild-def}
    \norm{ D^\valpha f } \le \valpha! (A|\valpha|^K)^{|\valpha|}, \qquad \forall\valpha\in\N^m.
  \end{equation}
\end{Def}

In \cite{jmt:mild} it is shown that every subanalytic set
$X\subset [0,1]^n$ admits $(A,0)$-mild (i.e. analytic)
parametrizations, but this result is not uniform over families and its
diophantine applications are restricted to the case of curves
\cite{jmt:mild} and surfaces \cite{jt:pfaff-surfaces}. We prove the
following uniform version. We refer the reader to
\cite{bm:subanalytic} for general background on subanalytic geometry.

\begin{Thm}
  Let $X=\{X_p\subset[0,1]^n\}$ be a subanalytic family of sets, with
  $\dim X_p\le\mu$. There exist constants $A=A(X)$ and $C=C(X)$ such
  that for any $p$ there exist $(A,2)$-mild maps
  $\phi_1,\ldots,\phi_C:(0,1)^\mu\to X_p$ whose images cover $X_p$.  If
  $\{X_p\}$ is semialgebraic as in the Yomdin-Gromov algebraic lemma
  then one may take $A,C=\poly_n(\beta)$.
\end{Thm}

The construction of a mild parametrization is a-priori more delicate
than its $C^r$ counterpart since it requires one to control
derivatives of all orders at once. In fact, Thomas \cite{thomas:mild}
has shown that there exist o-minimal structures without definable mild
parametrizations. In our results there is also a key difference
between the $C^r$ algebraic lemma and its mild counterpart. Namely, in
the $C^r$ version the parametrizing maps are themselves subanalytic,
and can be chosen to depend subanalytically on the parameters of the
family. We make no such guarantee in the mild version. In fact, we show
that any parametrization of a family of hyperbolas by a definable
family of $\R_\an$ maps must have unbounded $C^r$-norms for some
finite $r$. More precisely we have the following sharp estimate.

\begin{Prop}\label{prop:hyp-param}
  Let
  \begin{equation}
    F=\{F_\e=(x_\e,y_\e):[0,1]\to[0,1]^2\}, \qquad \e\in[0,1],
  \end{equation}
  be an $\R_\an$-definable family of maps satisfying
  $x_\e(t)\cdot y_\e(t)=\e$. Suppose that for each $r$ there exists a
  constant $M_r$ such that
  \begin{align}\label{eq:derivatives bounded}
    \left|\partial_t^r x_\e(t)\right|&<M_r & \left|\partial_t^r y_\e(t)\right|&<M_r
  \end{align}
  whenever the derivatives are defined. Then the Euclidean length of
  $\log F_\e([0,1])$ is bounded by a constant independent of $\e\neq0$
  (here $\log$ is applied coordinate-wise).
\end{Prop}

This result is easily seen to be essentially tight: one can indeed
construct parametrizations $F$ as in Proposition~\ref{prop:hyp-param}
to cover any definable family of intervals of constant logarithmic
length in the hyperbolas. Somewhat surprisingly, our proof goes via
reduction using complex cells to the following simple statement from
geometric function theory, which is proved using a $p$-valent version
of the Koebe 1/4-theorem.

We introduce the following notation to be used in this section. A more
general form of this notation is defined
in~\secref{sec:complex-cells}.  For $r>0$ (resp. $r_2>r_1>0$) and
$\delta\in(0,1)$ we denote
\begin{align}
  D(r)&:=\{|z|<r\} & D_\circ(r)&:=\{0<|z|<r\} & A(r_1,r_2)&:=\{r_1<|z|<r_2\} \\
  D^\delta(r)&:=D(\delta^{-1}r) & D_\circ^\delta(r)&:=D_\circ(\delta^{-1}r) & A^\delta(r_1,r_2)&:=A(\delta r_1,\delta^{-1}r_2).
\end{align}

\begin{Lem}\label{lem:log-length}
  Let $\xi:D(2)\to\C\setminus\{0\}$ be holomorphic and suppose that
  $\xi$ is $p$-valent. Then the length of $\log\xi([-1,1])$ does not
  exceed $8\pi p^2$.
\end{Lem}

The proof of Proposition~\ref{prop:hyp-param} and
Lemma~\ref{lem:log-length} is given
in~\secref{sec:sec:uniform-param-Ran}.

\subsection{Holomorphic parametrizations}

In the preceding section we discussed the problem of constructing
smooth parametrizations for semialgebraic and subanalytic sets. It is
natural to ask whether such constructions can be analytically
continued to give holomorphic parametrizations in some suitable
sense. In this section we demonstrate an obstruction of a
hyperbolic-geometric nature to a naive formulation, and motivate the
notion of a complex cell as a possible way of overcoming this
obstruction.

\subsubsection{Parametrizations of analytic curves and questions of
  uniformity}
\label{sec:parametrization-uniformity}

Let $D:=D(1)$ denote the unit disc. If $X$ is a holomorphic curve and
$K\subset X$ is compact then one can always cover $K$ by images of
analytic charts $f_j:D\to X$. To avoid bad behavior near the boundary,
we can also require that each $f_j$ extends as a holomorphic map to
$D^{1/2}$ (this is an arbitrary choice which could be replaced by any
fixed neighborhood of $\bar D$). Parametrizations of this type are
useful in analysis: they allow one to study holomorphic structures on
$X$ in terms of their pullback to $D\subset D^{1/2}$, where a rich
geometric function theory is available. Moreover, the theory of
resolution of singularities shows that this type of parametrization
can be extended to parametrizations of analytic sets of any
dimension. A notion of this type was introduced by Yomdin
\cite{yomdin:entropy-analytic} under the name \emph{analytic
  complexity unit (acu)}. A similar notion was considered under the
name ``doubling coverings'' in \cite{yf:doubling}.

Ideally one would like to prove that holomorphic parametrizations of
the type described above can be made with a uniform number of charts
when $X$ is replaced by an analytic family $X_\e$ of curves (or higher
dimensional sets) and $K$ by its relatively compact subfamily. Indeed, this would be an appropriate analog of the Yomdin-Gromov algebraic lemma in the
holomorphic category. Unfortunately this is impossible, even for
simple families of algebraic curves. Consider the family $X_\e$ of
hyperbolas restricted to the polydisc of radius 2 and $K_\e$ their
restriction to a polydisc of radius 1,
\begin{align}
  X_\e &:= \{(x,y): xy=\e,\ x,y\in D^{1/2}\} & K_\e &:= X_\e\cap (D\times D).
\end{align}
The projection to the $x$-axis gives biholomorphisms
$K_\e\simeq A_\e:=A(\e,1)$ and $X_\e\simeq A_\e^{1/2}=A(\e/2,2)$. Let
$f:D^{1/2}\to X_\e\simeq A_\e^{1/2}$. If we equip $D^{1/2}$ and
$A_\e^{1/2}$ with their respective hyperbolic metrics then the
Schwarz-Pick lemma (Lemma~\ref{lem:schwarz-pick}) implies that $f$ is
distance-contracting (see~\secref{sec:fund-lemmas} for a brief
reminder on this subject). In particular the conformal diameter of
$f(D)$ in $A_\e^{1/2}$ is bounded by the conformal diameter of $D$ in
$D^{1/2}$, which is a constant independent of $\e$. On the other hand,
the conformal diameter of $K_\e\simeq A_\e$ in $A_\e^{1/2}$ tends to
infinity as $\e\to0$ (in fact with order $\log|\log\e|$, see
Remark~\ref{rem:Adelta-diameter}). It follows that to cover $K_\e$ by
images of maps as above at least $\log|\log\e|$ charts are required,
and this is easily seen to be a tight asymptotic. We note that Yomdin
\cite{yomdin:param-dim2} obtains a similar result with $|\log\e|$
under the assumption that the maps are $p$-valent for some fixed $p$.

\subsubsection{A uniform parametrization result with discs and annuli}

It turns out that in the one-dimensional case, the hyperbolic
restriction described above is essentially the only obstacle to
uniform parametrization. Namely, suppose that in addition to covering
$K_\e$ by images $f(D)$ of holomorphic maps $f:D^{1/2}\to X_\e$ we
allow also images $f(A)$ of holomorphic maps $f:A^{1/2}\to X_\e$ for
any annulus $A$. Note that the conformal diameter of $A$ in $A^{1/2}$
is not uniformly bounded: it depends on the conformal modulus of
$A$. With this added freedom one can indeed construct a
parametrization with a uniformly bounded number of charts for a family
of curves. We proceed to explain the elementary geometric
considerations that lead to this result.

To simplify our presentation we suppose that
$X=\{X_\e\}\subset\C^2\times\C_\e$ is a family of algebraic plane
curves which project properly under $\pi:(x,y)\to x$ (although a
similar local argument works in the analytic case without the
assumption of properness). We will parametrize the fibers
$K_\e:=X_\e\cap(D\times D)$. Fix $\e\in\C$. Let $\Sigma_\e\subset\C$
denote the finite set of critical values of $\pi\rest{X_\e}$. Denote
by $\nu$ the order of the monodromy group of $\pi$. Note that
$\#\Sigma_\e$ and $\nu$ are bounded in terms of $\deg X_\e$ and in
particular uniformly in $\e$. Suppose that $D$ is covered by finitely
many sets $\cC\subset\C$ of the following types:
\begin{enumerate}
\item A point $p\in\Sigma_\e$.
\item A disc $p+D(r)$ such that $p+D(2r)$ does not meet $\Sigma_\e$.
\item A punctured disc $p+D_\circ(r)$ such that $p+D_\circ(2^\nu r)$
  does not meet $\Sigma_\e$.
\item An annulus $p+A(r_1,r_2)$ such that $p+A(r_1/2^\nu,2^\nu r_2)$ does not
  meet $\Sigma_\e$.
\end{enumerate}
In case 2 one can construct a map $f:D\to p+D(r)\to X_\e$ as a lift of
$\pi\rest{X_\e}$. In cases 3,4 the lift might be multivalued with a
monodromy of finite order dividing $\nu$, but precomposing with the
map $z\to z^\nu$ we similarly obtain univalued charts $f:D\to X_\e$ or
$f:A_r\to X_\e$ (the removable singularity theorem is used to extend
from $D_\circ$ to $D$ in case 3). Moreover the assumptions on $\cC$
ensure that $f$ extends holomorphically to $D^{1/2}$ or
$A_r^{1/2}$. Taking at most $\deg X_\e$ such lifts $f$ we obtain
charts covering $X_\e\cap\pi^{-1}(\cC)$.

It remains to show that $D$ can be covered by finitely many sets $\cC$
as above, with their number depending only on the number of points in
$\Sigma_\e$ (but not their positions). This is a simple exercise in
plane geometry which we urge the reader to attempt for themselves. It
is also instructive to check that a similar statement would not hold
had we not allowed annuli as in item 4 above.  In
Figure~\ref{fig:diphyp} we illustrate such a decomposition for the
critical points $\Sigma_\e=\{\pm\sqrt\e\}$ of the hyperbola
$y^2=x^2-\e$ (this is just our original example $xy=\e$ rotated to
satisfy our assumption of proper projection to the $x$-axis).
\begin{figure}
  \centering
  \includegraphics[width=\textwidth]{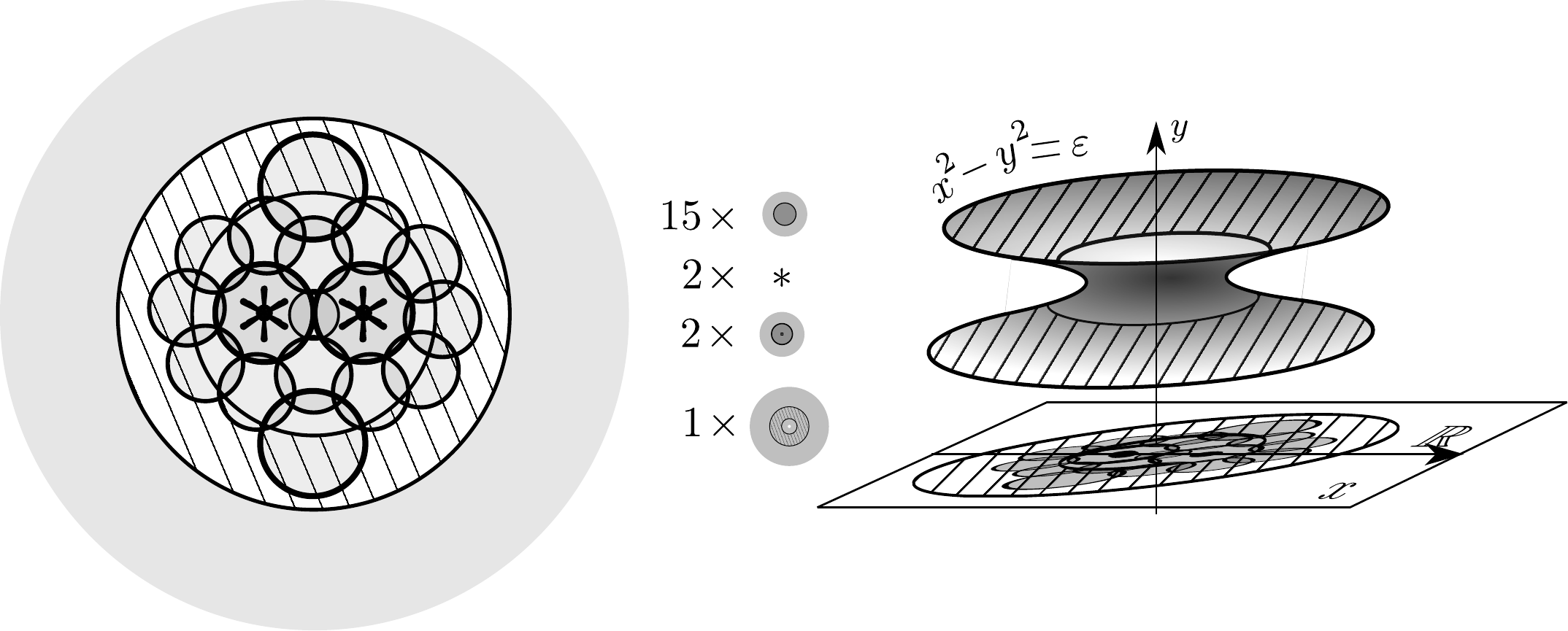}
  \caption{Covering of the hyperbola with discs and annuli, where $*$
    corresponds to points of $\Sigma_\e=\{\pm\sqrt{\e}\}$.}
  \label{fig:diphyp}
\end{figure}

Decompositions of a similar type appeared under the name
``Swiss-cheese decompositions'' in \cite{hs:variations} where they
were used in the study of Green functions on Riemann surfaces. They
also appeared in \cite{me:inf16} where they were used in the study of
collisions of singular points of Fuchsian differential equations.

\subsubsection{Parametrizing higher dimensional sets}

It is natural to ask whether a similar ``uniform parametrization''
statement might hold in higher dimensions. A naive attempt might be to
use products of discs and annuli as the domains of
definition. However, since complex annuli (unlike discs) have
conformal moduli, it turns out to be more natural to allow domains
consisting of families of annuli with varying moduli. As an
illustrative example, in the two-dimensional case we allow domains of
the form
\begin{equation}
  \cC=D_\circ(1)\odot A(\vz_1,2) := \{(\vz_1,\vz_2) : 0<|\vz_1|<1,\ |\vz_1|<|\vz_2|<2\}.
\end{equation}
Note that we plug a complex number $\vz_2$ as the radius of the
annulus $A(\vz_2,2)$, so that the radius is the absolute value of a
holomorphic function on the base. This allows us to define a kind of
analytic extension of cells as we show below. The notation $\odot$
above is an example of a more general notation introduced
in~\secref{sec:complex-cells}, and the domain above is an example of a
\emph{complex cell}. The definition closely resembles the notion of
cells/cylinders in semialgebraic/subanalytic geometry, with the order
relation $x<y$ replaced by the complex $|x|<|y|$.

To generalize the parametrization result from curves to higher
dimensional sets we require not only an analog of the domains
$D,D_\circ$ and $A$, but also an analog of their extensions
$D^{1/2},D_\circ^{1/2},A^{1/2}$. Correspondingly our complex cells
$\cC$ are endowed with a natural notion of $\delta$-extension
$\cC\subset\cC^\delta$. In the example above
\begin{multline}
  \cC^\delta = D_\circ^\delta(1)\odot A^\delta(\vz_1,2) = D_\circ(\delta^{-1})\odot A(\delta\vz_1,2\delta^{-1}) :=\\
  \{(\vz_1,\vz_2) : 0<|\vz_1|<\delta^{-1},\ \delta|\vz_1|<|\vz_2|<2\delta^{-1}\}.
\end{multline}
This is defined for any $1/2<\delta<1$: we require the original fiber
$A(\vz_1,2)$ to remain an annulus over the extension $D^\delta(1)$,
and with $\delta<1/2$ the fiber $A(\vz_1,2)$ becomes empty for
$\vz_1\in D^\delta(1)\setminus D^{1/2}(1)$.

Our main result, Theorem~\ref{thm:cpt}, shows in particular that one
can uniformize any family of analytic sets uniformly in the parameters
of the family if one is permitted to use general complex cells as the
domains for the charts.  More generally, Theorem~\ref{thm:cpt}
provides a cellular analog of the constructions of local resolution of
singularities (LRS), see e.g. \cite{bm:subanalytic}. Loosely speaking
it allows one to transform a collection of analytic functions into
normal crossings uniformly over families - using complex cells as the
domains of the charts. In the algebraic category,
Theorem~\ref{thm:cpt} gives effective control (polynomial in the
degrees) on the number and complexity of the charts in terms of the
complexity of the functions being transformed into normal
crossings. This is where the most important step toward improving the
asymptotics of the Yomdin-Gromov algebraic lemma takes place, and the
proof makes extensive use of the hyperbolic properties of a cell $\cC$
viewed as a subset of its extension $\cC^\delta$.

\subsubsection{Sectorial parametrizations}

In~\secref{sec:parametrization-uniformity} we discussed the problem of
establishing an analog of the algebraic lemma with holomorphic
functions admitting extensions to whole discs, and showed that (unless
one also allows annuli and more general complex cells) such an analog
is impossible. A more modest goal might be to establish an analog of
the algebraic lemma where the parametrizing maps admit analytic
continuation to some suitable large domain (though not a full
disc). If the domain is sufficiently large, one could then hope to
control the derivatives using complex-analytic methods, for instance
the Cauchy estimates.

It turns out that the correct domains for analytic continuation are
complex sectors. More specifically, let $\sec(\e)$ denote the sector
$\sec(\e)=\{|\Arg z|<\e,|z|<2\}$, and for $B=(0,1)^\mu$ let $B(\e)$ denote
the direct product of $\mu$ copies of $\sec(\e)$. All of our refinements
of the algebraic lemma follow from the following statement about
parametrizations with complex-analytic continuation to sectors.

\begin{Thm}
  Let $X=\{X_p\subset[0,1]^n\}$ be a subanalytic family of sets, with
  $\dim X_p\le\mu$. Set $B=(0,1)^\mu$. There exist constants $C=C(X)$
  and $\e=\e(X)$ such that for any $p$ there exist maps
  $\phi_1,\ldots,\phi_C:B\to X_p$ whose images cover $X_p$. Moreover
  each $\phi_j$ extends holomorphically to $B(\e)$ and has unit
  $C^1$-norm there. If $\{X_p\}$ is semialgebraic as in the
  Yomdin-Gromov algebraic lemma then one may take
  $C,\e^{-1}=\poly_n(\beta)$.
\end{Thm}

This theorem is proved (in a more general form)
in~\secref{sec:sectorial-param}. We start by constructing a
complex-cellular parametrization for $X_p$, and then show how to
construct maps from $B$ into a complex cell which extend to $B(\e)$
with bounded $C^1$-norms. The $C^r$ and mild statements of the
algebraic lemma are deduced from the sectorial version by a simple
(and self-contained) argument based on the Cauchy estimates
in~\secref{sec:param-proofs}.

\subsection{Applications in dynamics}
\label{sec:intro-dynamics}

The Yomdin-Gromov algebraic lemma was first invented for the purpose
of studying the properties of the topological entropy of smooth maps.
We recall the basic notions of topological entropy theory, Shub's
entropy conjecture and Yomdin's result for $C^\infty$-maps. We then
discuss a refinement of these results in the analytic category
culminating in Theorem~\ref{thm:analytic-entropy} confirming in
arbitrary dimension a conjecture of Yomdin that was previously known
only in dimension two. Theorem~\ref{thm:analytic-entropy} follows
immediately from the refined algebraic lemma in combination with the
work of Burguet, Liao and Yang \cite{bly}. A self-contained proof of
Theorem~\ref{thm:analytic-entropy} using the refined algebraic lemma
is given in~Appendix~\ref{appendix:yomdin} (by Yomdin). Below we
mostly follow the notations and presentation of \cite{bly}.

\subsubsection{Topological entropy}

Let $M$ be a compact metric space with metric $d:M^2\to\R_{\ge0}$ and
let $f:M\to M$ be a continuous map. For $n\in\N$ define the $n$-th
iterated metric by
\begin{equation}
  d_n:M\times M\to\R_{\ge0}, \qquad d_n(x,y)=\max_{i=0,\ldots,n-1} d(f^{\circ i}(x),f^{\circ i}(y)).
\end{equation}
If $M$ is a metric space and $\Lambda\subset M$, a set $K\subset M$ is
called $\e$-spanning for $\Lambda$ if the $\e$-neighborhood of $K$
contains $\Lambda$. For any subset $\Lambda\subset M$, $\e>0$ and
$n\in\N$ let $r_n(f,\Lambda,\e)$ denote the minimal cardinality of an
$\e$-spanning set of $\Lambda$ with respect to the $d_n$-metric.

The \emph{$\e$-topological} and \emph{topological} entropies of
$\Lambda$ are defined by
\begin{align}
  h(f,\Lambda,\e) &= \limsup_{n\to\infty} \frac1n \log r_n(f,\Lambda,\e), & h(f,\Lambda) &= \lim_{\e\to0} h(f,\Lambda,\e).
\end{align}
We set $h(f,\e):=h(f,M,\e)$ and $h(f):=h(f,M)$. This latter quantity
is called the \emph{topological entropy} of $f$. As its name implies,
the topological entropy is independent of the choice of metric $d$. It
was first defined by Adler, Konheim and McAndrew \cite{akm:entropy} in
a purely topological fashion, and later shown by Bowen
\cite{bowen:entropy-group} to be equivalent to the metric definition
presented above.

\subsubsection{Tail entropy}

For $x\in M$ we define the \emph{infinite dynamical ball}
$B_\infty(x,\e)$ by
\begin{equation}
  B_\infty(x,\e) := \{y\in M: d_n(x,y)<\e\quad \forall n\in\N\}.
\end{equation}
Following Bowen \cite{bowen:entropy-expansive}, the \emph{$\e$-tail}
and \emph{tail} entropies of $f$ are defined by
\begin{align}
  h^*(f,\e) &= \sup_{x\in M} h(f,B_\infty(x,\e)), & h^*(f)=\lim_{\e\to0} h^*(f,\e).
\end{align}
Bowen \cite{bowen:entropy-expansive} has shown that the $\e$-tail
entropy bounds the difference between the $\e$-entropy and the
entropy, that is
\begin{equation}\label{eq:entropy-v-tail}
  |h(f)-h(f,\e)|\le h^*(f,\e).
\end{equation}
If $h^*(f,\e)=0$ for some $\e>0$ then the system $(f,M)$ is called
entropy-expansive ($h$-expansive); if $h^*(f)=0$ then it is called
\emph{asymptotically entropy expansive} (asymptotically
$h$-expansive). This notion of tail entropy was first introduced (with
a different definition) by Misiurewicz \cite{misiurewicz:cond-entropy}
under the name \emph{conditional entropy}. Misiurewicz showed that for
asymptotically $h$-expansive systems the measure-theoretical entropy
is upper-semicontinuous, and such systems therefore always admit an
invariant measure of maximal entropy.

\subsubsection{Yomdin's results for smooth maps}
\label{sec:intro-yomdin-thm}

Assume from now on that $M$ is a compact $C^\infty$-smooth manifold
equipped with a Riemannian metric. In \cite{yomdin:entropy} Yomdin
proved Shub's entropy conjecture for $C^\infty$ maps. This conjecture
states that the logarithm of the spectral radius $\spec f$ of
$f_*:H_*(M,\R)\to H_*(M,\R)$ is a lower bound for $h(f)$. In fact it
is relatively easy to show (by comparing volumes) that
\begin{equation}\label{eq:spec-vs-h}
  \log\spec f\le h(f)+\max_{l=1,\ldots,\dim M} v_l^0(f)
\end{equation}
where $v_l^0(f)$ is the $l$-dimensional \emph{local volume growth}, a
quantity related to $h^*(f)$ whose precise definition we postpone to
Appendix~\ref{appendix:sec:volume-growth}. Yomdin's fundamental result
was that if $f$ is a $C^r$-map then
\begin{equation}
  v_l^0(f) \le \frac{l R(f)}{r}, \qquad
  R(f) = \lim_{n\to\infty} \frac1n \sup_{x\in M} \log \norm{D_x f^{\circ n}}
\end{equation}
For $C^\infty$-maps this implies $v_l^0(f)=0$, and in combination
with~\eqref{eq:spec-vs-h} yields Shub's entropy conjecture. Buzzi
\cite{buzzi} later observed that Yomdin's argument also implied the
same estimate $h^*(f)\le (\dim M/r) R(f)$ for the tail entropy. For
$C^\infty$-maps this implies $h^*(f)=0$, i.e. that $C^\infty$-maps are
asymptotically $h$-expansive.

\subsubsection{Controlling $h^*(f,\e)$ for analytic maps}

Bounding $h^*(f,\e)$ explicitly as a function of $\e$ is an important
problem with consequences for the computation of $h(f)$ (for instance
using~\eqref{eq:entropy-v-tail}) and for the study of the
semicontinuity properties of the entropy. However, for arbitrary
$C^\infty$ maps $f:M\to M$ the rate of convergence of $h^*(f,\e)$ to
zero as $\e\to0$ can be arbitrarily slow, as shown by the examples of
Burguet, Liao and Yang \cite[Theorem~L]{bly}.

As before, the asymptotic study of $h^*(f,\e)$ can be related to the
asymptotic study of the local volume growth $v_l^0(f,\e)$, whose
precise definition we again postpone
to Appendix~\ref{appendix:sec:volume-growth}. In
\cite{yomdin:entropy-analytic} Yomdin considered the case of a
real-analytic map $f:M\to M$ of a compact analytic surface $M$. In
this context, using a holomorphic variant of the algebraic lemma based
on the Bernstein inequality for polynomials, he was able to prove that
\begin{equation}
  v_1^0(f,\e)\le C(f)\cdot \frac{\log|\log\e|}{|\log\e|}.
\end{equation}
Yomdin conjectured \cite[Conjecture~6.1]{yomdin:entropy-analytic} that
a similar estimate should hold (for $v^0_l$) for $M$ of arbitrary
dimension, but the limitation of the holomorphic parametrization
technique to dimension one prevented such a generalization (see
\cite{yomdin:param-dim2} for some results in this direction in
dimension two). In \cite[Theorem~N]{bly} it is shown that the bound in
Yomdin's conjecture is essentially sharp (for a class of functions
slightly larger than the analytic class).

In \cite{bly} Burguet, Liao and Yang revisited the problem of the rate
of convergence of the tail entropy for analytic, and more general,
maps. The precise statement of their main result is technical and
depends on the behavior of the algebraic lemma's constant
$C(n,\mu,r,\beta)$. However, under the hypothesis
$C(n,\mu,r,\beta)=\poly_n(r,\beta)$ the results of
\cite[Corollary~B]{bly} take a particularly simple form: namely, they
imply the generalization of Yomdin's result (both for $v^0_l$ and
$h^*$) to arbitrary dimension. It is also shown in \cite{bly} that
$C(n,1,r,\beta)=\poly_n(r,\beta)$ and this recovers Yomdin's result
for analytic surfaces.

To summarize, following the work of \cite{bly} it has been clear that
proving that $C(n,\mu,r,\beta)=\poly_n(r,\beta)$ is the missing step
to establishing Yomdin's conjecture. Combined with the refined
algebraic lemma proved in this paper this is now a theorem.

\begin{Thm}\label{thm:analytic-entropy}
  Let $M$ be a compact analytic manifold and $f:M\to M$ be an analytic
  map. Then
  \begin{align}\label{eq:analytic-entropy}
    h^*(f,\e) &\le C(f)\cdot\frac{\log|\log\e|}{|\log\e|}, &
    v_l^0(f,\e)&\le C(f)\cdot\frac{\log|\log\e|}{|\log\e|}
  \end{align}
  for $l=1,\ldots,\dim M$.
\end{Thm}

An independent proof of the volume-growth part of
Theorem~\ref{thm:analytic-entropy} is given
in~\secref{sec:analytic-entropy-proof}.

\subsection{Applications in diophantine geometry}
\label{sec:intro-rational}

For $p\in\P^\ell(\Q)$ we define $H(p)$ to be $\max_i |\vp_i|$ where
$\vp\in\Z^{\ell+1}$ is a projective representative of $p$ with
\begin{equation}
  \gcd(\vp_0,\ldots,\vp_\ell)=1.
\end{equation}
For $\vx\in\A^\ell(\Q)$ we define its height to be the height of
$\iota(\vx)$ for the standard embedding
\begin{equation}
  \iota:\A^\ell\to\P^\ell,\qquad \iota(\vx_{1..\ell}) = (1:\vx_1:\cdots:\vx_\ell).
\end{equation}
For a set $X\subset\P^\ell(\R)$ we denote
\begin{equation}
 X(\Q,H) := \{ \vx\in X\cap \P(\Q)^\ell : H(x)\le H\},
\end{equation}
and similarly for $X\subset\R^\ell$.

In \cite{bombieri-pila} Bombieri and Pila introduced the interpolation
determinant method for estimating the quantity $\#X(\Q,H)$ as a
function of $H$ when $X$ is the graph of a $C^r$ (or $C^\infty$)
smooth function $f:[0,1]\to\R$. It turns out that two very different
asymptotic behaviors are obtained depending on whether the graph of
$f$ belongs to an algebraic plane curve. The algebraic lemma has been
used in both of these directions, to generalize from graphs of
functions to more general sets. In Appendix~\ref{sec:bp-complex} we
give a complex-cellular analog of the Bombieri-Pila interpolation
determinant method and use it to deduce some applications for both the
algebraic and transcendental contexts. We briefly describe the main
results below.

\subsubsection{A result for algebraic varieties}

We prove the following result.

\begin{Thm}\label{thm:improved-marmon}
  Let $X\subset\P(\C)^\ell$ be an irreducible algebraic variety of
  dimension $m$ and degree $d$. Then $X(\Q,H)$ is contained in $N$
  hypersurfaces of degree $k$, none of which contain $X$, where
  \begin{equation}
    N = \poly_\ell(d,k,\log H)\cdot H^{(m+1)d^{-1/m}(1+\poly_\ell(d)/k)}.
  \end{equation}
\end{Thm}

We make note of two particular choices for $k$, namely
\begin{align}
  k&=\poly_\ell(d)/\e & &\implies & N &= \poly_\ell(d,1/\e)\cdot H^{(m+1)d^{-1/m}+\e} \\
  k&=\poly_\ell(d)\cdot\log H & &\implies & N &= \poly_\ell(d,\log H)\cdot H^{(m+1)d^{-1/m}}.
\end{align}
The first choice improves the dependence on the degree $d$ to
polynomial in Heath-Brown's result \cite{heath-brown:density} for
hypersurfaces and Broberg's result \cite[Theorem~1]{broberg:note} (see
also Marmon \cite{marmon}). The second choice replaces an $H^\e$
factor by a power of $\log H$, similar to the various results
established for curves by Pila \cite{pila:pems} and Salberger
\cite{salberger:density}. We remark that in the case of curves, Walsh
\cite{walsh:boundd-rational} recently proved a result eliminating the
$\log H$ factor altogether and it would be interesting to study
whether this can be generalized to arbitrary dimension. We briefly
survey the history of these results
in Appendix~\ref{sec:alg-density-history} and prove
Theorem~\ref{thm:improved-marmon} in Appendix~\ref{sec:alg-density-proof}.

\subsubsection{The Pila-Wilkie theorem}
\label{sec:intro-pw}

Let $f:I\to\R^2$ and suppose that $X_f:=f(I)$ is
transcendental. Bombieri and Pila \cite{bombieri-pila} and Pila
\cite{pila:density-Q} used the interpolation determinant method to show that
$\#X(\Q,H)=O_{X,\e}(H^\e)$ for any $\e>0$. To generalize this result
to higher dimensions, let $X\subset\R^\ell$ and denote by $X^\alg$ the
union of all positive dimensional connected semialgebraic sets
contained in $X$ and by $X^\trans:=X\setminus X^\alg$. Pila and Wilkie
\cite{pila-wilkie} proved that for any $X$ definable in an o-minimal
expansion of $\R$ we have
\begin{equation}
  \#X^\trans(\Q,H)=O_{X,\e}(H^\e) \qquad \forall\e>0.
\end{equation}
Moreover, if $X$ varies over a definable family then the asymptotic
constants can be taken uniform over the family. The o-minimal version
of the Yomdin-Gromov algebraic lemma plays the central role in the
proof of the Pila-Wilkie theorem.

\subsubsection{The Wilkie conjecture}
Wilkie \cite{pila-wilkie} has conjectured that for sets definable in
$\R_{\exp}$ the asymptotic in the Pila-Wilkie theorem can be improved
to $\poly_X(\log H)$, and it is natural to make similar conjectures
for other ``natural'' geometric structures (although such a result
does not hold for general subanalytic curves, see
\cite[Example~7.5]{pila:subanalytic-dilation}). One of the key
obstacles to proving the Wilkie conjecture, at least following the
Pila-Wilkie strategy, is to obtain an estimate on $C(X,r)$ which is
polynomial in $r$ and the complexity of $X$. Some one-dimensional and
some restricted two-dimensional cases of the Wilkie conjecture have
been established using methods involving Pfaffian functions and mild
parametrizations
\cite{pila:pfaff,jt:pfaff-surfaces,pila:exp-alg-surface,butler}.  In
\cite{me:analytic-interpolation} we developed a complex-analytic
approach to the Pila-Wilkie theorem in the subanalytic case, and in
\cite{me:rest-wilkie} we used this approach to prove the Wilkie
conjecture for the structure $\R^\RE$ generated by the exponential and
trigonometric functions \emph{restricted} to some finite interval.

\subsubsection{Applications to unlikely intersections in diophantine geometry}
The Pila-Wilkie theorem has been applied to great effect in the study
of unlikely intersections in diophantine geometry. We briefly mention
Pila-Zannier's proof the the Manin-Mumford conjecture
\cite{pila-zannier}, Pila's proof of the Andr\'e-Oort conjecture for
modular curves \cite{pila:andre-oort} and Masser-Zannier's work on
simultaneous torsion points in elliptic families \cite{mz:torsion}. We
refer the reader to \cite{zannier:book,scanlon:survey} for a
survey. Some of the most striking diophantine applications,
particularly around the study of modular curves and Shimura varieties,
require the generality of the Pila-Wilkie theorem for $\R_{\an,\exp}$,
i.e. beyond the subanalytic category. Obtaining polylogarithmic
asymptotics for such unrestricted sets is the natural challenge going
beyond the scope of \cite{me:rest-wilkie}.

\subsubsection{Interpolating rational points on log-sets}

For definiteness let $\log_e:\C\setminus\{0\}\to\C$ denote the
principal branch of the standard complex logarithm. Below $\log$ can
denote $\lambda\log_e(\cdot)$ for any $\lambda\in\C\setminus\{0\}$,
i.e.  logarithm with respect to any (fixed) base. We extend $\log$
coordinate-wise as a mapping $\log:(\C\setminus\{0\})^\ell\to\C^\ell$.

\begin{Def}[Log-set]\label{def:log-set}
  If $A\subset(\C\setminus\{0\})^\ell$ is a bounded subanalytic set
  (where $\C^\ell$ is identified with $\R^{2\ell}$) then we call
  $\log A\subset\C^\ell$ a log-set.
\end{Def}

Any bounded subanalytic set is a log-set, but log-sets are of course
more general as they involve application of an \emph{unrestricted}
logarithm. In Appendix~\ref{sec:log-sets-pila} we show that log-sets
appear naturally in the application of the Pila-Wilkie theorem to the
Andr\`e-Oort conjecture for modular curves in Pila's work
\cite{pila:andre-oort}. They are therefore perhaps the most natural
candidate to consider looking for results in the direction of the
Wilkie conjecture for unrestricted sets (with an eye to the
diophantine applications). We prove the following result in this
direction.

\begin{Prop}\label{prop:log-set-pw}
  Let $X=\{X_\lambda\subset[0,1]^\ell\}$ be an $\R_\an$-definable
  family and  let $n$ denote the maximal dimension of the fibers of
  $X$. Fix $k\in\N$ and assume $k=\Omega_\ell(1)$. There exists an
  $\R_\an$-definable family $Y=\{Y_{\lambda,\mu}\subset X_\lambda\}$
  with maximal fiber dimension strictly smaller than $n$ and
  $\kappa={O_X(k^{-1/n})}$ such that the following holds. For any
  parameter $\lambda$ and any integer $H>1$ there is a collection of
  $N=\poly_\ell(k)\cdot H^\kappa$ parameters $\{\mu_j\}$ satisfying
  \begin{equation}
    (\log X_\lambda)^\trans(\Q,H)\subset \bigcup_{j=1}^N \log Y_{\lambda,\mu_j}
  \end{equation}
  Each $Y_{\lambda,\mu}$ is defined by intersecting $X_\lambda$ with a
  collection of polynomial equations of degree $k$ in the variables
  $\log \vx_1,\ldots,\log \vx_\ell$.
\end{Prop}

Proposition~\ref{prop:log-set-pw} is reminiscent of the main inductive
step in the proof of the Pila-Wilkie theorem. In particular, choosing
$k$ such that $\kappa<\e$ and using induction over the fiber
dimension, the proposition immediately yields a proof of the
Pila-Wilkie theorem for log-sets. On the other hand, setting
$k=(\log H)^n$ yields an interpolation result using a constant number
of hypersurfaces of polylogarithmic degree. This can be seen as the
first inductive step toward the Wilkie conjecture. This is similar to
the main result of \cite{cpw:params} for sets definable in the
o-minimal structure $\R_{\an}^{\text{pow}}$ obtained by adjoining all
unrestricted power functions $x^\lambda:\R_+\to\R_+$ for
$\lambda\in(0,\infty)$ to $\R_\an$.

The proof of Proposition~\ref{prop:log-set-pw} is given in
Appendix~\ref{sec:log-set-pw-proof}. It relies heavily on the theory
of complex cells, and is effective in the following sense: \emph{a
  generalization of Theorem~\ref{thm:cpt}, with polynomial complexity
  bounds, to a reduct of $\R_\an$ which includes the \emph{restricted}
  logarithm would automatically yield a proof of the Wilkie conjecture
  for log-sets $\log A$ with $A$ definable in the reduct}. While we do
not prove such a generalization in this paper (and there appear to be
some technical obstacles to doing so), this seems to offer a plausible
approach to proving the Wilkie conjecture for a large class of
unrestricted sets that are important in applications.

Even for the case that $k$ is a constant independent of $H$,
Proposition~\ref{prop:log-set-pw} yields additional information
compared to the standard proof of the Pila-Wilkie theorem. Namely, in
the standard proof the sets $Y_{\lambda,\mu}$ are also defined by
intersecting $X_\lambda$ with a collection of polynomial equations of
degree $k$ in the variables $\log \vx_1,\ldots,\log \vx_\ell$. In
general such sets would not be subanalytic. In our version certain
cancellations in the logarithmic terms allow one to find suitable
equations which are \emph{subanalytic} on $X_\lambda$, thus proving
that $Y_{\lambda,\mu}$ are also subanalytic. This eventually implies
that the part of $X_\lambda^\alg$ responsible for the presence of many
points of height $H$ is also a union of log-sets.

\subsection{A remark on asymptotic notations}
\label{sec:asymptotics}

In this paper each appearance of an expression $\alpha=O(\beta)$
should be interpreted as shorthand notation for $\alpha\le C\cdot\beta$
where $C$ is a universally fixed positive constant (which may be
different at each occurrence). Similarly, $\alpha=O_X(\beta)$ is a
shorthand for $\alpha\le C_X\cdot\beta$ where $X\to C_X$ is a
universally fixed, positive valued real function. We similarly write
$\alpha=\Omega_X(\beta)$ as shorthand for $\alpha\ge C_X\cdot\beta$
and $\alpha=\Theta_X(\beta)$ as shorthand for $\alpha=O_X(\beta)$ and
$\alpha=\Omega_X(\beta)$. Finally we write $\alpha=\poly_X(\beta)$ as
a shorthand for $\alpha\le P_X(\beta)$ where $X\to P_X$ is a
universally fixed mapping whose values are univariate polynomials with
positive coefficients.

To avoid confusion, we stress that our notation serve as mere
placeholders for specific constants that are left unspecified for
brevity, rather than expressing asymptotic behavior as parameters tend
to infinity or zero. For example, the condition
$1/\rho=\poly_\ell(\nu)$ in Lemma~\ref{lem:admissible-collision} means
that there is a specific polynomial $P_\ell(\cdot)$ depending on
$\ell$, such that for any $\rho$ satisfying $1/\rho<P_\ell(\nu)$ the
conclusion of the Lemma holds.

\subsection{Acknowledgements}

We wish to thank Yosef Yomdin for introducing us to the Yomdin-Gromov
algebraic lemma and encouraging us to pursue its refinements. We also
wish to thank Pierre Milman and Andrei Gabrielov for many helpful
discussions before and during the preparation of this paper. Finally
we wish to thank the anonymous referees for exceptionally detailed and
invaluable comments improving many aspects of the paper.

\section{Complex cellular structures}
\label{sec:complex-cells}

\subsection{Discs and annuli in the complex plane}

For $r\in\C$ (resp. $r_1,r_2\in\C$) with $|r|>0$
(resp. $|r_2|>|r_1|>0$) we denote
\begin{equation}\label{eq:basic-fibers}
  \begin{aligned}
    D(r)&:=\{|z|<|r|\} & D_\circ(r)&:=\{0<|z|<|r|\} \\
    A(r_1,r_2)&:=\{|r_1|<|z|<|r_2|\} & *&:=\{0\}.
  \end{aligned}
\end{equation}
We also set $S(r):=\partial D(r)$. For any $0<\delta<1$ we define
the $\delta$-extensions, denoted by superscript $\delta$, by
\begin{equation}\label{eq:fiber-delta-ext}
  \begin{aligned}
    D^\delta(r)&:=D(\delta^{-1}r) & D^\delta_\circ(r)&:=D_\circ(\delta^{-1}r) \\
    A^\delta(r_1,r_2)&:=A(\delta r_1,\delta^{-1}r_2) & *^\delta&:=*.
  \end{aligned}
\end{equation}
We also set $S^\delta(r)=A(\delta r,\delta^{-1}r)$. The notion of
$\delta$-extension is naturally associated with the Euclidean geometry
of the complex plane. However, in many cases it is more convenient to
use a different normalization associated with the hyperbolic geometry
of our domains. For any $0<\rho<\infty$ we define the
$\hrho$-extension $\cF^\hrho$ of $\cF$ to be $\cF^\delta$ where
$\delta$ satisfies the equations
\begin{equation}\label{eq:rho-ext-def}
  \begin{aligned}
    \rho &= \frac{2\pi\delta}{1-\delta^2} && \text{for $\cF$ of type
      $D$,} \\
    \rho &= \frac{\pi^2}{2|\log\delta|} && \text{for $\cF$ of type
      $D_\circ,A$}.
  \end{aligned}
\end{equation}

The motivation for this notation comes from the following fact,
describing the hyperbolic-metric properties of a domain $\cF$ within
its $\hrho$-extension. For a reminder on the hyperbolic metric
associated with a planar domain see~\secref{sec:fund-lemmas}.
Fact~\ref{fact:boundary-length} is proved by explicit computation
in~\secref{sec:explicit-consts}.

\begin{Fact}\label{fact:boundary-length}
  Let $\cF$ be a domain of type $A,D,D_\circ$ and let $S$ be a component
  of the boundary of $\cF$ in $\cF^\hrho$. Then the length of $S$ in
  $\cF^\hrho$ is at most $\rho$.
\end{Fact}

\subsection{Complex cells}

\subsubsection{The general setting}
\label{sec:cells-general-setting}

We introduce a notation that will be used throughout the paper. Let
$\cX,\cY$ be sets and $\cF:\cX\to2^\cY$ be a map taking points of
$\cX$ to subsets of $\cY$. Then we denote
\begin{equation}\label{eq:odot-def}
  \cX\odot\cF := \{(x,y) : x\in\cX, y\in\cF(x)\}.
\end{equation}
In this paper $\cX$ will be a subset of $\C^n$ and $\cY$ will be
$\C$. If $r:\cX\to\C\setminus\{0\}$ then for the purpose of this
notation we understand $D(r)$ to denote the map assigning to each
$x\in\cX$ the disc $D(r(x))$, and similarly for $D_\circ,A$.

We now introduce the central notion of this paper, namely the notion
of a complex cell of \emph{length} $\ell\in\Z_{\ge0}$ and \emph{type}
$\cT(\cC)\subset\{*,D,D_\circ,A\}^\ell$. If $U$ is a complex manifold
we denote by $\cO(U)$ the space of holomorphic functions on $U$, and
by $\cO_b(U)\subset\cO(U)$ the subspace of bounded functions. As a
shorthand we denote $\vz_{1..\ell}=(\vz_1,\ldots,\vz_\ell)$.

\begin{Def}[Complex cells]\label{def:cells}
  A complex cell $\cC$ of \emph{length} zero is the point $\C^0$. The
  \emph{type} of $\cC$ is the empty word.  A complex cell of length
  $\ell+1$ has the form $\cC_{1..\ell}\odot\cF$ where the \emph{base}
  $\cC_{1..\ell}$ is a cell of length $\ell$, and the \emph{fiber}
  $\cF$ is one of $*,D(r),D_\circ(r),A(r_1,r_2)$ where
  $r\in\cO_b(\cC_{1..\ell})$ satisfies $|r(\vz_{1..\ell})|>0$ for
  $\vz_{1..\ell}\in\cC_{1..\ell}$; and
  $r_1,r_2\in\cO_b(\cC_{1..\ell})$ satisfy
  $0<|r_1(\vz_{1..\ell})|<|r_2(\vz_{1..\ell})|$ for
  $\vz_{1..\ell}\in\cC_{1..\ell}$. The type $\cT(\cC)$ is
  $\cT(\cC_{1..\ell})$ followed by the type of the fiber $\cF$.
\end{Def}
Next, we define the notion of a $\vdelta$-extension (resp.
$\hvrho$-extension) of a cell of length $\ell$ where
$\vdelta\in(0,1)^\ell$ (resp. $\vrho\in(0,\infty)^\ell$).
\begin{Def}\label{def:cell-ext}
  The cell of length zero is defined to be its own
  $\vdelta$-extension. A cell $\cC$ of length $\ell+1$ admits a
  $\vdelta$-extension
  $\cC^\vdelta:=\cC_{1..\ell}^{\vdelta_{1..\ell}}\odot\cF^{\vdelta_{\ell+1}}$
  if $\cC_{1..\ell}$ admits a $\vdelta_{1..\ell}$-extension, and if
  the function $r$ (resp. $r_1,r_2$) involved in $\cF$ admits
  holomorphic continuation to $\cC_{1..\ell}^{\vdelta_{1..\ell}}$ and
  satisfies $|r(\vz_{1..\ell})|>0$
  (resp. $0<|r_1(\vz_{1..\ell})|<|r_2(\vz_{1..\ell})|$) in this larger
  domain. The $\hvrho$-extension $\cC^\hvrho$ is defined in an
  analogous manner.
\end{Def}

Note that for $0<\delta_1<\delta_2<1$ we have
$\cC^{\delta_2}\subset \cC^{\delta_1}$ and similarly for
$0<\rho_1<\rho_2$ we have
$\cC^{\he{\rho_2}}\subset\cC^{\he{\rho_1}}$. Thus a cell admitting
$\delta_1$-extension (resp. $\he{\rho_1}$-extension) automatically
admits $\delta_2$-extension (resp $\he{\rho_2}$-extension), but not
vice versa. However see Theorem~\ref{thm:cell-refinement} for a
construction that enables a partial converse.

As a shorthand, when say that $\cC^\vdelta$ is a complex cell
(resp. $\cC^\hvrho$) we mean that $\cC$ is a complex cell admitting a
$\vdelta$ (resp $\hvrho$) extension. We will usually speak about
$\delta$-extensions where $\delta\in(0,1)$ by identifying $\delta$
with $\vdelta:=(\delta,\ldots,\delta)$ and similarly with
$\rho\in(0,\infty)$.

\begin{Rem}\label{rem:repeated-ext}
  It is sometimes convenient to consider repeated extensions. We
  denote
  $\cC^{\he{\rho_1}\he{\rho_2}}:=(\cC^{\he{\rho_1}})^{\he{\rho_2}}$. A
  simple computation shows that
  $\cC^{\he{\rho_1}\he{\rho_2}}\subset\cC^\rho$ where
  $\rho^{-1}=\poly(\rho_1^{-1},\rho_2^{-1})$.
\end{Rem}

The \emph{dimension} of a cell denoted $\dim\cC$ is its length minus
the number of symbols $*$ in its type. This clearly agrees with the dimension
of $\cC$ as a complex manifold.

\subsubsection{The real setting}

We introduce the notion of a \emph{real} complex cell $\cC$, which we
refer to in this paper simply as \emph{real cells} (but note that
these are subsets of $\C^\ell$). We also define the notion of
\emph{real part} of a real cell $\cC$ (which lies in $\R^\ell$), and
of a \emph{real} holomorphic function on a real cell. Below we let
$\R_+$ denote the set of positive real numbers.

\begin{Def}[Real complex cells]\label{def:real-cells}
  The cell of length zero is real and equals its real part. A cell
  $\cC:=\cC_{1..\ell}\odot\cF$ is real if $\cC_{1..\ell}$ is real and
  the radii involved in $\cF$ can be chosen to be real holomorphic
  functions on $\cC_{1..\ell}$; The real part $\R\cC$ (resp. positive
  real part $\R_+\cC$) of $\cC$ is defined to be
  $\R\cC_{1..\ell}\odot\R\cF$ (resp. $\R_+\cC_{1..\ell}\odot\R_+\cF$)
  where $\R\cF:=\cF\cap\R$ (resp.  $\R_+\cF:=\cF\cap\R_+$) except the
  case $\cF=*$, where we set $\R*=\R_+*=*$; A holomorphic function on
  $\cC$ is said to be real if it is real on $\R\cC$.
\end{Def}
A simple induction shows that a real cell is invariant under the
conjugation $\vz\to\bar\vz$. Note that the positive real part of a
real cell is connected, while the real part is disconnected if $\cC$
has fibers of type $D_\circ,A$. We remark that for a function
$f:\cC\to\C$ to be real it is enough to require that it is real on
$\R_+\cC$. Indeed, in this case $f(z)=\overline{f(\bar z)}$ on
$\R_+\cC$ and it follows by holomorphicity the equality holds over
$\cC$.

\subsubsection{Algebraicity}

For a pure-dimensional algebraic variety $X\subset\C^\ell$ we define
the \emph{degree} $\deg X$ to be the number of intersections between
$X$ and a generic affine-linear hyperplane of complementary dimension
(this is the same as the degree of the projective closure of $X$ in
$\C P^\ell$). We extend this by linearity to arbitrary varieties, not
necessarily pure dimensional. We remind the reader that
$\deg(X\times Y)=\deg X\cdot\deg Y$ and by the Bezout theorem
$\deg(X\cap Y)\le \deg X\cdot \deg Y$.

For a domain $U\subset\C^\ell$, we say that a holomorphic function
$f:U\to\C$ is \emph{algebraic} if its graph
$G_f\subset\C^\ell\times\C$ is an analytic component of
$(U\times\C)\cap X$, where $X\subset\C^\ell\times\C$ is an algebraic
variety. The minimal degree $\beta$ of a variety $X$ satisfying this
condition is called the \emph{degree}, or \emph{complexity}, of
$f$. For $F:U\to\C^k$ we say that $F$ is algebraic if each of its
components is, and we define its complexity to be the maximum among
the complexities of the components. As an easy consequence of the
Bezout theorem we have for any pair of algebraic maps
$F:U\to\C^k,G:V\to\C^d$ where $F(U)\subset V$ the estimate
\begin{equation}
  \deg F\circ G \le \poly_{\ell,k}(\deg F,\deg G) .
\end{equation}

We define the notion of an \emph{algebraic} complex cell of complexity
$\beta$ by induction as follows: a cell of length $0$ is algebraic of
complexity $1$; A cell $\cC=\cC_{1..\ell}\odot\cF$ is algebraic if
$\cC_{1..\ell}$ is algebraic and the radii involved in $\cF$ are
algebraic, and the complexity of $\cC$ is the maximum among the
complexity of $\cC_{1..\ell}$ and the complexities of the radii defining
$\cF$.

\subsection{Cellular maps}

We equip the category of complex cells with \emph{cellular maps}
defined as follows.

\begin{Def}[Cellular map]\label{def:cell-maps}
  Let $\cC,\hat\cC$ be two cells of length $\ell$. We say that a
  holomorphic map $f:\cC\to\hat\cC$ is \emph{cellular} if it takes the
  form $\vw_{j}=\phi_j(\vz_{1..j})$ where $\phi_j\in\cO_b(\cC_{1..j})$
  for $j=1,\ldots,\ell$ and moreover $\phi_j$ is a monic polynomial of
  positive degree in $\vz_j$. We say that a cellular map $f$ is
  \emph{prepared} (resp. \emph{a translate}) in $\vz_j$ if
  $\phi_j(\vz_{1..j})=\vz_j^{q_j}+\tilde\phi_j(\vz_{1..j-1})$ for some
  $q_j\in\N_{\ge1}$ (resp. $q=1$) and $\tilde\phi_j$ is holomorphic on
  $\cC_{1..j-1}$. We say that $f$ is \emph{prepared} (resp. a
  translate) if it is prepared (resp. a translate) in
  $\vz_1,\ldots,\vz_\ell$. We say that $f$ is \emph{real} if
  $\cC,\hat\cC$ are real and the components of $f$ are real.
\end{Def}

Cellular maps preserve the length and dimension of cells. The
composition of two cellular maps is cellular, but note that the
composition of two prepared cellular maps is not necessarily
prepared. We will often be interested in covering a cell by cellular
images of other cells. Toward this end we introduce the following
definition.

\begin{Def}\label{def:cell-cover}
  Let $\cC^\vdelta$ be a cell and
  $\{f_j:\cC_j^{\vdelta'}\to\cC^\vdelta\}$ be a finite collection of
  cellular maps. We say that this collection is a
  \emph{$(\vdelta',\vdelta)$-cellular cover} of $\cC$ if
  $\cC\subset\cup_j(f_j(\cC_j))$. If $(\vdelta',\vdelta)$ are clear
  from the context we will speak simply of cellular covers.
\end{Def}

The number of maps $f_j$ in a cellular cover is called the \emph{size}
of the cover. If the maps $f_j$ are all algebraic of complexity
$\beta$ we say that the cover has complexity $\beta$.

A real cellular cover is defined as follows.

\begin{Def}\label{def:real-cell-cover}
  Let $\cC^\vdelta$ be a real cell and
  $\{f_j:\cC_j^{\vdelta'}\to\cC^\vdelta\}$ be a finite collection of
  real cellular maps. We say that this collection is a \emph{real
    $(\vdelta',\vdelta)$-cellular cover} of $\cC$ if
  $\R_+\cC\subset\cup_j(f_j(\R_+\cC_j))$. If $(\vdelta',\vdelta)$ are
  clear from the context we will speak simply of real cellular covers.
\end{Def}

The restriction to positive real parts in the definition of real
cellular covers is a notational convenience. One can cover the
remaining components of $\R\cC$, for instance using the signed
covering maps introduced in~\secref{sec:nu-cover}.

\begin{Rem}\label{rem:cell-cover-composition}
  We remark that if $\{f_j:\cC_j^{\vdelta'}\to\cC^{\vdelta}\}$ is a
  cellular cover of $\cC$ and
  $\{f_{jk}:\cC_{jk}^{\vdelta''}\to\cC_j^{\vdelta'}\}$ is a cellular cover
  of $\cC_j$ then $\{f_j\circ f_{jk}\}$ is a $(\vdelta'',\vdelta)$-cover
  of $\cC$. We will often use this basic principle without further
  reference.
\end{Rem}

In this paper when we say that a family of sets or functions is
\emph{definable} we mean that it is definable in $\R_\an$. The reader
unfamiliar with this terminology can think instead of a family whose
total space is a bounded subanalytic set.

The following theorem implies in particular, by the remark above, that
a cellular cover can always be replaced by a cellular cover consisting
of prepared maps.

\begin{Thm}[Cellular Preparation Theorem, CPrT]\label{thm:cprt}
  Let $f:\cC^\hrho\to\hat\cC$ be a (real) cellular map. Then there
  exists a (real) cellular cover $\{g_j:\cC_j^\hrho\to\cC^\hrho\}$ such that each $f\circ g_j$ is prepared.

  If $\cC^\hrho,\hat\cC,f$ vary in a definable family $\Lambda$ then the size of the
  cover is $\poly_{\Lambda}(\rho)$, and the maps $g_j$
  can be chosen from a single definable family.  If $\cC^\hrho,\hat\cC,f$
  are algebraic of complexity $\beta$ then the cover has size
  $\poly_\ell(\beta,\rho)$ and complexity $\poly_\ell(\beta)$.
\end{Thm}

In Theorem~\ref{thm:cprt}, if $\cC^\hrho,\hat\cC,f$ vary in a definable
family $\Lambda=\{\cC^{\he{\rho_\alpha}}_\alpha,\hat\cC_\alpha,f_\alpha\}_\alpha$ then we mean,
more explicitly, that there exists a single definable family
$\{G_\beta:\cC_\beta\to\hat\cC_\beta\}_\beta$ of maps such that for
every $f=f_\alpha$ and $\rho=\rho_\alpha$, every $g_j$ in the conclusion of the
theorem can be chosen as $g_j=G_{\beta_j}$ for a suitable parameter
value $\beta_j$. We use this shorthand formulation freely below.

The proof of Theorem~\ref{thm:cprt} is given
in~\secref{sec:proof-cprt}.

\subsection{The Cellular Parametrization Theorem}

Recall that in semialgebraic geometry, a cell is said to be
\emph{compatible} with a function if the function vanishes either
identically or nowhere on the cell. We introduce a complex analog
below.

\begin{Def}\label{def:compatible}
  For a complex cell $\cC$ and $F\in\cO_b(\cC)$ we say that $F$ is
  \emph{compatible} with $\cC$ if $F$ vanishes either identically or
  nowhere on $\cC$. For a cellular map $f:\hat\cC\to\cC$ we say that
  $f$ is compatible with $F$ if $f^*F$ is compatible with $\hat\cC$.
\end{Def}

The following is our main result.

\begin{Thm}[Cellular Parametrization Theorem, CPT]\label{thm:cpt}
  Let $\rho,\sigma\in(0,\infty)$. Let $\cC^\hrho$ be a (real) cell and
  $F_1,\ldots,F_M\in\cO_b(\cC^\hrho)$ (real) holomorphic
  functions. Then there exists a (real) cellular cover
  $\{f_j:\cC^\hsigma_j\to\cC^\hrho\}$ such that each $f_j$ is
  prepared and compatible with each $F_k$.

  If $\cC^\hrho,F_1,\ldots,F_M$ vary in a definable family $\Lambda$ then the cover
  has size $\poly_{\Lambda}(\rho,1/\sigma)$ and the maps
  $f_j$ can be chosen from a single definable family. If
  $\cC^\hrho,F_1,\ldots,F_M$ are algebraic of complexity $\beta$ then the
  cover has size $\poly_\ell(\beta,M,\rho,1/\sigma)$ and complexity
  $\poly_\ell(M,\beta)$.
\end{Thm}

The proof of Theorem~\ref{thm:cpt} is given in~\secref{sec:proof-cpt}.

In~\secref{sec:uniform-families} we show that the cellular structure
of the maps essentially implies automatic uniformity over families in
the statements of the CPT and CPrT. We state the family versions of
these theorems for convenience of use, to avoid having to make such
reductions on numerous occasions.

\subsection{Topology and hyperbolic geometry of complex cells}
\label{sec:cell-topology-geometry}

A complex cell $\cC$ is homotopy equivalent to a product of points
(for fibers $*,D$) and circles (for fibers $D_\circ,A$) by the map
$\vz_i\to\Arg\vz_i$. Thus $\pi_1(\cC)\simeq\prod G_i$ where $G_i$ is
trivial for $*,D$ and $\Z$ for $D_\circ,A$. We let $\gamma_i$ denote
the generator of $G_i$ chosen with positive complex orientation for
$G_i=\Z$ and $\gamma_i=e$ otherwise.

\begin{Def}\label{def:assoc-monom}
  Let $f:\cC\to\C\setminus\{0\}$ be continuous. We define the monomial
  associated with $f$ to be $\vz^{\valpha(f)}$ where
  \begin{equation}
    \valpha_i(f) = f_*\gamma_i \in \Z\simeq \pi_1(\C\setminus\{0\}).
  \end{equation}
\end{Def}

It is easy to verify that $f\mapsto z^{\valpha(f)}$ is a group homomorphism
from the multiplicative group of continuous maps
$f:\cC\to\C\setminus\{0\}$ to the multiplicative group of monomials,
which sends each monomial to itself.

For any hyperbolic Riemann surface $X$ we denote by
$\dist(\cdot,\cdot;X)$ the hyperbolic distance on $X$
(see~\secref{sec:fund-lemmas} for a reminder on this topic). We use
the same notation when $X=\C$ and $X=\R$ to denote the usual Euclidean
distance, and when $X=\C P^1$ to denote the Fubini-Study metric
normalized to have diameter $1$. For $x\in X$ and $r>0$ we denote by
$B(x,r;X)$ the open $r$-ball centered $x$ in $X$. For $A\subset X$ we
denote by $B(A,r;X)$ the union of $r$-balls centered at all points of
$A$.

If $f:\cC^\hrho\to\C\setminus\{0\}$ is a bounded holomorphic map then
we may decompose it as $f=\vz^{\valpha(f)}\cdot U(\vz)$, where
$U:\cC^\hrho\to\C\setminus\{0\}$ is a holomorphic map and
$\valpha(U)=0$. It follows that the branches of
$\log U:\cC^\hrho\to\C$ are univalued. The following lemma shows that
$U$ enjoys strong boundedness properties when restricted to $\cC$.

\begin{Lem}[Monomialization lemma]\label{lem:monomial}
  Let $0<\rho<\infty$ and let $f:\cC^\hrho\to\C\setminus\{0\}$ be a
  bounded holomorphic map. Then $f=\vz^{\valpha(f)}\cdot U(\vz)$ 
   where $\log U:\cC\to\C$ is univalued and bounded.

  If $\cC^\hrho$ and $f$ vary in a definable family $\Lambda$ then $|\valpha(f)|=O_\Lambda(1)$  and
  \begin{align}\label{eq:monom-lem-analytic}
  \diam(\log U(\cC);\C) &= O_\Lambda(\rho), & \diam(\Im\log U(\cC);\R) &= O_\Lambda(1).
  \end{align}
 If $\cC^\hrho,f$ are algebraic of
  complexity $\beta$ then $|\valpha(f)|=\poly_\ell(\beta)$ and
  \begin{align}
    \diam(\log U(\cC);\C) &< \poly_\ell(\beta)\cdot\rho, & \diam(\Im\log U(\cC);\R) &< \poly_\ell(\beta).
  \end{align}
\end{Lem}

The proof of Lemma~\ref{lem:monomial} is given
in~\secref{sec:proof-monom}.

In light of the monomialization lemma, the CPT can be viewed as a
cellular analog of the monomialization of functions/ideals in the
theory of resolution of singularities. Indeed, if $\cC^\hrho$ is a
cell compatible with $F$ then either $F$ vanishes identically or
$F:\cC^\delta\to\C\setminus\{0\}$, in which case $F$ is equivalent to
$\vz^{\valpha(f)}$ up to a unit on $\cC$; hence the cells constructed
in the CPT may be viewed as ``cellular charts'' where $F_1,\ldots,F_M$
are monomialized.

We now give several results on the geometry of holomorphic maps from
cells to hyperbolic Riemann surfaces. We begin with the following
\emph{domination lemma}, which is valid for arbitrary cellular
extensions. In this paper we never use this lemma directly: instead,
we use the finer \emph{fundamental lemmas} stated later, which are
valid only for extensions with a sufficiently small $\rho$. However we
still state and prove the domination lemma to stress another line of
close analogy between complex cells and resolution of singularities.

\begin{Lem}[Domination Lemma]\label{lem:domination}
  Let $\cC^\hrho$ be a complex cell and suppose\footnote{Note that if
    $\cC$ admits a $\hrho$-extension for $\rho<2e$ then it also admits
    $\he{2e}$ extension.} that $\rho\ge2e$. Let
  $f:\cC^\hrho\to\C\setminus\{0,1\}$ be holomorphic. Then on $\cC$ one
  of the following holds:
  \begin{align}\label{eq:dom-lemma}
    |f| &= O_\ell(\log\log\rho) & &\text{or} & |1/f| &= O_\ell(\log\log\rho).
  \end{align}
\end{Lem}

The proof of the domination lemma is given
in~\secref{sec:fund-lemmas}.

The domination lemma can be seen as a cellular analog of a standard
argument from the theory of resolution of singularities
\cite[Lemma~4.7]{bm:subanalytic}: \emph{suppose $f,g$ and $f-g$ are
  monomials (up to a unit). Then either $f$ divides $g$ or $g$ divides
  $f$}. This allows one to principalize an ideal by monomializing its
generators and their pairwise differences. To see the analogy with the
domination lemma, suppose $\cC^\hrho$ is a cell compatible with
$f,g,f-g$. Then $f/g:\cC^\hrho\to\C\setminus\{0,1\}$ and the
domination lemma implies that either $f/g$ or $g/f$ is bounded from
above (i.e. one divides the other in $\cO_b(\cC)$). Moreover, the
bound depends only on $\ell$ and $\rho$.

\begin{Rem}
  The domination lemma for one-dimensional discs immediately implies
  the Little Picard Theorem. Indeed, suppose
  $f:\C\to\C\setminus\{0,1\}$ is entire. Applying the domination lemma
  to $f\rest{D(r)}$ for every $r>0$ we see that $f$ is bounded away
  from $0$ or $\infty$ uniformly in $\C$ and therefore constant by
  Liouville's theorem.

  Similarly, the domination lemma for one-dimensional punctured discs
  implies the Great Picard Theorem. Indeed, suppose
  $f:D_\circ(1)\to\C\setminus\{0,1\}$ has an essential singularity at
  $0$.  Applying the domination lemma we see that $f$ is bounded away
  from $0$ or $\infty$ uniformly in $D_\circ(1/2)$ which is
  impossible, for instance by the removable singularity theorem.
\end{Rem}

Next, we state the more refined \emph{fundamental lemmas} on maps
$f:\cC^\hrho\to X$ into a hyperbolic Riemann surface $X$. By contrast
with the domination lemma, these lemmas yield more precise asymptotic
estimates on the image of a complex cell as $\rho$ tends to zero. The
proofs of the fundamental lemmas rely on the (hyperbolic) geometry and
topology of $X$. Rather than formulate the most general possible form,
we prefer to give three separate formulations for the most useful
cases $X=\D,\D\setminus\{0\},\C\setminus\{0,1\}$ where $\D$ denotes
the unit disc.

\begin{Lem}[Fundamental Lemma for $\D$]\label{lem:fund-D}
  Let $\cC^\hrho$ be a complex cell. Let $f:\cC^\hrho\to\D$ be
  holomorphic. Then
  \begin{equation}
    \diam(f(\cC);\D)=O_\ell(\rho).
  \end{equation}
\end{Lem}

\begin{Lem}[Fundamental Lemma for $\D\setminus\{0\}$]\label{lem:fund-Dcirc}
  Let $\cC^\hrho$ be a complex cell and $0<\rho<1$. Let
  $f:\cC^\hrho\to\D\setminus\{0\}$ be holomorphic. Then one of the
  following holds:
  \begin{align}
    f(\cC)&\subset B(0,e^{-\Omega_\ell(1/\rho)};\C) & &\text{or} & \diam(f(\cC);\D\setminus\{0\})&=O_\ell(\rho).
  \end{align}
  In particular, one of the following holds:
  \begin{align}
    \log |f(\cC)| &\subset (-\infty,-\Omega_\ell(1/\rho)) & &\text{or} & \diam(\log|\log|f(\cC)||;\R)&=O_\ell(\rho).
  \end{align}
\end{Lem}

\begin{Lem}[Fundamental Lemma for $\C\setminus\{0,1\}$]\label{lem:fund-C01}
  Let $\cC^\hrho$ be a complex cell and let
  $f:\cC^\hrho\to\C\setminus\{0,1\}$ be holomorphic. Then one of the
  following holds:
  \begin{align}\label{eq:fund-C01}
    f(\cC)&\subset B(\{0,1,\infty\},e^{-\Omega_\ell(1/\rho)};\C P^1)& &\text{or} & \diam(f(\cC);\C\setminus\{0,1\})&=O_\ell(\rho).
  \end{align}
\end{Lem}

The proofs of the three fundamental lemmas are given
in~\secref{sec:fund-lemmas}.

\subsection{The $\nu$-cover of a cell}
\label{sec:nu-cover}

A key difference between the classical notion of real cellular
decompositions and the complex counterpart is the presence of a
non-trivial fundamental group, and with it the existence of
non-trivial covering maps. Recall that for a cell $\cC$ of length
$\ell$ we identify $\pi_1(\cC)\simeq\prod G_i$ where $G_i$ is trivial
for $*,D$ and $\Z$ for $D_\circ,A$.

\begin{Def}[The $\vnu$-cover of a cell]\label{def:nu-cover}
  Let $\cC$ be a cell of length $1$.  For $\cC=D_\circ,A$ and
  $\nu\in\Z$ we define the $\nu$-cover $\cC_{\times\nu}$ by
  \begin{align}
    D_\circ(r)_{\times\nu}&:=D_\circ(r^{1/\nu}) & A(r_1,r_2)_{\times\nu}&:=A(r_1^{1/\nu},r_2^{1/\nu})
  \end{align}
  For $\cC=D(r),*$ the cover $\cC_{\times\nu}$ is defined only for
  $\nu=1$. In all cases we define $R_\nu:\cC_{\times\nu}\to\cC$ by
  $R_\nu(z)=z^\nu$.
  
  Let $\cC$ be a cell of length $\ell$ and let
  $\vnu=(\vnu_1,\ldots,\vnu_\ell)\in\pi_1(\cC)$ be such that
  $\vnu_j\vert\vnu_k$ whenever $j>k$ and $G_j=G_k=\Z$. We define the
  \emph{$\vnu$-cover} $\cC_{\times\vnu}$ of $\cC$ and the associated
  \emph{cellular map} $R_\vnu:\cC_{\times\vnu}\to\cC$ by induction on
  $\ell$. For $\cC=\cC_{1..\ell-1}\odot\cF$ we let
  \begin{equation}
    \cC_{\times\vnu}:=(\cC_{1..\ell-1})_{\times\vnu_{1..\ell-1}}
    \odot (R_{\vnu_{1..\ell-1}}^*\cF_{\times\vnu_{\ell}})
  \end{equation}
  We define $R_\vnu(\vz_{1..\ell}):=\vz^\vnu$.
\end{Def}

Note that the $D_\circ(r)_{\times\vnu_\ell}$ fiber (and similarly with
$A$) does not conform, a-priori, with the definition of a complex cell
since $r^{1/\vnu_\ell}$ may in general be multivalued (with cyclic
monodromy of order dividing $\vnu_\ell$). However the divisibility
conditions on $\vnu$ guarantee that
$R_\vnu:(\cC_{1..\ell-1})_{\times\vnu_{1..\ell-1}}\to\cC_{1..\ell-1}$
maps $\pi_1(\cC_{1..\ell-1})_{\times\vnu}$ into
$\vnu_\ell\pi_1(\cC_{1..\ell-1})$ and the pullback
$R_{\vnu_\ell}^*(D_\circ(r))_{\times\vnu_\ell}$ is indeed univalued and well
defined up to a root of unity. We will usually consider the
$\nu$-cover with $\nu\in\N$, meaning that we take $\vnu$ with $\vnu_i=\nu$
when $G_i=\Z$ and $\vnu_i=1$ otherwise.

\begin{Ex}[On the condition $\vnu_j\vert\vnu_k$ in Definition~\ref{def:nu-cover}]
  Consider the cell $\cC=D_\circ(1)\odot A(\vz_1,1)$. If we try to
  compute a cellular cover $\cC_{\times(1,2)}$ this should formally be
  given by
  \begin{equation}
    \cC_{\times(1,2)} = D_\circ(1)\odot A(\sqrt{\vz_1},1)=D_\circ(1)\odot A(\sqrt{\vw_1},1)
  \end{equation}
  where $\vw$ denotes the coordinates on $\cC_{\times(1,2)}$. This is
  not a cell in our sense since $\sqrt{\vw_1}$ is not a univalued
  function on $D_\circ(1)$. Computing $\cC_{\times(2,2)}$ on the other
  hand we have
  \begin{equation}
    \cC_{\times(2,2)} = D_\circ(1)\odot A(\sqrt{\vz_1},1)=D_\circ(1)\odot A(\vw_1,1)
  \end{equation}
  because $\vz_1=\vw_1^2$, and we do obtain a cell.
\end{Ex}

A minor technicality arises in the real setting. Namely, the real part
$\R\cC_{\times\vnu}$ does not always cover $\R\cC$: if $\vnu$ is even
then only the positive real part is covered. To cover the remaining
components of the real part we introduce the notion of a \emph{signed}
cover. Namely, for a sequence $\vsigma\in\{\pm1\}^\ell$ we define
$\cC_{\times\vnu,\vsigma}$ and $R_{\vnu,\vsigma}$ by induction as
above but taking
$R_{\vnu,\vsigma}(\vz_{1..\ell})_k:=\vsigma_k\cdot\vz_{k}^{\vnu_k}$. It
is then clear that $\R\cC_{\times\vnu,\vsigma}$ cover $\R\cC$ when
$\vsigma$ ranges over all possible signs.

The pullback to a $\vnu$-cover will be used in our treatment to
resolve the ramification of multivalued cellular maps. We record a
simple proposition concerning the interaction between extensions and
$\nu$-covers. 

\begin{Prop}\label{prop:ext-v-cover}
  Let $\cC$ be a complex cell and $\nu\in\N$.
  \begin{enumerate}
  \item If $\cC$ admits a $\delta$-extension then $\cC_{\times\nu}$
    admits a $\delta^{1/\nu}$-extension and the covering map $R_\nu$
    extends to
    $R_\nu:(\cC_{\times\nu})^{\delta^{1/\nu}}\to\cC^\delta$.
  \item If $\cC$ admits a $\hrho$-extension then $\cC_{\times\nu}$
    admits an $\he{\nu\rho}$-extension and the covering map $R_\nu$
    extends to $R_\nu:(\cC_{\times\nu})^{\he{\nu\rho}}\to\cC^\hrho$.
  \end{enumerate}
  If $\cC$ is algebraic of complexity $\beta$ then $C_{\times\nu}$ is
  algebraic of complexity $\poly_\ell(\beta,\nu)$.
\end{Prop}

We leave the simple inductive proof to the reader. We remark that for
the part 2 it is crucial that covering maps are defined to be the
identity on fibers of type $D$.

\subsection{Overview of the proof of the CPT (Theorem~\ref{thm:cpt})}
\label{sec:cpt-proof-sketch}

In this section we aim to give an intuitive overview of the proof of
the CPT. We will consider only the algebraic case, which requires the
most delicate arguments to control the complexity with respect to
$\beta$. We will construct a covering using general cellular maps
instead of prepared maps (one can later use the CPrT to obtain a
prepared cover).

The proof will proceed by induction on the length and dimension of
$\cC$: we will works with cells of length $\ell+1$, assuming that the
CPT is already established for cells of length at most $\ell$, and for
cells of length $\ell+1$ and smaller dimension. The cases $\ell=0$ or
$\dim\cC=0$ are trivial.

\subsubsection{Reduction to the case of one function ($M=1$)}

We first note that the CPT can be easily reduced to the case of one
function $F$. Indeed, suppose we wish to perform a cellular
decomposition for the functions $F_1,\ldots,F_M$. Let $F$ be the
product of all the $F_k$, except those that vanish identically on
$\cC$. We construct a cellular decomposition
$f_j:\cC_j^\hsigma\to\cC^\hrho$ compatible with $F$. Each of the maps
$f_j$ whose image is disjoint from the zeros of $F$ is already
compatible with each $F_k$. Each of the remaining cells $\cC_j$ has
dimension strictly smaller than $\cC$ since cellular maps preserve
dimension. Thus we can proceed by induction over the dimension.

\subsubsection{Reduction to arbitrary $\rho,\sigma$}

We will allow ourselves to assume that $\rho$ is as small as we wish
as long as $1/\rho=\poly_\ell(\beta)$. We will also allow ourselves to
assume that $\sigma$ is as large as we wish as long as
$\sigma=\poly_\ell(\beta)$. Both of these assumptions are justified by
the \emph{refinement theorem} (Theorem~\ref{thm:cell-refinement})
which shows that cells with a given extension can be ``refined'' into
cells with a wider extension. This theorem generalizes the idea of
covering a disc or annulus by several smaller discs or annuli. The
proof of this theorem is independent of the CPT and the reader can
safely skip the details for now.

\subsubsection{Reduction to a proper covering map}
Let $\cC:=\cC_{1..\ell}\odot\cF$. Let $\cD\in\C[\vz_{1..\ell}]$ denote
the discriminant of $F$ and construct a cellular decomposition
$f_j:\cC_j^\hrho\to\cC_{1..\ell}^\hrho$ compatible with $\cD$. Each of
the cells $\cC_j$ with $f_j(\cC_j)$ contained in the zeros of $\cD$
has dimension strictly smaller than $\cC_{1..\ell}$, and we can
therefore obtain a cellular cover for the cell $\cC_j\odot(f_j^*\cF)$
and the function $f_j^*F$ by induction over dimension. Composing this
cover with $(f_j,\id)$, see Remark~\ref{rem:cell-cover-composition},
we obtain a covering for the part of $\cC$ lying over the zeros of
$\cD$.

It remains to cover the part lying over the images $f_j(\cC_j)$ that
are disjoint from the zeros of $\cD$. For each such cell $\cC_j$ we
can reduce again to the cell $\cC_j\odot(f_j^*\cF)$ and the function
$f_j^*F$. We return to the original notation, replacing
$\cC_j\odot(f_j^*\cF),f_j^*F$ by $\cC\odot\cF,F$. What we have gained
is that the projection
\begin{equation}
  \pi: (\cC\odot\cF)^{\hrho}\cap\{F=0\}\to\cC, \qquad \pi(\vz_{1..\ell+1})=\vz_{1..\ell}
\end{equation}
is now a proper covering map. We denote the degree of $\pi$ by $\nu$,
and note that $\nu=\poly_\ell(\beta)$.

\subsubsection{Covering the zeros}

If we set $\hat\cC:=\cC_{\times\nu!}$ then one can choose $\nu$
univalued sections $y_j:\hat\cC\to\C$ of $\pi$ which are pointwise
disjoint and cover the fiber of $\pi$. However, the covering of order
$\nu!$ has complexity exponential in $\nu$ and therefore in $\beta$,
and thus cannot be used explicitly in our construction.  In fact, one
can show that $y_j$ descends to a univalued function
$y_j:\cC_{\times\nu_j}\to\C$ where $\nu_j\le\nu$ is the size of the
$\pi_1(\cC)$-orbit of $y_j$, see Lemma~\ref{lem:y_j-monodromy}. This
is a simple statement about permutation groups, using crucially the
fact that $\pi_1(\cC)$ is abelian. We denote
$\hat\cC_j:=\cC_{\times\nu_j}$.

Using the sections $y_j$ we construct cellular maps
$(R_{\nu_j},\id):\hat\cC_j\odot*\to\cC\odot\cF$ whose images are
contained in the zeros of $F$ and hence compatible with $F$. Moreover
Proposition~\ref{prop:ext-v-cover} shows that these maps admit
$\he{\rho\cdot\nu_j}$-extensions, and choosing $\rho\cdot\nu_j<\sigma$
gives the required extension property.

\subsubsection{Covering the complement of the zeros}

We now make a simplifying assumption, namely we replace $\cF$ by
$\C$. This simplifies the presentation by allowing us to avoid minor
technicalities related to $\partial\cF$. Since $\C$ is unbounded we
will formally have to allow unbounded cells in the covering that we
construct, but the difference will be minor. The goal is therefore to
cover each fiber $\C\setminus\{y_1,\ldots,y_\nu\}$ by a finite
collection of discs, punctured discs and annuli admitting
$\hsigma$-extensions. The challenge is that this covering should
depend holomorphically on the base point $\vz_{1..\ell}\in\cC$, and
should be of size $\poly_\ell(\beta)$. The remainder
of~\secref{sec:cpt-proof-sketch} is devoted to addressing this
challenge.

\subsubsection{The affine invariants $s_{i,j,k}$}

To achieve our goal we will study the relative positions of the points
$y_j\in\C$. Toward this end we introduce the following notation. Let
$y_i,y_j,y_k$ be three distinct sections and denote
$\hat\cC_{i,j,k}:=\cC_{\times\nu_i\nu_j\nu_k}$. We define a map
$s_{i,j,k}$ as follows,
\begin{equation}
  s_{i,j,k}:\hat\cC_{i,j,k}\to\C\setminus\{0,1\}, \qquad s_{i,j,k}=\frac{y_i-y_j}{y_i-y_k}.
\end{equation}
By the fundamental lemma for maps into $\C\setminus\{0,1\}$
(Lemma~\ref{lem:fund-C01}) one of the following holds:
\begin{equation}\label{eq:intro-sijk-fund}
  \begin{gathered}
    s_{i,j,k}(\hat\cC_{i,j,k})\subset B(\{0,1,\infty\},e^{-\Omega_\ell(1/(\nu^3\rho))};\C P^1) \quad\text {or}\\
    \diam(s_{i,j,k}(\hat\cC_{i,j,k});\C\setminus\{0,1\})=O_\ell(\nu^3\rho).
  \end{gathered}
\end{equation}
We remark that equivalently~\eqref{eq:intro-sijk-fund} holds if we
replace $\hat\cC_{i,j,k}$ by $\hat\cC$, since $s_{i,j,k}$ on $\hat\cC$
factors through $\hat\cC_{i,j,k}$. We note that $s_{i,j,k}$ is
invariant under affine transformations of $\C$, and is in fact the
only affine invariant of $(y_i,y_j,y_k)$. We think
of~\eqref{eq:intro-sijk-fund} as implying that the relative positions
of any three sections change very little in $\hat\cC$ for
$\rho\ll1$. This is the crucial ingredient from hyperbolic geometry
that will enable us to carry out the decomposition into discs and
annuli uniformly over the base $\hat\cC_{i,j,k}$.

\begin{Rem}[Fulton-MacPherson compactification]
  In \cite{fm:compact} Fulton and MacPherson construct a
  compactification $X[\nu]$ for the configuration space
  $X^\nu\setminus\Delta$ of $\nu$ distinct labeled points in an
  algebraic variety $X$ (where $\Delta$ is the union of all
  diagonals). In the case $X=\C$, the construction using
  \emph{screens} is as follows. Write $y_1,\ldots,y_\nu$ for the
  coordinates on $\C^\nu\setminus\Delta$. For every subset
  $S\subset\{1,\ldots,\nu\}$ the \emph{screen} $\iota_S$ is defined as
  the map
  \begin{equation}
    \iota_S:\C^\nu\setminus\Delta\to\P(\C^{|S|}/\C), \qquad i_S(y_1,\ldots,y_\nu)=(y_\alpha)_{\alpha\in S}
  \end{equation}
  where $\C\subset\C^{|S|}$ is the diagonally embedded subspace. The
  map $\iota_S$ remembers the relative positions of
  $\{y_\alpha\}_{\alpha\in S}$ up to translation (corresponding to the
  diagonally embedded $\C$) and homothety (corresponding to the
  projectivization). Consider the map
  \begin{equation}\label{eq:iota-def}
    \iota:\C^\nu\setminus\Delta\to(\P^1)^{\binom\nu3}, \qquad \iota = \prod_{|S|=3} \iota_S.
  \end{equation}
  Then $\C[\nu]$ is the Zariski closure of the graph of $\iota$. It is
  an important fact that only the screens $\iota_S$ for $S$ of size
  three are required to fully determine the compactification: if one
  defines $\iota$ by a product over all $S\subset\{1,\ldots,\nu\}$ one
  obtains an isomorphic compactification.

  In our setting, we have a natural map
  $\vy=(y_1,\ldots,y_\nu):\hat\cC\to\C^\nu\setminus\Delta$. For an
  appropriate choice of a rational coordinate on each $\P^1$ factor
  in~\eqref{eq:iota-def}, the coordinates of $\iota$ can be identified
  with our maps $s_{i,j,k}$ which are indeed invariant under
  translation and homothety. The equations~\eqref{eq:intro-sijk-fund}
  imply in particular that the image $\vy(\hat\cC)$ is of small
  diameter, not only in the metric induced from $\C^\nu$, but also in
  the finer metric induced from the compactification $\C[\nu]$. It is
  crucial here that one only needs to use the screens $S$ of size
  three. If we were to write equations similar
  to~\eqref{eq:intro-sijk-fund} for screens of arbitrary size the
  factor $\nu^3$ would be replaced by $\nu!$, which is exponential in
  $\beta$ and hence too large for our purposes.

  As we proceed to describe the cellular cover for
  $\C\setminus\{y_1,\ldots,y_\nu\}$ we will make no explicit reference
  to the Fulton-MacPherson compactification, working instead directly
  with~\eqref{eq:intro-sijk-fund}. However the reader may find it
  helpful to interpret these inequalities as estimates on the metric
  induced from $\C[\nu]$.
\end{Rem}

\subsubsection{Clustering around a center $y_i$}

We fix a section $y_i$. We will cluster the remaining sections into
annuli centered at $y_i$. Since our constructions are invariant under
affine transformations we may assume after an appropriate
transformation that $y_i=0$, and then $s_{i,j,k}=y_j/y_k$.

Pick an arbitrary base point $p\in\hat\cC$. We will say that $y_j$ is
close to $y_k$ if they satisfy $\abs{\log|y_k(p)/y_j(p)|}<1/\nu$, and
define the \emph{clusters} around $y_i$ to be the equivalence classes
of the transitive closure of this relation. Then for any $y_j,y_k$, if
they are in the same cluster we have $\abs{\log|y_k(p)/y_j(p)|}<1$ and
otherwise we have $\abs{\log|y_k(p)/y_j(p)|}\ge1/\nu$. See
Figure~\ref{fig:sketch-yi-cluster} for an illustration.
  
\subsubsection{The Voronoi cells associated with $y_i$}

We will use the clusters around $y_i$ to construct a collection of
\emph{Voronoi cells} and their mappings into $\cC\odot\C$. Their
defining property, which motivates our choice of naming, is the
following: the Voronoi cells associated with $y_i$ cover every point
$z\in\C\setminus\{y_1,\ldots,y_\nu\}$ such that
$\dist(z,y_i)\le 2\dist(z,y_j)$ for any $j\neq i$. In other words, for
every section $y_i$ we will construct cells that cover every point in
$\C$ except those that are (twice) closer to some other section
$y_j$. Clearly then the union of these cells for every $y_i$ will
cover $\C\setminus\{y_1,\ldots,y_\nu\}$, which is our goal.

We begin by covering the empty area between two clusters. This is
fairly straightforward. Suppose that $y_j$ is a section with
$|y_j(p)|$ maximal within its cluster, and $y_k$ is a section with
$|y_k(p)|$ minimal within the next cluster (see
Figure~\ref{fig:sketch-yi-cluster}). Then the annulus
$A(y_je^\e,y_ke^{-\e})$ for say $\e=1/\nu^2$ does not meet any of the
sections $y_j$ over $p$, and assuming $\rho$ is sufficiently small
this remains true uniformly over $\hat\cC$
by~\eqref{eq:intro-sijk-fund}. We similarly cover the area between
$y_i=0$ and the first cluster by a punctured disc, and outside the
last cluster by $A(\cdot,\infty)$.

\begin{figure}
  \centering
  \includegraphics[width=\textwidth]{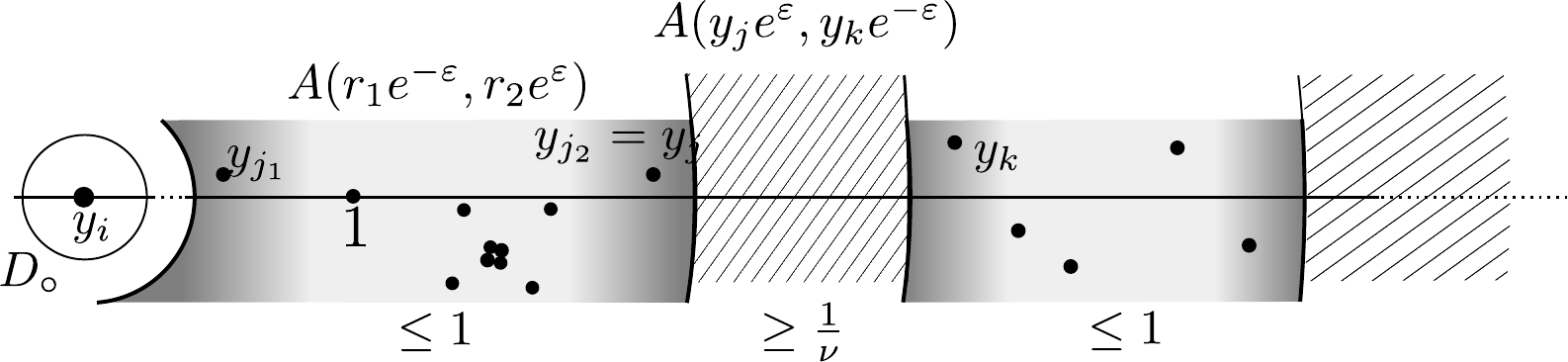}
  \caption{Clustering around $y_i$; distances are in $\log$-scale.}
  \label{fig:sketch-yi-cluster}
\end{figure}

\subsubsection{Voronoi cells in a cluster}

It remains to cover the part of $\C$ that remains outside the annuli
above, i.e. near one of the clusters. Choose one of the clusters and
let $y_j$ be a section in the cluster. Since our constructions are
invariant under affine transformations we may assume after an
appropriate transformation that $y_i=0$ and also $y_j=1$. With this
choice of coordinates $s_{i,k,j}=y_k$ and~\eqref{eq:intro-sijk-fund}
implies that $y_k$ does not change much over the cell $\hat\cC$. We will
construct our cells over the base
$\hat\cC_{i,j}:=\cC_{\times\nu_i\nu_j}$. Note that the complexity of
this cell is $\poly_\ell(\beta)$ as required.

Suppose $y_{j_1}$ (resp. $y_{j_2}$) is the section with
$r_1:=|y_{j_1}(p)|$ minimal (resp. $r_2:=|y_{j_2}(p)|$ maximal) within
its cluster. Then $1/e<r_1\le r_2<e$. Moreover the sections in the
cluster uniformly remain within $\cA:=A(r_1e^{-\e},r_2e^\e)$ while the
sections belonging to other clusters remain outside
$\cA':=\cA^{\Theta(1-1/\nu)}$.

Let $U$ be the set obtained from $\cA$ by removing discs of radius
$1/10$ centered at each of the points $y_k(p)$ for $y_k$ belonging to the
cluster. Note that any point in these discs is much closer to $y_k(p)$
than to $y_i(p)=0$, so we do not need to cover them with the Voronoi
cells of $y_i$. Moreover, the same remains true uniformly over
$\hat\cC$ since by~\eqref{eq:intro-sijk-fund} the points $y_k$ change
very little. To finish the construction we must therefore cover $U$
using discs whose $\hsigma$-extensions remain inside $\cA'$ and away
from the $1/20$-discs centered at the points $y_k(p)$ (this way even as $p$
varies over $\hat\cC$ the discs and their extensions will not meet
them). This is clearly possible to achieve with
$\poly_\ell(\nu)=\poly_\ell(\beta)$ discs as required. See
Figure~\ref{fig:sketch-clusters-voronoi} for an illustration of this
construction.

\begin{figure}
  \centering
  \includegraphics[width=0.6\textwidth]{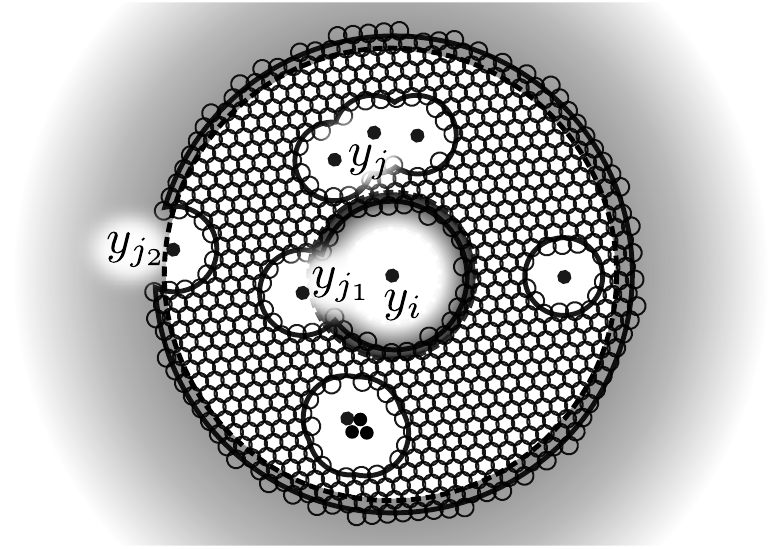}
  \caption{Voronoi cells in a cluster}
  \label{fig:sketch-clusters-voronoi}
\end{figure}

\section{Generalities on complex cells}

In this section we cover some generalities on complex cells, their
subanalytic structure and some uniformity results for families.

\subsection{Uniformity in families}
\label{sec:uniform-families}

In this section we show how the cellular structure of our maps implies
automatic uniformity over families, for instance in the statement of
the CPT. 

For $p\in\cC_{1..j}$ we denote
\begin{equation}
  \cC_p = \{\vz_{j+1..\ell} : (p,\vz_{j+1..\ell})\in\cC\}.
\end{equation}
If $\cC$ has type $\cF_1\odot\cdots\odot\cF_\ell$ then $\cC_p$ is a
cell of type $\cF_{j+1}\odot\cdots\odot\cF_\ell$. If $\cC$ admits a
$\delta$-extension then so does $\cC_p$.

By definition cellular map $f:\cC\to\hat\cC$ induces a cellular map
$f_{1..j}:\cC_{1..j}\to\hat\cC_{1..j}$ for $j=1,\ldots,\ell$ by
restriction to the first $j$ coordinates. The following proposition
follows directly from the definitions.

\begin{Prop}\label{prop:map-fiber-bound}
  Let $f:\cC\to\hat\cC$ be a cellular map between cells of length
  $\ell$ and $j=1,\ldots,\ell$. Then the number of points in
  $f_{1..j}^{-1}(p)$ for $p\in\hat\cC_{1..j}$ is bounded by a constant
  $\nu(f,j)$ independent of $p$. Explicitly, one may take
  \begin{equation}
    \nu(f,j) = \prod_{k=1}^j \deg_{\vz_k}\phi_k
  \end{equation}
  in the notations of Definition~\ref{def:cell-maps}.
\end{Prop}

The following remark illustrates how the cellular structure of the
maps in the CPT automatically implies uniformity over families.

\begin{Rem}\label{rem:uniform-cpt}
  Suppose $\ell=n+m$. We view $\cC$ in the statement of the CPT as a
  family of cells $\cC_p$ of length $m$ parametrized over
  $p\in\cC_{1..n}$. Let $f_j:\cC_j^\hsigma\to\cC^\hrho$ be the maps
  constructed in the CPT. By Proposition~\ref{prop:map-fiber-bound}
  the sets
  \begin{equation}
    P_j = \{p_{j,k}\} := (f_j)^{-1}(p)
  \end{equation}
  are finite with the number of points uniformly bounded over $p$. For
  each $p_{j,k}$ restriction of $f_j$ to the fiber gives a prepared
  cellular map $f_{j,k}:(\cC_j)^\hsigma_{p_{j,k}}\to\cC^\hrho_p$ which
  is compatible with the restrictions of $F_1,\ldots,F_m$ to $\cC_p$,
  and such that $f_{j,k}((\cC_j)_{p_{j,k}})$ cover $\cC_p$. In other words
  we obtain a cellular decomposition as in the CPT for the fibers
  $\cC_p$ with the number of cells uniformly bounded over $p$.
\end{Rem}

\subsection{Laurent expansion in a complex cell}
\label{sec:laurent}

Let $\cC=\cF_1\odot\dots\odot\cF_\ell$ be a complex cell. In this
section we show that any holomorphic function on $\cC$ can be expanded
into a series analogous to the classical Laurent series. We begin by
introducing the notion of \emph{normalized monomials}. Let
$\valpha\in\Z^\ell$. We define the $\cC$-normalized monomial
$\vz^{[\valpha]}:=\vz_1^{[\valpha_1]}\cdots\vz_\ell^{[\valpha_\ell]}$
where: if $\cF_j=D(r_j)$ or $\cF_j=D_\circ(r_j)$ then
\begin{equation}
  \vz_j^{[\valpha_j]} = \
  \begin{cases}
    (\vz_j/r_j)^{\valpha_j} & \valpha_j\ge 0 \\
    0 & \text{otherwise};
  \end{cases}
\end{equation}
if $\cF_j=A(r_{1,j},r_{2,j})$ then
\begin{equation}
  \vz_j^{[\valpha_j]} = \
  \begin{cases}
    (\vz_j/r_{1,j})^{\valpha_j} & \valpha_j\ge 0 \\
    (\vz_j/r_{2,j})^{\valpha_j} & \text{otherwise};
  \end{cases}
\end{equation}
and if $\cF_j=*$ then
\begin{equation}
  \vz_j^{[\valpha_j]} = \
  \begin{cases}
    1 & \valpha_j=0 \\
    0 & \text{otherwise}.
  \end{cases}
\end{equation}
The normalization is such that $|\vz^{[\valpha]}|$ is bounded by $1$ in
$\cC$ and achieves the bound only when $\valpha=0$.

Below we write $\pos\valpha$ for the vector whose $i$-th coordinate is
$|\valpha_i|$.
\begin{Prop}\label{prop:norm-Laurent}
  Let $f\in\cO_b(\cC)$. Then $f$ has a series expansion, absolutely
  convergent on compacts in $\cC$, as follows:
  \begin{equation}\label{eq:norm-Laurent}
    f(\vz) = \sum_{\valpha\in\Z^\ell} c_\valpha \vz^{[\valpha]}
  \end{equation}
  where $|c_\valpha|\le\norm{f}$. If $\cC$ admits a $\vdelta$-extension
  and $f\in\cO_b(\cC^\vdelta)$ then also
  $|c_\valpha|<\vdelta^{\pos\valpha} \norm{f}_{\cC^\delta}$.
\end{Prop}
\begin{proof}
  We proceed by induction on $\ell$. The case $\cF_\ell=*$ reduces
  immediately to the claim for $\cC_{1..\ell-1}$. Otherwise the
  standard formula for Laurent expansions in $\cF_\ell$ gives
  \begin{equation}
    f(\vz) = \sum_{\valpha_\ell=-\infty}^\infty c_{\valpha_\ell}(\vz_{1..\ell-1}) \vz_\ell^{[\valpha_\ell]}
  \end{equation}
  where
  \begin{equation}
    c_{\valpha_\ell}(\vz_{1..\ell-1}) = \frac{\vz_\ell^{\valpha_\ell}}{\vz_\ell^{[\valpha_\ell]}}
    (2\pi i)^{-1} \oint \frac{f(\vz_{1..\ell-1},\zeta)}{\zeta^{\valpha_\ell+1}}\d\zeta
  \end{equation}
  and the integral is over a simple positively oriented curve in
  $\cF_\ell$. The Cauchy estimate now implies
  $\norm{c_{\valpha_\ell}}_{\cC_{1.\ell-1}}\le\norm{f}$ and we proceed
  to expand $c_{\valpha_\ell}$ by induction on $\ell$ to obtain a
  series~\eqref{eq:norm-Laurent} with $|c_\valpha|<\norm{f}$. To see
  that this multivariate series is absolutely convergent on compacts
  note that for every $\vz\in\cC$ and $\valpha\in\Z^\ell$ we have
  \begin{equation}
     |\vz^{[\valpha]}|\le\rho^{|\valpha|}\quad \text{where}\quad \rho:= \max_{\vsigma\in\{-1,1\}^\ell} |\vz^{[\vsigma]}| < 1
  \end{equation}
  and~\eqref{eq:norm-Laurent} is thus majorated by
  $\sum_{\valpha}(\rho+\e)^{|\valpha|}$ in a neighborhood of
  $\vz$. For the final claim we write a Laurent expansion in
  $\cC^\vdelta$ and rewrite the $\cC^\vdelta$-normalized monomials as
  $\cC$-normalized monomials, gaining an extra factor of
  $\vdelta^{\pos\valpha}$.
\end{proof}

Let $\cP_\ell:=D(1)^{\times\ell}$ denote the standard unit
polydisc. We write simply $\cP$ when $\ell$ is clear from the context.

\begin{Cor}\label{cor:laurent-disc-decmp}
  Let $f\in\cO_b(\cC)$. Then there is a decomposition
  \begin{equation}
    f(\vz) = \sum_{\vsigma\in\{-1,1\}^\ell} f_\vsigma(\vz^{[\vsigma]})
  \end{equation}
  where $f_\vsigma\in\cO(\cP)$. If $\cC$ admits a $\delta$-extension
  and $f\in\cO_b(\cC^\delta)$ then $f_\sigma\in\cO(\cP^\delta)$, and
  moreover for any $\delta<\e<1$ we have
  \begin{equation}
    \norm{f_\vsigma}_{\cP^\e}\le(1-\delta/\e)^{-\ell}\norm{f}_{\cC^\delta}.
  \end{equation}
\end{Cor}
\begin{proof}
  For the first statement we simply collect all summands with index
  $\valpha$ of sign $\vsigma$ in~\eqref{eq:norm-Laurent} into
  $f_\vsigma$ (if $\valpha_j=0$ we arbitrarily treat it as positive
  for this purpose). For the second part, we have
  \begin{equation}
    \norm{f_{1,\ldots,1}(\vw)}_{\cP^\e}  \le \sum_{\valpha\in(\Z_{\ge0})^\ell} |c_\valpha| \norm{\vw^\valpha}_{\cP^\e}\le
    \sum_{\valpha\in(\Z_{\ge0})^\ell} (\delta/\e)^{|\valpha|}=(1-\delta/\e)^{-\ell},
  \end{equation}
  and similarly for the other choices of the signs.
\end{proof}

\subsection{Subanalyticity}

We say that a function is subanalytic on a domain $U\subset\C^n$ if it
is defined there and its graph over $U$ forms a subanalytic set. Our
goal in this section is to prove the following proposition.

\begin{Prop}\label{prop:cell-subanalytic}
  Let $0<\delta<1$. A complex cell $\cC$ admitting a
  $\delta$-extension is a subanalytic set. If $f\in\cO_b(\cC^\delta)$
  then $f$ is subanalytic on $\cC$.
\end{Prop}

\begin{proof}
  We prove both claims by induction on the length of $\cC$, the case
  of length zero being trivial. Let $\cC=\cC_{1..\ell}\odot\cF$.  Then
  the radii $r(z)$ or $r_1(z),r_2(z)$ of $\cF$ are subanalytic by
  induction and it easily follows that $\cC$ is subanalytic. Now let
  $f\in\cO_b(\cC^\delta)$. By Corollary~\ref{cor:laurent-disc-decmp}
  we may write $f$ as a sum of $2^{\ell+1}$ summands
  $f_\vsigma(\vz^{[\vsigma]})$. In the definition of $\vz^{[\vsigma]}$
  each of the radii involved are subanalytic on $\cC_{1..\ell}$ by
  induction, and each of the divisions involved are ``restricted'' in
  the sense of \cite{dvdd:subanalytic}, i.e. we always have
  $|\vz^{[\vsigma]}|\le 1$ for $\vz\in\cC$. It then follows, for
  instance using the subanalytic language of \cite{dvdd:subanalytic},
  that the graph of each $f_\vsigma$ and hence of $f$ is indeed
  subanalytic over $\cC$.
\end{proof}

\begin{Rem}
  In Proposition~\ref{prop:cell-subanalytic} one cannot replace
  subanalyticity by the stronger condition of
  \emph{semianalyticity}. For example let
  $\cC:=D_\circ(1)\odot D(\vz_1)$ and
  $f(\vz_1,\vz_2):=\vz_2 e^{\vz_2/\vz_1}$. Then $f\in\cO_b(\cC^{1/2})$
  but the graph of $f\rest\cC$ is not semianalytic (see Osgood's
  example \cite[Example~2.14]{bm:subanalytic}).

  The statement of Proposition~\ref{prop:cell-subanalytic} also fails
  if we do not assume $\delta$-extendability of $\cC$, as one can see
  in the example $\cC=D(1)$ and $f=\exp(\frac{z+1}{z-1})$.
\end{Rem}

\section{Semialgebraic and subanalytic cells}

In our terminology, a cell in an o-minimal structure (below
\emph{cylindrical cell}, see \cite{vdd:book}) is the $\odot$-product
of a sequence of open intervals and points with the endpoints/points
given by continuous definable functions. The image $f(\R_+\cC)$ of a
real prepared cellular map is itself a cylindrical cell. More
explicitly if $f(\vz_{1..\ell})_j=\vz_j^{q_j}+\phi_j(\vz_{1..j-1})$
and $\cC=\cF_1\odot\cdots\odot\cF_\ell$ then
\begin{equation}
  f(\R_+\cC) = \cI_1\odot\cdots\odot\cI_\ell, \quad I_j :=
  \begin{cases}
    \{\phi_j\} & \cF_j = * \\
    (\phi_j,\phi_j+|r^{q_j}|) & \cF_j = D(r),D_\circ(r) \\
    (\phi_j+|r_1^{q_j}|,\phi_j+|r_2^{q_j}|) & \cF_j = A(r_1,r_2).
  \end{cases}
\end{equation}
Furthermore, we note that $f$ restricts to a real-analytic
diffeomorphism from $\R_+\cC$ onto its image.

A cylindrical cell is called compatible with a continuous function $F$
if $\sign(F)$ is constant on the cell. Since $\R_+\cC$ is connected,
if $f$ is compatible with $F$ in the sense of
Definition~\ref{def:compatible} then cylindrical cell $f(\R_+\cC)$ is
compatible with $F$.

\subsection{Cellular parametrizations for semialgebraic and
  subanalytic sets}

\begin{Def}\label{def:semialg-complexity}
  A semialgebraic set $S\subset\R^n$ has \emph{complexity} $(\ell,\beta)$
  if $S=\pi_{1..n}(\tilde S)$ where $\tilde S\subset\R^\ell$ is given
  by a sign condition
  \begin{equation}
    \tilde S = \{ \sign(P_1)=\sigma_1,\ldots,\sign(P_N)=\sigma_N \}, \qquad \sigma_1,\ldots,\sigma_N\in\{-1,0,1\}
  \end{equation}
  where $P_1,\ldots,P_N\in\R[\vx_{1..\ell}]$ have degrees at most
  $\beta$ and $N\le\beta$. If $\ell=n$ we will say simply that $S$ has
  complexity $\beta$.
\end{Def}

The following semialgebraic parametrization result is a simple
consequence of the CPT.

\begin{Cor}\label{cor:cpt-semialg}
  Let $\rho,\sigma\in\R_+$ and let $S\subset(0,1)^n$ be semialgebraic
  of complexity $(\ell,\beta)$. Then there exist
  $\poly_\ell(\beta,\rho,1/\sigma)$ real cellular maps
  $f_j:\cC^\hsigma_j\to\cP_n^\hrho$, each of complexity
  $\poly_\ell(\beta)$, such that $f_j(\R_+\cC_j^\hsigma)\subset S$ and
  $\cup_j f_j(\R_+\cC_j)=S$.
\end{Cor}

\begin{proof}
  We first reduce to the case $\tilde S\subset(0,1)^\ell$ in the
  notations of Definition~\ref{def:semialg-complexity} by pulling back
  the variables $x_{n+1},\ldots,x_\ell$ by the bijection
  $t\to (2t-1)/(t-t^2)$ between $(0,1)$ and $\R$ (which increases the
  degrees at most polynomially).

  Now apply the real CPT to the cell $\cP_\ell^\hrho$ with the
  collection of polynomials defining $\tilde S$. We obtain a
  collection of real prepared maps
  $\tilde f_j:\tilde \cC^\hsigma_j\to\cP_\ell^\hrho$ with the required
  complexity estimates such that
  $\tilde f_j(\R_+\tilde\cC_j^\hsigma)\subset \tilde S$ and
  $\cup_j \tilde f_j(\R_+\tilde\cC_j)=\tilde S$. The cellular
  structure of $\tilde f_j$ implies that if we now take
  $\cC_j:=(\tilde C_j)_{1..n}$ and $f_j:=(\tilde f_j)_{1..n}$ we have
  indeed $f_j(\R_+\cC_j^\hsigma)\subset S$ and
  $\cup_j f_j(\R_+\cC_j)=S$.
\end{proof}

In an analogous manner one obtains the following subanalytic version.

\begin{Cor}\label{cor:cpt-subanalytic}
  Let $\rho,\sigma\in\R_+$ and let $S\subset(0,1)^n$ be
  subanalytic. Then there exist $\poly_S(\rho,1/\sigma)$ real
  cellular maps $f_j:\cC^\hsigma_j\to\cP_n^\hrho$ such that
  $f_j(\R_+\cC_j^\hsigma)\subset S$ and $\cup_j f_j(\R_+\cC_j)=S$.
\end{Cor}

\begin{Rem}\label{rem:cpt-semi-extra}
  We remark that from the proof it is clear that in
  Corollaries~\ref{cor:cpt-semialg} and~\ref{cor:cpt-subanalytic} one
  can also require the maps $f_j$ to be compatible with an additional
  collection of functions $F_j\in\cO_b(\cP_n^\hrho)$. We also remark
  that by rescaling one can replace the domain $(0,1)^n$ in
  Corollaries~\ref{cor:cpt-semialg} and~\ref{cor:cpt-subanalytic} by
  any other bounded semialgebraic/subanalytic ambient set. We will
  sometimes use $[0,1]^n$.
\end{Rem}

\subsection{Preparation theorems}

The real CPT implies the preparation theorem for subanalytic functions
of Parusinski \cite{parusinski:preparation} and Lion-Rolin
\cite{lion-rolin}, as we illustrate below. We illustrate the algebraic
case here (where we get more effective information) but the
subanalytic case follows in a similar manner. Let $F:(-1,1)^n\to\R$ be
a bounded semialgebraic function and let $G_F\subset\R^n_x\times\R_y$
be its graph. We aim to cover $(0,1)^n$ by cylinders where $F$ admits
a simple expansion.

We apply Corollary~\ref{cor:cpt-semialg} for the set $G_F$, requiring
also that the maps $f_j$ be compatible with $y$ (see
Remark~\ref{rem:cpt-semi-extra}). We obtain maps
$f_j:\cC_j^\hsigma\to \cP_n^{1/2}\times\cP^{1/2}$ such that
$f_j(\R_+\cC_j^\hsigma)\subset G_F$ and $f_j(\R_+\cC_j)$ cover $G_F$,
and moreover each map is compatible with $y$.

Let $f:\cC^\hsigma\to\cP_n^{1/2}\times\cP^{1/2}$ be one of the maps
$f_j$. Since $G_F$ is a graph the type of $\cC$ must end with $*$. It
follows from the monomialization lemma (Lemma~\ref{lem:monomial}) that
on each $\cC^\hsigma$ we have either $f^*y\equiv0$ or
\begin{equation}
  f^* y = \vz^{\valpha(j)} U_j(\vz)
\end{equation}
where $U$ is a holomorphic map bounded away from zero and infinity on
$\cC$. To rewrite this expansion in the $\vx$-coordinates recall that
\begin{equation}\label{eq:vz-v-vx}
  \vz_j = (\vx_j - \phi_j(\vz_{1..j-1}))^{1/\nu_j}
\end{equation}
where we restrict $\vz_j$ to the positive real part $\R_+\cC_j$ and
take the positive branch. Since $F(x)\equiv y$ on $G_F$ we have on
the cylinder $f(\R_+\cC_j)$ the expansion
\begin{equation}
  F(\vx) = \vz^\valpha U(\vz)
\end{equation}
where $\vz$ is given by~\eqref{eq:vz-v-vx}. In other words we have
obtained cylinders where $F$ expands as a monomial with fractional
powers times a unit. This implies the preparation theorem of
\cite{lion-rolin} (in the bounded semialgebraic case). Indeed, in
\cite{lion-rolin} the unit is required to be bounded away from zero
and infinity and satisfy an analytic expansion of the form
$U(\vz) = V(\psi(\vx_{1..n-1},\vx_n))$ where
\begin{multline}\label{eq:lion-rolin-expansion}
  \psi(\vx_{1..n-1},\vx_n) = (\psi_1(\vx_{1..n-1}),\ldots,\psi_s(\vx_{1..n-1}),\\
  \vx_n^{1/p}/a_1(\vx_{1..n-1}),b_1(\vx_{1..n-1})/\vx_n^{1/p}),
\end{multline}
with $p$ a positive integer and $\psi_i,a_1,b_1$ are bounded
subanalytic functions, and $V$ is a non-zero analytic function on the
compact closure of the image of $\psi$. In our case this expansion is
the Laurent expansion of $U_j(\vz)$ with respect to $\vz_n$.

\begin{Rem}
  Note that in comparison to the preparation theorem of
  \cite{lion-rolin} we obtain explicit bounds on the number of
  cylinders and their complexity in the CPT, and effective estimates
  on the monomial $\vz^\alpha$ and the unit $U(\vz)$ in the
  monomialization lemma (Lemma~\ref{lem:monomial}).

  The existence of holomorphic continuations to $\delta$-extensions of
  complex cells also greatly simplifies the fundamental definitions of
  \cite{lion-rolin}. For instance, while the notion of a function
  \emph{reducible} in a cylinder requires a careful inductive
  definition in \cite{lion-rolin}, any holomorphic function compatible
  with a complex cell is \emph{automatically} reducible
  there. Similarly, while the definition of a unit involves a delicate
  analytic expansion~\eqref{eq:lion-rolin-expansion} in
  \cite{lion-rolin}, a unit in a complex cell is a holomorphic
  function satisfying a purely topological definition (the associated
  monomial being equal to zero), which implies the real condition.
\end{Rem}

\section{The principal lemmas}
\label{sec:fund-lemmas}

Let $\cC$ be a complex cell of length $\ell$. For the results
discussed in this section any coordinate with type $*$ can be removed
without loss of generality, so we assume that the type does not
contain $*$. Then $n:=\dim\cC=\ell$.

\subsection{Hyperbolic geometry of complex cells}

For the proofs of the domination, fundamental and monomialization
lemmas of~\secref{sec:cell-topology-geometry} we need some basic
notions of hyperbolic geometry. Recall that the upper half-plane $\H$
admits a unique hyperbolic metric of constant curvature $-4$ given by
$|\d z|/2y$. A Riemann surface $U$ is called \emph{hyperbolic} if its
universal cover is the upper half-plane $\H$. In this case $U$
inherits from $\H$ a unique metric of constant curvature $-4$ which we
denote by $\dist(\cdot,\cdot;U)$ (we sometimes omit $U$ from this
notation if it is clear from the context). In particular a domain
$U\subset\C$ is hyperbolic if its complement contains at least two
points.

\begin{Lem}[\protect{Schwarz-Pick \cite[Theorem~2.11]{milnor:dynamics}}]\label{lem:schwarz-pick}
  If $f:S\to S'$ is a holomorphic map between hyperbolic surfaces
  $S,S'$ then
  \begin{equation}
    \dist(f(p),f(q);S') \le \dist(p,q;S) \qquad \forall p,q\in S.
  \end{equation}
\end{Lem}

\subsection{Maps from cells into hyperbolic Riemann surfaces}

We need the following notion of \emph{skeleton} of a cell.

\begin{Def}[Skeleton of a cell]\label{def:skeleton}
  If $\cC$ is a cell whose type does not include $D_\circ$ then we
  define its \emph{skeleton} $\cS(\cC)$ as follows: the skeleton of
  the cell of length zero is the singleton $\C^0$; The skeleton of
  $\cC_{1..\ell}\odot \cF$ is $\cS(\cC_{1..\ell})\odot \partial\cF$
  where
  \begin{align}
    \partial*&:=* & \partial D(r) &:= S(r) & \partial A(r_1,r_2)&:=
                                                                  S(r_1)\cup S(r_2).
  \end{align}
  Each connected component of $\cS(\cC)$ is a product of $\ell$
  circles and points, and the number of connected components is equal
  to $2^\alpha$, where $\alpha$ is the number of symbols $A$ in the
  type of $\cC$.
\end{Def}

In this section we assume that the type of $\cC$ does not contain
$D_\circ$. We assume that $\cC$ admits a $\hrho$ extension for some
$\hrho>0$ and that $f:\cC^\hrho\to X$ is a holomorphic map to a
hyperbolic Riemann surface $X$. We begin by studying the hyperbolic
behavior of $f$ on the skeleton $\cS(\cC)$.

\begin{Lem}\label{lem:fS-hyperbolic-diam}
  Let $S$ be a component of the skeleton $\cS(\cC)$. Then
  \begin{equation}
    \diam(f(S);X) \le n\rho.
  \end{equation}
\end{Lem}
\begin{proof}
  Note that the case $n=1$ follows
  immediately from Fact~\ref{fact:boundary-length} and the
  Schwarz-Pick lemma.

  We proceed by induction and suppose the claim is proved for cells of
  dimension smaller than $n$. Let $\vz,\vz''\in S$. We will construct
  a point $\vz'\in S$ such that $\dist(f(\vz),f(\vz'))\le\rho$ and
  such that $\vz'_1=\vz''_1$. Then $\vz'_{2..n},\vz''_{2..n}$ both
  belong to the $n-1$ dimensional cell $\cC_{\vz_1'}$ and applying the
  inductive hypothesis to the restriction of $f$ to this fiber we
  conclude that $\dist(f(\vz'),f(\vz''))\le(n-1)\rho$ thus completing
  the proof.

  Let us construct $\vz'$. By definition of the skeleton, we have
  \begin{align}
    \vz_1&=\e_1r_1 & \vz_2&=\e_2r_2(\vz_1) & &\dots & \vz_n&=\e_n 
    r_n(\vz_{1..n-1}), \quad \abs{\e_i}=1.
  \end{align}
  where $r_1>0$ and $r_2,\ldots,r_n$ are holomorphic functions on
  $\cC^\hrho$. Let $\cF$ denote the first fiber in $\cC$. We define a map
  $\gamma(\vz_1):\cF^\hrho\to\cC^\hrho$ by $\gamma(\vz_1)=(\vz_{1..n})$ where
  \begin{align}
    \vz_2 &= \e_2 r_2(\vz_1) & &\dots & \vz_n =\e_n r_n(\vz_{1..n-1}).
  \end{align}
  Then $\gamma$ is holomorphic in $\cF^\hrho$ and satisfies
  $\gamma(\vz_1)=\vz$ and $\gamma(\{|\vz_1|=r_1\})\subset S$. We take
  $\vz':=\gamma(\vz''_1)\in S$ and note that
  \begin{equation}
    \dist(f(\vz),f(\vz');X) =
    \dist(f\circ\gamma(\vz_1),f\circ\gamma(\vz''_1);X)\le \dist(\vz_1,\vz''_1;\cF^\hrho)\le\rho
  \end{equation}
  where the first inequality follows from the Schwarz-Pick lemma for
  $f\circ\gamma$ and the second inequality follows since
  $\vz_1,\vz_1''$ belong to the boundary of $\cF$ in $\cF^\hrho$.
\end{proof}

Next, we show (essentially by the open mapping theorem) that the
boundary of $f(\cC)$ is controlled by the skeleton.

\begin{Lem}\label{lem:boundary-vs-skeleton}
  We have the inclusion $\partial f(\cC) \subset f(\cS(\cC))$.
\end{Lem}
\begin{proof}
  We assume for simplicity that the type of $\cC$ is $A\cdots A$
  (other cases are treated similarly). Let $p\in\partial f(\cC)$ and
  fix a point $\vz\in\cC$ such that $f(\vz)=p$. If $\vz\in\cS(\cC)$ then we
  are done. Otherwise one of the following holds
  \begin{equation}\label{eq:k-index-def}
    \begin{aligned}
      r_{1,1} < &|\vz_1|< r_{1,2}  \\
      |r_{2,1}(\vz_1)| < &|\vz_2| <|r_{2,2}(\vz_1)| \\
      &\vdots \\
      |r_{n,1}(\vz_{1..n-1})| < &|\vz_n| < |r_{n,2}(\vz_{1..n-1})|.
    \end{aligned}
  \end{equation}
  Let $k$ be the largest index for which such inequalities hold. Let
  \begin{equation}
    A := A(r_{k,1}(\vz_{1..k-1}),r_{k,2}(\vz_{1..k-1})).
  \end{equation}
  We suppose further that $\vz$ is chosen such that $k$ is the minimal
  possible. We define a map $\gamma(t):A^\delta\to\cC$ by
  $\gamma(t) =(z_{1..n})$ where $z_{1..k-1}=\vz_{1..k-1},z_k=t$ and
  \begin{equation}
    z_{j+1} = \frac{\vz_{j+1}}{r_{j+1,i(j+1)}(\vz_{1..j})} r_{j+1,i(j+1)}(z_{1..j}),  \qquad
    j\ge k
  \end{equation}
  where $i(j)$ is the index, either 1 or 2, such that equality instead
  of inequality holds in the $j$th line
  of~\eqref{eq:k-index-def}. By definition we have $\gamma(\vz_k)=\vz$.

  If $f\circ\gamma$ is non-constant then by the open mapping theorem
  $p=f\circ\gamma(\vz_k)$ is an interior point of
  $f\circ\gamma(A)\subset f(\cC)$ contrary to our
  assumption. Otherwise we have $p=f(\vz')$ where $\vz'=\gamma(t)$ and
  $t$ is any point on $\partial A$. The index $k$ obtained for $\vz'$
  is by definition smaller than that obtained for $\vz$, which
  contradicts our choice of $\vz$.
\end{proof}

Finally, we show that under a suitable topological condition the
hyperbolic diameter of $f(\cC)$ itself can be bounded.

\begin{Lem}\label{lem:fC-hyperbolic-diam}
  Suppose that $f_*(\pi_1(\cC))=\{e\}\subset\pi_1(X)$. Then
  \begin{equation}
    \diam(f(\cC);X)<2^nn\rho.
  \end{equation}
\end{Lem}
\begin{proof}
  Denote by $\pi:\D\to X$ the universal covering map. By our
  assumption we may lift $f$ to a map $F:\cC^\hrho\to\D$ satisfying
  $f=\pi\circ F$. Since $\pi$ is non-expanding it is enough to prove
  the claim for the hyperbolic diameter of $F(\cC)$. By
  Lemma~\ref{lem:fS-hyperbolic-diam} the hyperbolic diameter of
  $F(S)$, where $S$ is any component of the skeleton $\cS(\cC)$, is
  bounded by $n\rho$. By Lemma~\ref{lem:boundary-vs-skeleton}
  we also have the inclusion $\partial F(\cC) \subset F(\cS(\cC))$.

  Recall that we assume that the type of $\cC$ does not contain
  $D_\circ$ and hence $\bar\cC\subset\cC^\delta$. In particular it
  follows that $Z:=F(\cC)$ is relatively compact, and hence bounded,
  in $\D$. Let $U$ denote the unbounded component of
  $\D\setminus\bar Z$ and set $U^c:=\D\setminus U$ and $\Gamma:=\partial U$. To avoid pathologies
  of plane topology we remark that $Z$ is subanalytic, hence $Z,U$ and
  their boundaries can be triangulated and the Mayer-Vietoris sequence
  for reduced homology
  \begin{equation}
    0=H_1(\D)\to \tilde H_0(\Gamma)\to \tilde H_0(\bar U)\oplus \tilde H_0(U^c)\to \tilde H_0(\D)\to0
  \end{equation}
  is exact. Since $\tilde H_0(\D)=\tilde H_0(\bar U)=0$ we have
  $\tilde H_0(\Gamma)\simeq \tilde H_0(U^c)$. We claim that $U^c$ is
  connected. Assume otherwise and write $U^c=V_1\cup V_2$. Since
  $\bar Z$ is connected, without loss of generality $\bar Z$ is
  contained in $V_1$ and therefore disjoint from $V_2$. In particular
  $\partial V_2$ is disjoint from $\bar Z$, contradicting
  \begin{equation}
    \partial V_2\subset\partial U^c=\partial U\subset\partial Z.
  \end{equation}
  In conclusion we have $\tilde H_0(\Gamma)\simeq \tilde H_0(U^c)=0$,
  i.e. $\Gamma$ is connected.

  We claim that the hyperbolic diameter of $F(\cC)$ is bounded by the
  hyperbolic diameter of $\Gamma$. Indeed, for any two points
  $p,q\in F(\cC)$ let $\ell$ denote the geodesic line connecting them
  and denote by $p'\in\ell\cap\Gamma$ some point on $\ell$ before $p$
  and by $q'\in\ell\cap\Gamma$ some point on $\ell$ after $q$. Then
  $\dist(p,q)\le\dist(p',q')$ which is bounded by the diameter of
  $\Gamma$. Finally, recall that $\Gamma$ is connected and contained
  in the union of $F(S_j)$ where $S_j$ run over the components of
  $\cS(\cC)$ (whose number is at most $2^n$), and the diameter of each
  $F(S_j)$ is bounded by $n\rho$. From this it follows easily
  that the diameter of $\Gamma$ is at most $2^nn\rho$.
  \begin{figure}
    \centering
    \includegraphics[width=0.6\textwidth]{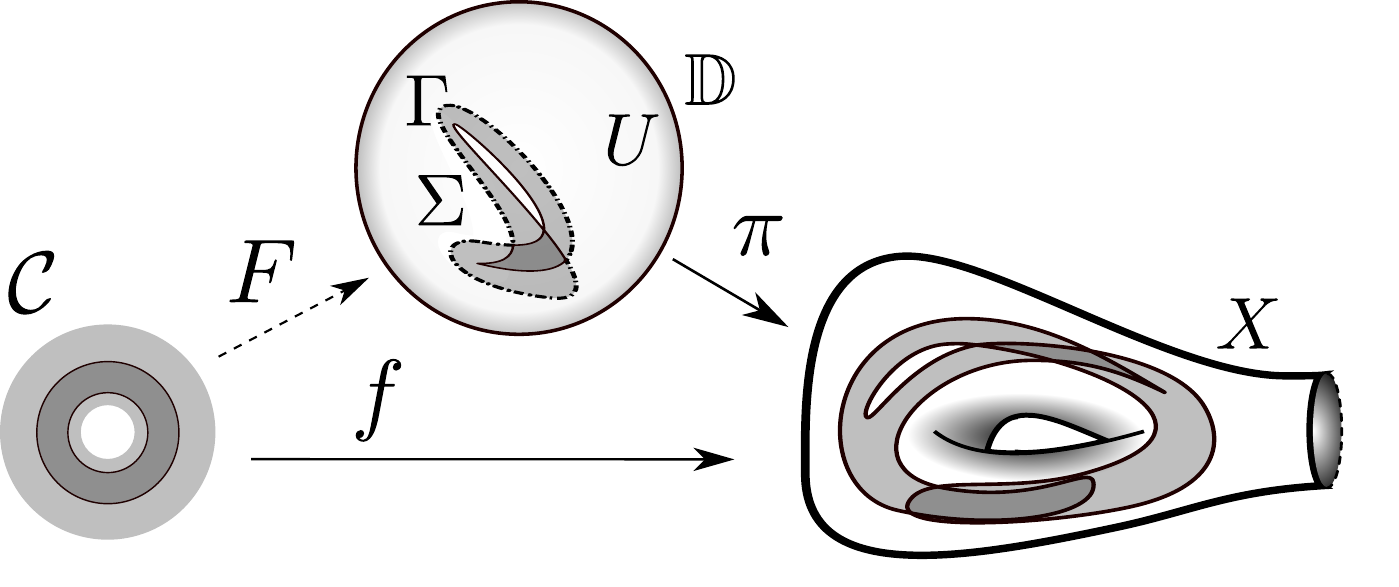}
    \caption{Proof of Lemma~\ref{lem:fC-hyperbolic-diam}.}
    \label{fig:DominationAnn}
  \end{figure}
\end{proof}

\subsection{Proof of the domination lemma
  (Lemma~\ref{lem:domination})}
\label{sec:domination-proof}

We begin by assuming that the type of $\cC$ does not contain
$D_\circ$. Let
\begin{align}\label{eq:U01infty}
  U_0 &:= \{|z|<\tfrac12\} & U_1&:= \{|z-1|<\tfrac12\} & U_\infty = \{|z|>2\}.
\end{align}
We choose $s>1$ such that the hyperbolic distance in
$\C\setminus\{0,1\}$ between $U_q^s$ and $\partial U_q$ is greater
than $n\rho$ for $q=0,1,\infty$, where
\begin{align}
  U_0^s &:= \{|z|<\tfrac1{2s}\} & U_1^s&:= \{|z-1|<\tfrac1{2s}\} & U_\infty^s = \{|z|>2s\}.
\end{align}
From the explicit computations in~\secref{sec:explicit-consts} we can
take $s=O(\log|\log(n\rho)|)$. If $f(\cC)$ does not meet $U_0^s$ or
$U_\infty^s$ then the proof of the domination lemma is
completed. Henceforth we assume that $f(\cC)$ meets both $U_0^s$ and
$U_\infty^s$.

\begin{Lem}\label{lem:trivial-homotopy}
  Suppose $f(\cC)$ meets both $U_0^s$ and $U_\infty^s$. Then
  \begin{equation}
    f_*\pi_1(\cC)=\{e\}\subset\pi_1(\C\setminus\{0,1\}).
  \end{equation}
\end{Lem}
\begin{proof}
  Recall that we assume that the type of $\cC$ does not contain
  $D_\circ$, and in particular $\bar\cC\subset\cC^\delta$. Thus $f$
  extends to a continuous function on $\bar\cC$, so
  $f(\bar\cC)\subset\C\setminus\{0,1\}$ is compact and it follows that
  $\partial f(\cC)$ meets both $U_0^s$ and $U_\infty^s$. By
  Lemma~\ref{lem:boundary-vs-skeleton} we conclude that there exist
  two components $S_0,S_\infty$ of the skeleton $\cS(\cC)$ such that
  $f(S_0)$ meets $U_0^s$ and $f(S_\infty)$ meets $U_\infty^s$. From
  Lemma~\ref{lem:fS-hyperbolic-diam} and the choice of $s$ we conclude
  that $f(S_0)$ does not meet the boundary of $U_0$,
  i.e. $f(S_0)\subset U_0$ and similarly
  $f(S_\infty)\subset U_\infty$.

  It is easy to verify (it is enough to check this for the standard
  polyannulus) that $\pi_1(S)\to\pi_1(\cC)$ is epimorphic. In
  particular, if $\gamma\in\pi_1(\cC)$ denotes any loop then this loop
  is free-homotopy equivalent to a loop contained in $S_0$ and to a
  loop contained in $S_\infty$. Consequently $f_*(\gamma)$ is
  free-homotopy equivalent to a loop contained in $U_0$ and to a loop
  contained in $U_\infty$. However two such loops cannot be
  homotopically equivalent in $\C\setminus\{0,1\}$ unless they are
  both contractible, hence proving the claim.
\end{proof}

Lemma~\ref{lem:trivial-homotopy} implies the condition of
Lemma~\ref{lem:fC-hyperbolic-diam}, and we conclude that the
hyperbolic diameter of $f(\cC)$ is bounded by $2^nn\rho$. We now
choose $\hat s>1$ such that the hyperbolic distance between
$U_0^{\hat s}$ and $U^{\hat s}_\infty$ is greater than $2^nn\rho$ (see
Figure~\ref{fig:domination-disc} for an
illustration). By~\secref{sec:explicit-consts} we may take
$\hat s=O(\log(n+\log\rho))$. Then $f(\cC)$ cannot meet both
$U_0^{\hat s},U^{\hat s}_\infty$ and the domination lemma is proved.

\begin{figure}
  \centering
  \includegraphics[width=\textwidth]{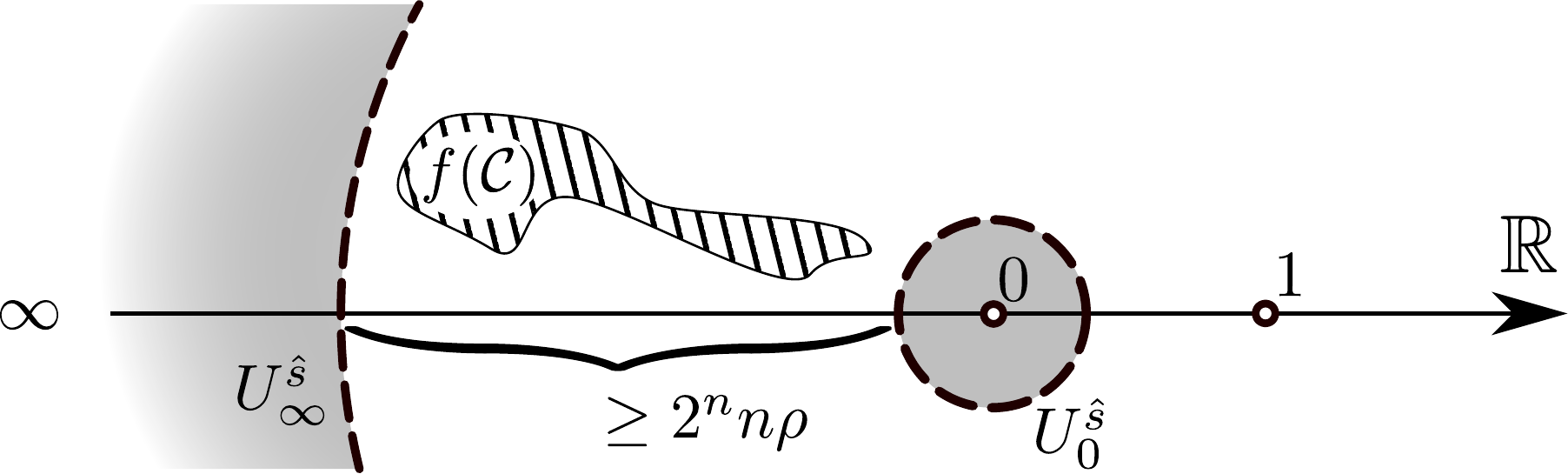}
  \caption{Proof of the domination lemma.}
  \label{fig:domination-disc}
\end{figure}

It remains to consider the case that the type of $\cC$ contains
$D_\circ$. Let $0<\e<1$ and let $\cC_\e$ be the cell obtained from
$\cC$ by replacing each occurrence of a fiber $D_\circ(r)$ by
$A(\e r,r)$. It is clear that $\cC_\e$ admits a $\hrho$-extension and
$\cup_{\e>0}\cC_\e=\cC$. The domination lemma for $\cC$ thus follows
immediately from the domination lemma for $\cC_\e$, which was already
established.

\subsection{Proofs of the fundamental lemmas
  (Lemmas~\ref{lem:fund-D}--~\ref{lem:fund-C01})}

In this section we may assume that the type of $\cC$ does not contain
$D_\circ$. The general case can be reduced to this case as in the end
of~\secref{sec:domination-proof}. The fundamental lemma for $\D$
(Lemma~\ref{lem:fund-D}) is already proved as a consequence of
Lemma~\ref{lem:fC-hyperbolic-diam}. The proof of the fundamental lemma
for $\D\setminus\{0\}$ (Lemma~\ref{lem:fund-Dcirc}) is based on the
following simple geometric lemma.

\begin{Lem}\label{lem:shortest-loop}
  Let $z\in\D\setminus\{0\}$ and let $\gamma_z$ denote the shortest
  non-contractible loop passing through $z$. Then (i)
  $\length(\gamma_z)\to\infty$ as $|z|\to1$; and (ii) for every
  $0<r<1$ and $|z|<r$, we have $\length(\gamma_z)=\Theta_r(1/|\log|z||)$.
\end{Lem}
\begin{proof}
  Let $\zeta=i^{-1}\log z$. Lifting to the universal cover
  $\exp(i\zeta):\H\to\D\setminus\{0\}$ we must calculate
  $\dist(\zeta,\zeta+2\pi;\H)$. By a standard formula for the
  hyperbolic distance \cite[Theorem~1.2.6]{katok:fuchsian-groups} we
  have
  \begin{equation}
    \dist(\zeta,\zeta+2\pi;\H) = 2\ln\left(\frac{\pi+\sqrt{\pi^2+\Im^2\zeta}}{\Im\zeta}\right)
  \end{equation}
  and the right hand side tends to infinity as $|z|\to1$ and is
  $\Theta_r(1/|\log|z||)$ when $|z|<r$.
\end{proof}

Let $f:\cC^\hrho\to\D\setminus\{0\}$ be holomorphic and let
$z\in\D\setminus\{0\}$ be the point of $f(\bar\cC)$ with maximum
absolute value, which by Lemma~\ref{lem:boundary-vs-skeleton} belongs
to the image $f(S)$ of a component $S$ of the skeleton of $\cC$. By
Lemma~\ref{lem:fS-hyperbolic-diam} we have
$\diam(f(S);\D\setminus\{0\})=O_\ell(\rho)$.

By Lemma~\ref{lem:shortest-loop} part (i) using $\rho<1$ there exists
a constant $0<r(\ell)<1$ such that if $|z|\ge r(\ell)$ then $f(S)$ cannot
contain a non-contractible loop in $\D\setminus\{0\}$. Then, since
$\pi_1(S)\to\pi_1(\cC)$ is epimorphic we have
\begin{equation}
  f_*(\pi_1(\cC))=f_*(\pi_1(S))=\{e\}\subset\pi_1(\D\setminus\{0\}).
\end{equation}
In this case by Lemma~\ref{lem:fC-hyperbolic-diam} we have
$\diam(f(\cC);\D\setminus\{0\})=O_\ell(\rho)$. Supposing now that
$|z|<r(\ell)$ and using Lemma~\ref{lem:shortest-loop} part (ii), we
similarly see that if $1/|\log|z||=\Omega_\ell(\rho)$ then $f(S)$
cannot contain a non-contractible loop in $\D\setminus\{0\}$ and
finish the proof in the same way. In the remaining case we have
$\log|z|=-\Omega_\ell(1/\rho)$, which concludes the proof of the first
statement. The second statement follows using
\begin{equation}
  \dist(\zeta_1,\zeta_2;\H) \ge \dist(\log\Im\zeta_1,\log\Im\zeta_2;\R).
\end{equation}

We now pass to the proof of the fundamental lemma for
$\C\setminus\{0,1\}$ (Lemma~\ref{lem:fund-C01}). It will suffice to
prove the statement for $\rho$ smaller than some constant $O_\ell(1)$
to be chosen later. Indeed, for $\rho$ larger than such a constant we
can make the first condition in~\eqref{eq:fund-C01} trivial with an
appropriate choice of the asymptotic constant, since every point in
$\C P^1$ lies at distance strictly smaller than $1$ from
$\{0,1,\infty\}$.

Let $S$ be a component of the skeleton of $\cC$. By
Lemma~\ref{lem:fS-hyperbolic-diam} we have
$\diam(f(S);C\setminus\{0,1\})=O_\ell(\rho)$. In particular, for
$\rho=O_\ell(1)$ we see that $f(S)$ cannot meet more than one of the
sets $U_0,U_1,U_\infty$ in the notation~\eqref{eq:U01infty}. Suppose
first that the $f(S)$ meets none of these sets. Let $\rho_0$ denote
the length of the shortest non-contractible loop in the compact set
$(U_0\cup U_1\cup U_\infty)^c\subset\C\setminus\{0,1\}$. Then for
$\rho=O_\ell(\rho_0)=O_\ell(1)$ we get
$\diam(f(\cC);\C\setminus\{0,1\})=O_\ell(\rho)$ as in the proof of the
case $\D\setminus\{0\}$.

Next, suppose that for two different components $S,S'$ the images
$f(S),f(S')$ meet two of the sets $U_0,U_1,U_\infty$, say $U_0$ and
$U_1$ respectively. Then for $\rho=O_\ell(1)$ the images $f_*(\pi_1(S))$
(resp. $f_*(\pi_1(S'))$ can only be a power of the fundamental loop
around $0$ (resp. $1$) and as in the proof of the domination lemma we
conclude that $f_*(\pi_1(\cC))=\{e\}$ and
$\diam(f(\cC);\C\setminus\{0,1\})=O_\ell(\rho)$.

Finally, suppose all skeleton components $S$ meet one of the sets
$U_0,U_1,U_\infty$, say $U_0$. Then for $\rho=O_\ell(1)$ we see that
$f(S)\subset\D\setminus\{0\}$ for each skeleton component. In this
case we can finish the proof as in the case of $\D\setminus\{0\}$. We
need only the estimate for the length of the shortest geodesic in
$\C\setminus\{0,1\}$ passing through a given point $z$ close to $0$:
this is asymptotically the same as in the metric of
$\D\setminus\{0\}$, for instance by the estimate~\eqref{eq:hyperb} of
\cite{bp:poincare}, given explicitly in~\secref{sec:explicit-consts}.

\subsection{Proof of the monomialization lemma (Lemma~\ref{lem:monomial})}
\label{sec:proof-monom}

The \emph{Voorhoeve index} \cite{ky:rolle} of a holomorphic function
$f:U\to\C$ along a subanalytic curve $\Gamma\subset U$ is defined by
\begin{equation}
  V_{\Gamma}(f):= \frac1{2\pi} \int_{\Gamma}|\d\Arg f(z)|.
\end{equation}
We need the following basic fact about Voorhoeve indices.
\begin{Lem}\label{lem:voorhoeve-estimate}
  Fix $\rho>0$. Let $\cF_\lambda$ be a definable family of
  one-dimensional cells, let $f_\lambda:\cF_\lambda^\rho\to\C$ be a definable
  family of holomorphic functions, and let
  $\Gamma_\lambda\subset\cF_\lambda$ be a definable family of curves.
  Then $V_{\Gamma_\lambda}(f_\lambda)$ is uniformly bounded over
  $\lambda$.

  If $f,\Gamma$ are algebraic of complexity $\beta$ then
  $V_\Gamma(f)=\poly(\beta)$.
\end{Lem}
\begin{proof}
  Note that $2\pi V_\Gamma(f)$ is the total length of the curve
  $(\frac f{|f|})(\Gamma)\subset S(1)$, which by standard integral
  geometry is given by the average number of intersections between the
  curve and a ray through the origin, i.e. the average number of
  (isolated) solutions of the equation $\arg f(z)=\alpha$ for
  $z\in\Gamma$ $\alpha\in[0,2\pi)$. By o-minimality the \emph{maximal}
  number (and in particular the average number) of (isolated)
  solutions for the pair $f_\lambda,\Gamma_\lambda$ is bounded by a
  constant independent of $\lambda$. In the algebraic case the number
  of solutions is bounded by $\poly(\beta)$ by the Bezout theorem.
\end{proof}

The basic ingredient in the proof of the monomialization lemma is the
following one-dimensional version, proved using the Voorhoeve index.

\begin{Lem}\label{lem:monom-dim1}
  Let $\cF$ be a one-dimensional cell and
  $f:\cF^\hrho\to\C\setminus\{0\}$. Then $f=z^{\alpha(f)}\cdot U(z)$ where $\log U:\cF\to\C$ is univalued and bounded.

  If $\cF^\hrho,f$ vary in a definable family $\Lambda$ then $|\alpha(f)|=O_\Lambda(1)$
  and 
    \begin{align*}
  \diam(\log U(\cF);\C) &=O_\Lambda(\rho), & \diam(\Im\log U(\cF);\R) &= O_\Lambda(1).
  \end{align*}
 If $f$ is algebraic of complexity $\beta$ then
  $|\valpha(f)|=\poly(\beta)$ and
  \begin{align*}
    \diam(\log U(\cF);\C) &< \poly(\beta)\cdot\rho, & \diam(\Im\log U(\cF);\R) &< \poly(\beta).
  \end{align*}
\end{Lem}
\begin{proof}
  We follow the idea of \cite[Section 4.5]{ky:rolle}. By definition
  $\alpha(f)$ is equal to the total winding number of $f$ along a
  concentric circle contained in $\cF$ for types $D_\circ,A$ and zero
  in type $D$. In particular it is bounded by $V_\Gamma(f)$ where
  $\Gamma$ is such a circle, and the statements about $\alpha(f)$ then
  follow from Lemma~\ref{lem:voorhoeve-estimate}. In the algebraic
  case we see also that the complexity of $U$ is $\poly(\beta)$.

  Any two points in $\cF^{\he{2\rho}}$ can be joined by two consecutive algebraic
  curves: a radial ray and a circle concentric with $\cF$. By
  Lemma~\ref{lem:voorhoeve-estimate} the Voorhoeve index of $f$ along
  these curves is uniformly bounded (resp.  bounded by $\poly(\beta)$
  in the algebraic case). Having already established that
  $|\alpha(f)|$ is uniformly bounded (resp. bounded by $\poly(\beta)$
  in the algebraic case) we conclude that the Voorhoeve index of $U$
  along any two such curves is similarly bounded by 
  $v=O_\Lambda(1)$ (and $v=\poly(\beta)$ in the algebraic case).

  Let $p\in\cF$ be an arbitrary point. Since the claim is invariant
  under scalar multiplication of $f$ we may assume without loss of
  generality that $U(p)=1$. Since $U_*\pi_1(\cF)=\{e\}$ we have a
  well-defined root $W:\cF^{\he{\rho}}\to\C\setminus\{0\}$ satisfying
  $W=iU^{1/(4v)}$ and $W(p)=i$. Connecting $p$ to any other point
  $q\in\cF^{\he{2\rho}}$ by two curves as above, the Voorhoeve index of $W$ along
  the curves is bounded by $1/4$, hence the total variation of
  argument is bounded by $\pi/2$. Since $\arg W(p)=\pi/2$ we conclude that
  in fact $W:\cF^{\he{2\rho}}\to\H$. By Lemma~\ref{lem:fC-hyperbolic-diam} we
  then have $\diam(W(\cF);\H)=O_\Lambda(\rho)$. In particular this implies
  \begin{align}
    \diam(\log W(\cF);\C) &= O_\Lambda(\rho) & \diam(\Im\log W(\cF);\R) &< \pi
  \end{align}
  and the claim for $U$ follows immediately (using the fact that
  $v=\poly(\beta)$ in the algebraic case).
\end{proof}

We are now ready to finish the proof of the monomialization lemma by
induction on $\ell$. Let $\cC=\cC_{1..\ell}\odot\cF$. The case $\cF=*$
reduces trivially to the claim for $\cC_{1..\ell}$ and the case
$\ell=0$ is proved in Lemma~\ref{lem:monom-dim1}. Assume first that
the outer radius of $\cF$ is $1$.  Let
$\hat f:\cC_{1..\ell}^\hrho\to\C\setminus\{0\}$ be defined by
$\hat f:=f(\vz_{1..\ell},1)$. If $\cC,f$ are algebraic of complexity
$\beta$ then $\hat f$ is algebraic of complexity $\poly_\ell(\beta)$.

By definition we have $\valpha(\hat f)=\valpha_{1..\ell}(f)$, so
$\hat U:=U(\vz_{1..\ell},1)$ is equal to
$\hat f/\vz^{\alpha(\hat f)}$. By the inductive hypothesis we have
\begin{align}
  \diam(\log \hat U(\cC_{1..\ell});\C) &< O_\Lambda(\rho), & \diam(\Im\log \hat U(\cC_{1..\ell});\R) &< O_\Lambda(1).
\end{align}
 In the
algebraic case $|\alpha(\hat f)|=\poly_\ell(\beta)$ and
\begin{align}
  \diam(\log \hat U(\cC_{1..\ell});\C) &< \poly_\ell(\beta)\cdot\rho, & \diam(\Im\log \hat U(\cC_{1..\ell});\R) &< \poly_\ell(\beta).
\end{align}
Also, for each fixed $p\in\cC_{1..\ell}$ we have by
Lemma~\ref{lem:monom-dim1} for the fiber $\cC_p$
\begin{align}
  \diam(\log U(\cC_p);\C) &< O_\Lambda(\rho), & \diam(\Im\log U(\cC_p);\R) &< O_\Lambda(1).
\end{align} In the
algebraic case $|\alpha_{\ell+1}(f)|=\poly_\ell(\beta)$ and
\begin{align}
  \diam(\log U(\cC_p);\C) &< \poly_\ell(\beta)\cdot\rho, & \diam(\Im\log U(\cC_p);\R) &< \poly_\ell(\beta).
\end{align}
The triangle inequality now finishes the proof.

For $\cF$ with an arbitrary outer radius $r$, let $\cF'$ denote the
fiber with outer radius normalized to $1$ and
$\cC'=\cC_{1..\ell}\odot\cF'$. There is a natural biholomorphism
$\cC'\to\cC$ given by
$\vw_{1..\ell+1}\to(\vw_{1..\ell},r\vw_{\ell+1})$. By what was already
proved we obtain a decompostion
\begin{equation}\label{eq:monom-proof-general-r}
  f= \vw^{\valpha'}\cdot U' = \vz_1^{\valpha'_1}\cdots\vz_\ell^{\valpha'_\ell}\cdot(\vz_{\ell+1}/r)^{\valpha'_{\ell+1}} U'
\end{equation}
with bounds on $\log U'$ as above. Finally, we monomialize $r$ by induction
over $\ell$, and plugging into~\eqref{eq:monom-proof-general-r} we
obtain a monomialization of $f$ in the $\vz$-coordinates.

\subsection{Appendix: computations of hyperbolic lengths}
\label{sec:explicit-consts}

We fix the hyperbolic metric, $\lambda_D(z) |dz|$ with
$\lambda_D(z)=(1-|z|^2)^{-1}$, of constant curvature $-4$ on the unit
disc $D$. An explicit lower bound for the hyperbolic metric
$\lambda_{0,1}(z)|dz|$ on $\C\setminus\{0,1\}$ in $U_0$ is given in
\cite{bp:poincare},
\begin{equation}\label{eq:hyperb}
  \lambda_{0,1}(z)\ge 
  \frac{1}{2|z|\sqrt{2}\left[4+\log(3+2\sqrt{2})-\log|z|\right]},\quad |z|\le  
  \frac{1}{2}.
\end{equation}
Integrating \eqref{eq:hyperb}, we see that the hyperbolic distance from 
$\{|z|=r\}$ to $\{|z|=1/2\}$ is greater than $s$ for $r<\tilde{\rho}(s)<\frac 1 
2$, where
\begin{equation*}
\log\tilde{\rho}(s)=-e^{\Theta(s)}+O(1).
\end{equation*}
As $z\to z^{-1}$ is an isometry of $\C\setminus\{0,1\}$, the 
hyperbolic distance from $\{|z|=r\}$ to $\{|z|=2\}$ is greater than $s$ for 
$r>\tilde{\rho}(s)^{-1}$.

\subsubsection{Proof of Fact~\ref{fact:boundary-length}.}

Let $\cF$ be a domain of type $A,D,D_\circ$ and let $S$ be a component
of the boundary of $\cF$ in $\cF^\hrho$. In this section we prove that
the length of $S$ in $\cF^\hrho$ is at most $\rho$.

First, consider the case of $\mathcal{F}$ of type $D$. Up to rescaling, 
$\mathcal{F}^\delta = D(1)$ and $\partial \mathcal{F}=\{\abs{z}=\delta\}$. Then 
$\lambda_D(z)\abs{dz}\equiv(1-\delta^2)^{-1}\abs{dz}$ on $\partial 
\mathcal{F}$, so 
the length of $\partial \mathcal{F}$ in $D$ is equal to 
$\frac{2\pi\delta}{1-\delta^2}$.

Second, consider the cases $\mathcal{F}$ of type $D_\circ, A$. If
$S(r)$ is (a component of) $\partial\mathcal{F}$, then
$S(r)\subset S^\delta(r)\subset\cF^\delta$. Therefore, by
Schwartz-Pick lemma (Lemma~\ref{lem:schwarz-pick}), it is enough to
prove the following.

\begin{Lem}\label{lem:hyperbcircle}
  The length  of $S(1)$ in the hyperbolic metric of the annulus
  $S^\delta(1)$ is equal to $\frac{\pi^2}{2\abs{\log\delta}}$.
\end{Lem}
\begin{proof}
  Indeed, the mapping $\phi(z)=\frac{\pi i}{2\abs{\log\delta}}\log z$ maps
  the universal cover of $S^\delta(1)$ to the strip
  $\Pi=\{|\Im w|\le \frac{\pi}{2}\}$, with
  $\phi^{-1}([0,\frac{\pi^2}{\abs{\log\delta}}])=S^1$. This map preserves
  hyperbolic distances, so it is enough to find the hyperbolic length
  of $[0, \frac{\pi^2}{\abs{\log\delta}}]$ in the hyperbolic metric 
  $\lambda_{\Pi}(w)\abs{dw}$ of $\Pi$. As $\Pi$ is invariant
  under shifts by reals, it is enough to find the $\lambda_{\Pi}(0)$.  The
  map $\psi(w)=\frac{e^w-1}{e^w+1}$ sends $\Pi$ isometrically to the
  unit disc, with $\psi(0)=0$, so $\lambda_{\Pi}(0)=\psi'(0)=1/2$ and the
  length of $[0, \frac{\pi^2}{\abs{\log\delta}}]$ in $\Pi$ is equal to
  $\frac{\pi^2}{2\abs{\log\delta}}$.
\end{proof}

\begin{Rem}\label{rem:Adelta-diameter}
  The same chain of conformal mappings allows to compute explicitly
  the distance between $r>0$ and $1$ in the hyperbolic metric of
  $A^\delta=A(\delta r, \delta^{-1})$:
  \begin{equation}
    \dist_{A^\delta}(r,1)=\log\frac{1+\tan\phi}{1-\tan\phi}, \quad 
    \text{where}\quad
    \phi=\frac\pi4\left(1-\frac{\log \delta}{\log(\delta \sqrt{r})}\right).
  \end{equation}
  As $r\to 0$, we get
  \begin{equation}
    \dist_{A^\delta}(r,1)=\log\left(\frac{16}{\pi}\frac{\log(\delta 
        \sqrt{r})}{\log 
        \delta}\right)+O(1)=\log\abs{\log r} +O(1).
  \end{equation}
  From the inequalities
  $\dist_{A^\delta}(r,1)<\diam_{A^\delta}A< \dist_{A^\delta}(r,1)+\frac{\pi^2}{2\abs{\log\delta}}$,
  we see that
  \begin{equation}
    \diam_{A^\delta}A=\log\abs{\log r} +O(1)\quad \text{as} \quad r\to0.
  \end{equation}
\end{Rem}

\section{Geometric constructions with cells}

In this section we develop the two key geometric constructions used in
the proofs of our main theorems: cellular refinement and clustering in
fibers of proper covers.

\subsection{Refinement of cells}

In this section we show how cells with a $\hrho$-extension can be
\emph{refined} into cells with a $\hsigma$-extension for
$0<\sigma<\rho$. We remark that while the statement appears innocuous,
the proof actually requires the full strength of the fundamental lemma
for $\D\setminus\{0\}$ (Lemma~\ref{lem:fund-Dcirc}). The key
difficulty is to construct the refinement of an annulus fiber in a
manner that depends holomorphically on the base.

\begin{Thm}[Refinement theorem]\label{thm:cell-refinement}
  Let $\cC^\hrho$ be a (real) cell and $0<\sigma<\rho$. Then there
  exists a (real) cellular cover $\{f_j:\cC_j^\hsigma\to\cC^\hrho\}$
  of size $\poly_\ell(\rho,1/\sigma)$ where each $f_j$ is a cellular
  translate map.

  If $\cC^\hrho$ varies in a definable family (and $\sigma,\rho$ vary under
  the condition $0<\sigma<\rho<\infty$) then the cells $\cC_j$ and
  maps $f_j$ can also be chosen from a single definable family. If
  $\cC$ is algebraic of complexity $\beta$ then $\cC_j,f_j$ are
  algebraic of complexity $\poly_\ell(\beta)$.
\end{Thm}

The proof of the refinement theorem will occupy the remainder of this
subsection. We begin with the complex case, and indicate the necessary
modifications for the real case at the end. We proceed by induction on
$\ell$ (the case $\ell=0$ is trivial). Consider a cell
$\cC\odot\cF$. By applying the inductive hypothesis to $\cC$ and
replacing $\cC\odot\cF$ by each $\cC_j\odot f_j^*\cF$ we may assume
without loss of generality that the base $\cC$ already admits
$E$-extension, where $E$ will be chosen later. The case $\cF=*$
reduces to the inductive hypothesis directly with $E=\hsigma$.

As a notational convenience we allow ourselves to rescale the fiber
$\cF$ by a non-vanishing holomorphic function
$s\in\cO_b(\cC^\hsigma)$, where it is understood that the covering
cells $\cC_j$ that we construct will eventually be rescaled to cover
the original $\cF$. When we describe the covering of $\cF$ we allow
ourselves to use discs centered at a point $p\in\cO_b(\cC^\hsigma)$
where it is understood that such discs will be centered at the origin
in the covering cells $\cC_j$ that we construct, and the maps $f_j$
will translate the origin to $p$. We note that rescaling has the
effect of eventually rescaling the centers of the covering discs that
we construct below by a factor of $s$. In the algebraic case we will
always choose $s$ and $p$ to be algebraic of complexity $\poly_\ell(\beta)$
so that the maps that we eventually construct will indeed be of
complexity $\poly_\ell(\beta)$.

We will consider two separate cases: $\rho>1,\sigma=1$ and
$\rho=1,\sigma<1$. The general case can be reduced to a composition of
these two cases, first refining $\hrho$-extensions into
$\he1$-extensions and then refining $\he1$-extensions into
$\hsigma$-extensions

\subsubsection{The case $\rho>1$, $\sigma=1$}
In this case regardless of the type of $\cF$ we have
$\cF^\hrho=\cF^\delta$ where $\delta=1-\Theta(\rho^{-1})$. We will
consider the different fiber types separately.

Suppose $\cF$ is of type $D$. Up to rescaling $\cF=D(1)$.  Then we
take $E=\he1$. Our goal to cover $D(1)$ by discs whose
$\he1$-extensions remain in
$D(\delta^{-1})=D(1+\Theta(\rho^{-1}))$. One can use for example any
disc of radius $O(\rho^{-1})$ centered in $D(1)$; clearly
$\poly(\rho)$ such discs suffice to cover $D(1)$.

Suppose $\cF$ is of type $D_\circ$. Up to rescaling
$\cF=D_\circ(1)$. Again we take $E=\he1$. Inside $D_\circ(1)$ we
choose the disc $D_\circ$ such that $D_\circ^{\he1}$ remains in
$D_\circ(1)$. It remains to cover the annulus
$D_\circ(1)\setminus D_\circ$, which has inner radius $\Omega(1)$,
using $\poly(\rho)$ discs whose $\he1$-extensions remain in
$D_\circ(1)^{\he\rho}$. This can be done as in the case $\cF=D(r)$.

Suppose $\cF$ is of type $A$. Up to rescaling $\cF=A(r,1)$. We take
$E=\he1\he{\hat\rho}$ where $\hat\rho$ will be chosen later. By the
fundamental lemma for $\D\setminus\{0\}$ (Lemma~\ref{lem:fund-Dcirc})
applied to $r$ on the cell $\cC^{\he1}$ one of the following holds
\begin{align}\label{eq:fund-r-v1}
  \log |r(\cC^{\he1})| &\subset (-\infty,-\Omega_\ell(1/\hat\rho)) & &\text{or} & \diam(\log|\log|r(\cC^{\he1})||;\R)&=O_\ell(\hat\rho).
\end{align}
We will use the following simple lemma (see Figure~\ref{fig:circle}).

\begin{Lem}\label{lem:annulus-cover-v1}
  Let $\alpha\in(0,1]$ (resp. $\beta\in[1,\infty)$) and suppose
  $r<\alpha$ (resp. $r\beta<1$) uniformly in $\cC^{\he1}$. Then one
  can cover $\cC\odot S^{1-\Omega(1/\rho)}(\alpha)$
  (resp. $\cC\odot S^{1-\Omega(1/\rho)}(\beta r)$) by $\poly(\rho)$
  cells whose $\he1$-extensions remain in $(\cC\odot\cF)^\hrho$.
\end{Lem}
\begin{proof}
  The $\he1$-extension of a disc of radius $O(\alpha/\rho)$ centered at a
  point of $S(\alpha)$ remains in
  $S^{1-\Omega(1/\rho)}(\alpha)$. Moreover for any point in
  $\cC^{\he1}$ this remains in $A^\hrho(r,1)$ since $r<\alpha$. We
  choose a collection of $\poly(\rho)$ such discs to cover
  $S^{1-\Omega(1/\rho)}(\alpha)$. Then $\cC\odot D_i$ gives the
  required covering. The respective case is similar, with $S(\alpha)$
  replaced by $S(\beta r)$.
\end{proof}

\begin{figure}
  \centering
  \includegraphics[width=.6\textwidth]{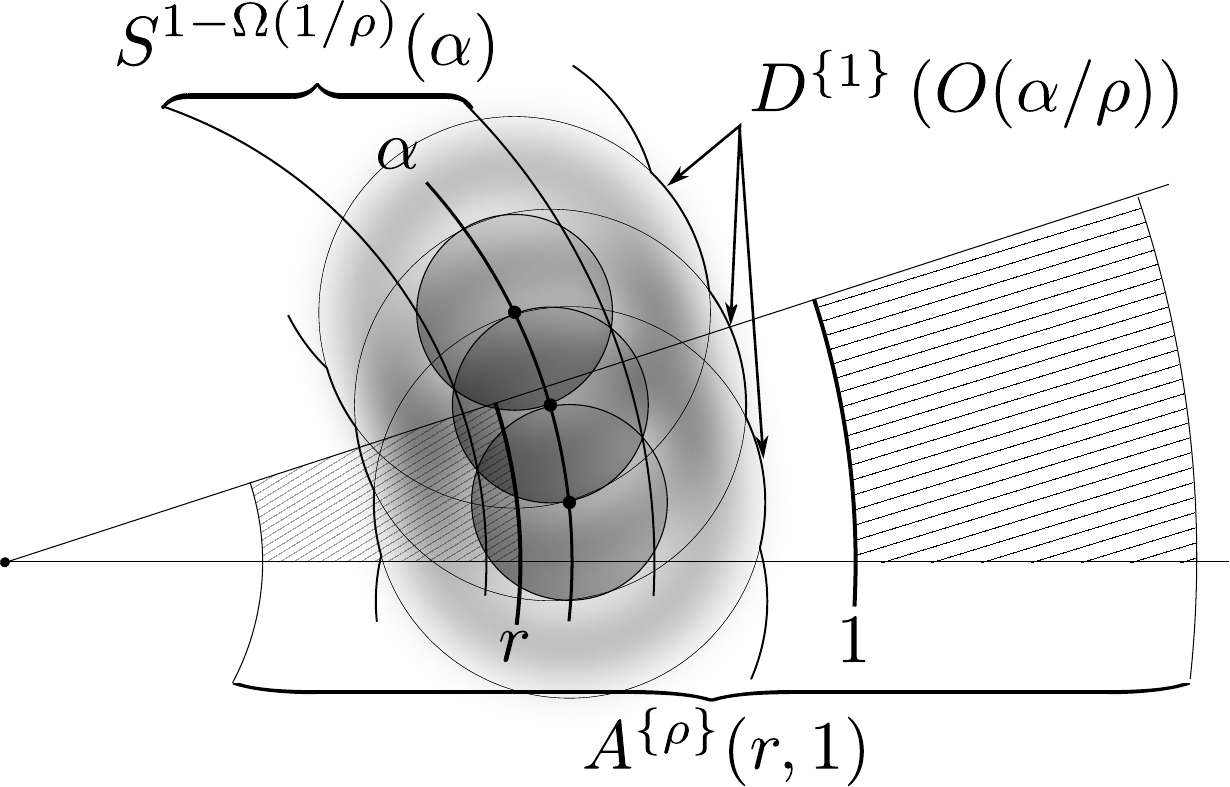}
  \caption{Proof of Lemma~\ref{lem:annulus-cover-v1}.}
  \label{fig:circle}
\end{figure}

Below we refer to the map $z\to\log|z|$ as the \emph{logarithmic
  scale}. An annulus $A=A(r_1,r_2)$ corresponds in this scale to an
interval $(\log \abs{r_1},\log \abs{r_2})$ whose length we refer to as the
\emph{logarithmic width} of $A$.

Suppose that we are in the first case of~\eqref{eq:fund-r-v1}. Inside
$A(r,1)$ we choose the annulus $A$ such that $A^{\he1}=A(r,1)$: this
can be done with $\hat\rho=\Omega_\ell(1)$. It remains to cover the two
annuli components of $A(r,1)\setminus A$. Each of these has
logarithmic width $O(1)$. Lemma~\ref{lem:annulus-cover-v1} allows us
to cover subannuli of logarithmic width $\Omega(1/\rho)$: with
$\alpha$ for the outer component and with $\beta$ for the inner
component. Clearly $\poly(\rho)$ applications of the lemma suffice to
cover the two components. See Figure~\ref{fig:refining} for an
illustration.

\begin{figure}
  \centering
  \includegraphics[width=.6\textwidth]{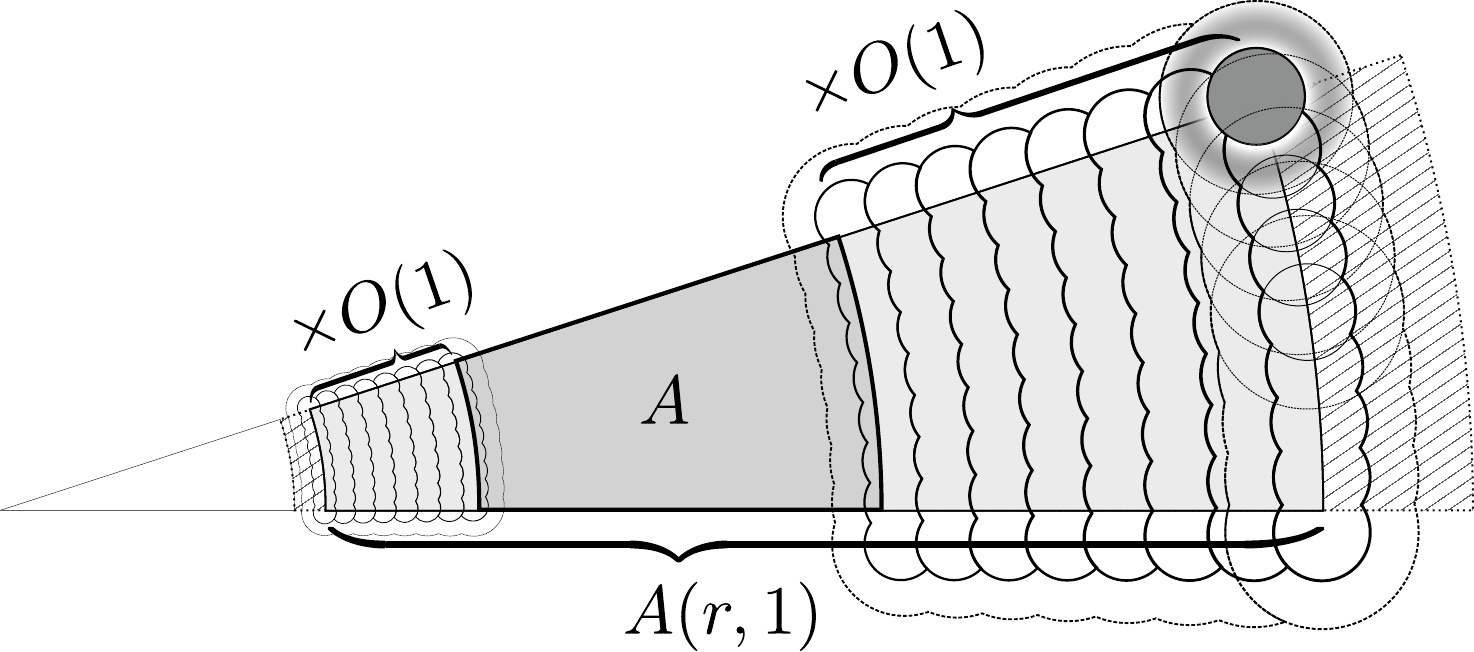}
  \caption{Refinement of an annulus.}
  \label{fig:refining}
\end{figure}

Suppose that we are in the second case of~\eqref{eq:fund-r-v1}.  We
may choose $\hat\rho=\Omega_\ell(1)$ such that the ratio of $\log |r|$
between any two points of $\cC^{\he1}$ is in $(9/10,10/9)$. We again
distinguish two cases: first, suppose that we can uniformly over
$\cC^{\he1}$ choose inside $A(r,1)$ the annulus $A$ such that
$A^{\he1}=A(r,1)$. In this case we can proceed as above. Otherwise,
over some point in $\cC^{\he1}$ we have $r(\vz_{1..\ell})=r_0$ with
$\log|r_0|=-O(1)$. We now apply Lemma~\ref{lem:annulus-cover-v1} to
cover the annulus given in the logarithmic scale by $(2/3\log|r_0|,0)$
(using $\alpha$) and the annulus given in the logarithmic scale by
$(\log|r|,\log|r|-2/3\log|r_0|)$ (using $\beta$). Crucially, the ratio
condition on $\log|r|$ in $\cC^{\he1}$ ensures that these two annuli
remain in $A(r,1)$ and cover it uniformly over $\cC^{\he1}$. Since the
width of each of these annuli is $O(1)$ we see as above that
$\poly(\rho)$ applications of the lemma suffice to cover them.

\subsubsection{The case $\rho=1$, $\sigma<1$}
In this case we have $\cF^\hsigma=\cF^\e$ where
$\e=\Theta(\sigma)$ in type $D$ and $\e= e^{-\Theta(1/\sigma)}$ in types
$D_\circ,A$. We will consider the different fiber types separately.

Suppose $\cF$ is of type $D$. Up to rescaling $\cF=D(1)$. Then we take
$E=\hsigma$. Our goal to cover $D(1)$ by discs whose
$\he\sigma$-extensions remain in $D(1)^{\he1}$. One can use for
example any disc of radius $O(\sigma)$ centered in $D(1)$; clearly
$\poly(1/\sigma)$ such discs suffice to cover $D(1)$.

Suppose now that $\cF$ is of type $D_\circ$ or $A$. Up to rescaling
$\cF=D_\circ(1)$ or $\cF=A(r,1)$. We will use the following analog of
Lemma~\ref{lem:annulus-cover-v1}. The $\beta$-case of the lemma is
valid only for type $A$.
\begin{Lem}\label{lem:annulus-cover-v2}
  Let $\alpha\in(0,1]$ (resp. $\beta\in[1,\infty)$) and suppose
  $r<\alpha$ (resp. $r\beta<1$) uniformly in $\cC^\hsigma$. Then one
  can cover $\cC\odot S^{\Omega(1)}(\alpha)$
  (resp. $\cC\odot S^{\Omega(1)}(\beta r)$) by $\poly(1/\sigma)$
  cells whose $\hsigma$-extensions remain in $(\cC\odot\cF)^{\he1}$.
\end{Lem}
\begin{proof}
  The $\he{1}$-extension of a disc of radius $\alpha/10$ centered at a
  point of $S(\alpha)$ remains in
  $S^{e^{-\pi^2/2}}(\alpha)\subset\cF^{\he1}$. We choose a collection
  of $O(1)$ such discs to cover $S^{\Omega(1)}(\alpha)$. By what was
  already proved for discs, we can further cover each of the discs by
  $\poly(1/\sigma)$ discs whose $\hsigma$-extensions remain in
  $S^{e^{-\pi^2/2}}(\alpha)$. Thus for any point in $\cC^\hsigma$
  these extensions remain in $\cF^{\he1}$ (in type $A$ we use that
  $r<\alpha$). If $\{D_i\}$ denotes the collection of all discs
  obtained in this process then $\cC\odot D_i$ gives the required
  covering. The case of $\cC\odot S^{\Omega(1)}(\beta r)$ is similar,
  with $S(\alpha)$ replaced by $S(\beta r)$.
\end{proof}

Suppose $\cF=D_\circ(1)$. We take $E=\hsigma$. We first embed
$D_\circ(e^{-\Theta(1/\sigma)})$ in $D_\circ(1)$ so that the
$\hsigma$-extensions remains in $D_\circ(1)$. It remains to cover the
annulus $A(e^{-\Theta(1/\sigma)},1)$, or in the logarithmic scale
$(-\Theta(1/\sigma),0)$. Lemma~\ref{lem:annulus-cover-v2} allows us to
cover subannuli of logarithmic width $\Omega(1)$, and indeed
$\poly(1/\sigma)$ applications suffice to cover the annulus.

Finally suppose $\cF=A(r,1)$. We take $E=\hsigma\he{\hat\rho}$ where
$\hat\rho$ will be chosen later. By the fundamental lemma for
$\D\setminus\{0\}$ (Lemma~\ref{lem:fund-Dcirc}) applied to $r$ on the
cell $\cC^\hsigma$ one of the following holds
\begin{align}\label{eq:fund-r-v2}
  \log |r(\cC^\hsigma)| &\subset (-\infty,-\Omega_\ell(1/\hat\rho))& &\text{or} & \diam(\log|\log|r(\cC^\hsigma)||;\R)&=O_\ell(\hat\rho).
\end{align}
  
Suppose that we are in the first case of~\eqref{eq:fund-r-v2}.  Inside
$A(r,1)$ we choose the annulus $A$ (uniformly over $\cC^\hsigma$) such
that $A^\hsigma=A(r,1)$: this is possible if
$\log|r|=-\Omega(1/\sigma)$, i.e. it can be done with
$\hat\rho=\Omega_\ell(\sigma)$. It remains to cover the two annuli
components of $A(r,1)\setminus A$. Each of these has width
$O(1/\sigma)$ in the logarithmic
scale. Lemma~\ref{lem:annulus-cover-v2} allows us to cover subannuli
of width $\Omega(1)$: with $\alpha$ for the outer component and with
$\beta$ for the inner component. Clearly $\poly(1/\sigma)$
applications of the lemma suffice to cover the two components.

Suppose that we are in the second case of~\eqref{eq:fund-r-v2}. We may
choose $\hat\rho=\Omega_\ell(1)$ such that the ratio of $\log |r|$ between
any two points of $\cC^{\he1}$ is in $(9/10,10/9)$. We again
distinguish two cases: first, suppose that we can uniformly over
$\cC^\hsigma$ choose inside $A(r,1)$ the annulus $A$ such that
$A^\hsigma=A(r,1)$. In this case we can proceed as above. Otherwise,
over some point in $\cC^\hsigma$ we have $r(\vz_{1..\ell})=r_0$ with
$\log|r_0|=-O(1/\sigma)$. We now apply
Lemma~\ref{lem:annulus-cover-v2} to cover the annulus given in the
logarithmic scale by $(2/3\log|r_0|,0)$ (using $\alpha$) and the
annulus given in the logarithmic scale by
$(\log|r|,\log|r|-2/3\log|r_0|)$ (using $\beta$). Crucially, the ratio
condition on $\log|r|$ in $\cC^\hsigma$ ensures that these two annuli
remain in $A(r,1)$ and cover it uniformly over $\cC^\hsigma$. Since
the logarithmic width of each of these annuli is $O(1/\sigma)$ we see
as above that $\poly(1/\sigma)$ applications of the lemma suffice to
cover them.

\subsubsection{The real case}

Only a very minor modification is needed in order to treat the real
case. If $\cC$ is real then all the rescalings performed during the
proof are also real. Wherever we construct a collection of discs to
cover a domain $\cF$ by discs in the complex case, we now choose a
collection of discs with real centers to cover $\R_+\cF$. This ensures
that the cells that we construct are real, and the rest of the proof
remains unchanged.

\subsubsection{Uniformity over families and monomialization}

The statement regarding uniformity over families follows by inspection
of the proof. By induction the refinement maps chosen for the base can
all be chosen from a single definable family. The covering constructed
for the fiber, after a uniform rescaling, consists of discs or annuli
with a constant center and radius. The family of all such discs or
annuli is certainly definable. We will also make use of the following
remark in the sequel.

\begin{Rem}[Refinement and monomialization]\label{rem:refinement-vs-monom}
  The type of any cell $\cC_j$ obtained by refinement of $\cC^\hrho$
  is obtained from the type of $\cC$ by possibly replacing some
  $D_\circ,A$ fibers by $D$ fibers. Moreover
  $(f_j)_*:\pi_1(\cC_j)\to\pi_1(\cC)$ is the natural injection. In
  particular if $F:\cC\to\C\setminus\{0\}$ then $\valpha(f_j^*F)$ is
  obtained from $\valpha(F)$ by eliminating indices corresponding to
  fibers that were replaced by $D$ in $\cC_j$.
\end{Rem}

Remark~\ref{rem:refinement-vs-monom} follows immediately from the
proof: it suffices to note that $D$ fibers are always covered by $D$
fibers, and $D_\circ,A$ fibers are covered by a collection of $D$
fibers and possibly one additional fibers of type $D_\circ,A$; and
this additional fiber is always centered at zero, and hence homotopy
equivalent to the original fiber by the injection map.

\subsection{Monomial cells}

Let $\cC$ be a cell of length $\ell$. An \emph{admissible monomial} on
$\cC$ is a function of the form $c\cdot\vz^\valpha$ where $c\in\C$ and
$\valpha=\valpha(f)$ is the associated monomial of some function
$f\in\cO_b(\cC)$.

\begin{Def}[Monomial cell]\label{def:monom-cell}
  We will say that a cell $\cC$ is \emph{monomial} if $\cC=*$; or if
  $\cC=\cC_{1..\ell}\odot\cF$ where $\cC_{1..\ell}$ is monomial and
  the radii involved in $\cF$ are admissible monomials on
  $\cC_{1..\ell}$.
\end{Def}

Using the refinement theorem (Theorem~\ref{thm:cell-refinement}) and
the monomialization lemma (Lemma~\ref{lem:monomial}) we prove the
following proposition.

\begin{Prop}\label{prop:monomial-cover}
	Let $\cC^\hrho$ be a (real) cell and $0<\sigma<\rho$. Then there
    exists a (real) cellular cover $\{f_j:\cC_j^\hsigma\to\cC^\hrho\}$
	where each $\cC_j$ is a monomial cell and each $f_j$ is a cellular
	translate map.

	If $\cC^\hrho$ varies in a definable family $\Lambda$  then  the cover
	has size $\poly_{\Lambda}(\rho,1/\sigma)$ and the cells $\cC_j$ and
	maps $f_j$ can be chosen from a single definable family. If
	$\cC$ is algebraic of complexity $\beta$ then  the cover
	has size $\poly_{\ell}(\beta,\rho,1/\sigma)$ and
	$\cC_j,f_j$ are algebraic of complexity $\poly_\ell(\beta)$.
\end{Prop}

Proposition~\ref{prop:monomial-cover} implies that we could use
monomial cells, rather than general complex cells, as our standard
models for cellular covers. In some cases this is more convenient, as
the monomial cells have a more transparent combinatorial and algebraic
structure. However for the most part we have chosen in this paper to
state our constructions for general complex cells. 

The proof of Proposition~\ref{prop:monomial-cover} will occupy the
remainder of this section. We start by applying the refinement theorem
(Theorem~\ref{thm:cell-refinement}) to $\cC$ with some
$\he{\hat\sigma}=\he{\hat\sigma_1}\he{\hat\sigma_2}$ to be chosen
later. Recall that all of the cells constructed in the refinement
theorem can be chosen from one definable family independent of
$\hat\sigma$. By the monomialization lemma (Lemma~\ref{lem:monomial}),
if $r(\vz)$ is any of the radii involved in the definition of one of
these cells $\cC_j$ then we have $r(\vz)=\vz^{\valpha(r)}U(\vz)$ where
\begin{equation}\label{eq:monomial-cell-U}
  \begin{aligned}
    \diam(\Re\log U(\cC_j^{\he{\hat\sigma_2}});\R)&< O_\Lambda(\hat\sigma_1) \\
    \diam(\Re\log U(\cC_j^{\he{\hat\sigma_2}});\R)&< \poly_\ell(\beta)\cdot\hat\sigma_1  
  \end{aligned}
\end{equation}
in the subanalytic and algebraic cases respectively. We will now
construct a monomial cell $\tilde\cC_j^\hsigma$ such that
$\cC_j\subset\tilde\cC_j\subset\tilde\cC_j^\hsigma\subset\cC^{\he{\hat\sigma}}$,
and the identity map then gives the covering of $\cC_j$ by a monomial
cell as required.

We construct $\tilde\cC_j$ by induction on $\ell$. If $\cC_j$ is of
length zero or one then it is already monomial so we may take
$\tilde\cC_j:=\cC_j$. If $\cC_j:=(\cC_j)_{1..\ell}\odot\cF$ then we
set $\tilde\cC_j:=\widetilde{(\cC_j)_{1..\ell}}\odot\tilde\cF$ where
$\tilde\cF$ is defined as follows. Suppose $\cF=A(r_1,r_2)$ and write
$r_1=\vz^{\valpha_1}U_1$ and $r_2=\vz^{\valpha_2}U_2$. Then
$\tilde\cF=A(\tilde r_1,\tilde r_2)$ for
\begin{align}
  \tilde r_1&=\vz^{\valpha_1} \min|U_1(z)| &
  \tilde r_2&=\vz^{\valpha_2} \max|U_2(z)|,
\end{align}
where the minimum and maximum are taken over
$\vz_{1..\ell}\in\widetilde{(\cC_j)_{1..\ell}^\hsigma}$. It is clear
that $\cC_j\subset\tilde\cC_j$, and what remains to be verified is
that
$\tilde\cC_j^\hsigma\subset\cC_j^{\he{\hat\sigma}}$. By~\eqref{eq:monomial-cell-U},
for an appropriate choice of $\hat\sigma_1$ we can make
$\diam(\Re\log U(\cC_j^{\he{\hat\sigma_2}});\R)<1$, which implies that
$\tilde r_1/r_1\ge1/e$ and $\tilde r_2/r_1<e$ on
$\cC_j^{\he{\hat\sigma_2}}$. Now choosing $\hat\sigma_2<\sigma$ and
also $\he{\hat\sigma}<\hsigma/e$ we have indeed
$\tilde\cC_j^\hsigma\subset\cC_j^{\he{\hat\sigma}}$.

\subsection{Clustering in fibers of proper covering maps}

In this section we consider a proper covering map
$\pi:Z\subset\cC^\hrho\times\C\to\cC^\hrho$. We choose one point in
the fiber of $\pi$ as the \emph{center}. Our goal is to group the
remaining points into clusters based on their distances from the
center, so that two points belong to the same cluster if their
distances from the center are close in a suitable sense. The key
difficulty of the construction is guaranteeing that these clusters
vary holomorphically with the basepoint $\vz\in\cC$, and in particular
that their combinatorial structure remains constant.

We treat the complex case in~\secref{sec:cover-clustering} and the
real case in~\secref{sec:cover-clustering-real}.

\subsubsection{The general (complex) setting}
\label{sec:cover-clustering}

Let $\cC^\hrho$ be a cell and let $Z\subset\cC^\hrho\times\C$ be an
analytic set such that the natural projection $\pi:Z\to\cC^\hrho$ is a
proper covering map. Let $\nu$ denote the degree of $\pi$ and set
$\hat\cC:=\cC_{\times\nu!}$. Let $\hat Z\subset\hat\cC\times\C$ be the
pullback $(R_{\nu!},\id)^*Z$. Then the sections $y_j:\hat\cC\to\C$ of
$\hat Z$ over $\hat\cC$ are univalued, and we denote their collection
by $\Sigma$. Each section $y_j$ in fact extends holomorphically to the
cell $\hat\cC^{\he{\nu!\cdot\rho}}$ by
Proposition~\ref{prop:ext-v-cover}, but this exponential (in $\nu$)
loss in the size of the extension is too large for our
purposes. Instead, we have the following lemma. Denote by
$\nu_j\le\nu$ the size of the $\pi_1(\cC)$-orbit of $y_j$ thought of
as a multivalued section over $\cC$.

\begin{Lem}\label{lem:y_j-monodromy}
  Let $y_j\in\Sigma$. Then $y_j$ is already univalued as a section
  over the cover $\hat\cC_j:=\cC_{\times\nu_j}$, and extends
  holomorphically to $\hat\cC_j^{\he{\nu_j\cdot\rho}}$.
\end{Lem}
\begin{proof}
  The group $\pi_1(\cC)$ is abelian.  It is enough to check that for
  any $g\in\pi_1(\cC)$ we have $g^{\nu_j}(y_j)=y_j$. This is
  elementary: the $\<g\>$-action induces a partition of the orbit
  $\pi_1(\cC)\cdot y_j$ of size $\nu_j$ into $\<g\>$-orbits, and
  $\pi_1(\cC)$ acts transitively on these orbits. Thus they are all of
  the same size which divides $\nu_j$.
\end{proof}

Let $y_i,y_j,y_k\in\Sigma$ be three distinct sections. Since $\pi$ is
unramified, the sections are pairwise distinct over any point of
$\hat\cC$. We let $\nu_{i,j,k}:=\lcm(\nu_i,\nu_j,\nu_k)$ and
$\hat\cC_{i,j,k}:=\cC_{\times\nu_{i,j,k}}$. We define a map
$s_{i,j,k}$ as follows,
\begin{equation}
  s_{i,j,k}:\hat\cC_{i,j,k}^{\he{\nu_{i,j,k}\cdot\rho}}\to\C\setminus\{0,1\}, \qquad s_{i,j,k}=\frac{y_i-y_j}{y_i-y_k}.
\end{equation}
By the fundamental lemma for maps into $\C\setminus\{0,1\}$
(Lemma~\ref{lem:fund-C01}) one of the following holds:
\begin{equation}\label{eq:sijk-fund}
  \begin{gathered}
    s_{i,j,k}(\hat\cC_{i,j,k})\subset B(\{0,1,\infty\},e^{-\Omega_\ell(1/(\nu^3\rho))};\C P^1) \quad\text{or}\\
    \diam(s_{i,j,k}(\hat\cC_{i,j,k});\C\setminus\{0,1\})=O_\ell(\nu^3\rho).
  \end{gathered}
\end{equation}
We remark that~\eqref{eq:sijk-fund} holds if we replace
$\hat\cC_{i,j,k}$ by $\hat\cC$, since $s_{i,j,k}$ on $\hat\cC$ factors
through $\hat\cC_{i,j,k}$.

Fix $y_i\in\Sigma$. We will cluster the remaining sections into annuli
according to their relative distances from $y_i$, which are expressed
by the quantities $s_{i,j,k}$. Since these quantities are invariant
under affine transformations of $\C$, we may assume for simplicity of
the notation that $y_i=0$. We record a useful corollary
of~\eqref{eq:sijk-fund} in this normalization.

\begin{Lem}\label{lem:yj-v-yk}
  Suppose $\rho=O_\ell(1/\nu^3)$. Let $j,k\neq i$ and write
  $R:=\log|y_j/y_k|$. One of the following holds:
  \begin{equation}\label{eq:yj-v-yk}
    \begin{gathered}
      R\rest{\hat\cC} < -\Omega_\ell(1/\nu^3\rho) \quad\text{or}\quad  R\rest{\hat\cC} > \Omega_\ell(1/\nu^3\rho), \\
    \diam(R(\hat\cC),\R) = O_\ell(\nu^3\rho) \\
    \frac{\max_{\vz\in\hat\cC} R(z)}{\min_{\vz\in\hat\cC} R(z)} < 1+O_\ell(\nu^3\rho).
    \end{gathered}
  \end{equation}
\end{Lem}
\begin{proof}
  If at some point in $\hat\cC$ we have $y_j/y_k\in A(1/2,2)$ then for
  $\rho=O_\ell(1)$ we have by~\eqref{eq:sijk-fund} that
  $y_j/y_k(\hat\cC)\subset A(1/4,4)$. In this domain the
  $\log|\cdot|$-distance is bounded up to constant by the
  $\C\setminus\{0,1\}$ distance so
  $\diam(R(\hat\cC);\R)=O_\ell(\nu^3\rho)$. It remains to consider the
  case $y_j/y_k(\hat\cC)\subset D_\circ(1/2)$ (or $A(2,\infty)$ which
  is the same up to inversion). In the first case
  of~\eqref{eq:sijk-fund} we have $R=-\Omega_\ell(1/(\nu^3\rho))$. In
  the second case of~\eqref{eq:sijk-fund}, since the
  $\C\setminus\{0,1\}$ metric is equivalent to the $D_\circ(1)$ metric
  in $D_\circ(1/2)$ we have (as in the fundamental lemma for
  $D_\circ$, Lemma~\ref{lem:fund-Dcirc}) the estimate
  \begin{equation}
    \diam(\log|R(\hat\cC)|;\R) = O_\ell(\nu^3\rho).
  \end{equation}
  By $\rho=O_\ell(1/\nu^3)$ the right hand side is $O_\ell(1)$ and
  exponentiating we conclude that $R(\hat\cC)$ varies multiplicatively
  by a factor of size at most $1+O_\ell(\nu^3\rho)$.
\end{proof}

We fix a quantity $0<\gamma<1$ which we call the \emph{gap}. Pick an
arbitrary point $p\in\hat\cC$ and let
\begin{equation}
  S_i:=\{\log|y_j(p)|:y_j\in\Sigma, y_j\neq y_i\}\subset\R.
\end{equation}
We say that two points $s,s'$ in $S_i$ belong to the same cluster if
they are connected in the transitive closure of the relation
$|s-s'|<5|\log\gamma|$. We order the clusters with respect to $<$ on
$\R$. Let $m_i$ be the number of clusters in $S_i$. For
$1\le q\le m_i$ we let $I_{i,q}$ denote the minimal closed interval
containing the $2|\log\gamma|$-neighborhood of the $q$-th cluster
$S_{i,q}$. For each $q$ we arbitrarily choose $\hat y_{i,q}\in\Sigma$
to be one of the sections with $\log|\hat y_{i,q}(p)|\in I_{i,q}$, and
call it the \emph{center} of the cluster. We define $l_{i,q}$ and
$r_{i,q}$ by
\begin{equation}
  e^{I_{i,q}} = [l_{i,q}|\hat y_{i,q}(p)|,r_{i,q}|\hat y_{i,q}(p)|].
\end{equation}
Note that
\begin{align}\label{eq:cluster-edge-bounds}
  \gamma^{5\nu}<&l_{i,q}<\gamma^2, & \gamma^{-2}<&r_{i,q}<\gamma^{-5\nu}.
\end{align}
We also fix some $\delta<1$ arbitrarily close to $1$ (merely to ensure
that the boundary circles below are covered).

We define three types of fibers over $\hat\cC$ as follows:
\begin{gather}
  \cF_{i,q}=A^\delta(l_{i,q}\hat y_{i,q},r_{i,q}\hat y_{i,q}) \qquad q=1,\ldots,m_i \\
  \cF_{i,q+}= A^\delta(r_{i,q}\hat y_{i,q},l_{i,q+1}\hat y_{i,q+1}), \qquad q=1,\ldots,m_i-1 \\
  \cF_{i,0+}=D_\circ^\delta(l_{i,1}\hat y_{i,1}), \qquad \cF_{i,m+} = A^\delta(r_{i,q}\hat y_{i,q},\infty).
\end{gather}
Note that each of these fibers actually depends only on
$y_i,\hat y_{i,q}$ for $\cF_{i,q}$ and on
$y_i,\hat y_{i,q},\hat y_{i,q+1}$ for $\cF_{i,q+}$. Thus each fiber
actually arises as a pullback by the projection
$\hat\cC\to\hat\cC_{i,j,k}$ of a fiber defined over $\hat\cC_{i,j,k}$
for a suitable choice of $i,j,k$. We denote this cover by
$\hat\cC_{i,q}$ or $\hat\cC_{i,q+}$.

The key properties of these domains are summarized in the following
proposition.
\begin{Prop}\label{prop:clusters}
  Suppose $1/\rho>\poly_\ell(\nu,|\log\gamma|)$. Then the following hold
  uniformly over $\hat\cC$:
  \begin{enumerate}
  \item The fibers $\cF_{i,q},\cF_{i,q+}$ are well-defined and cover
    $\C\setminus\{0\}$.
  \item The domains $\cF_{i,q}^\gamma\setminus\cF_{i,q}$ do not contain any
    of the points $y_j$ for $q=1,\ldots,m_i$.
  \item The domains $\cF_{i,q+}^\gamma$ do not contain any of the
    points $y_j$ for $q=0,\ldots,m_i$.
  \end{enumerate}
\end{Prop}
\begin{proof}
  The only non-trivial assertion in the first statement is that
  $\cF_{i,q+}$ is well-defined, i.e. that
  $r_{i,q}\hat y_{i,q}<l_{i,q+1}\hat y_{i,q+1}$ uniformly over
  $\hat\cC$. At $p$ we have by construction
  \begin{equation}
    \log |\hat y_{i,q+1}(p)/\hat y_{i,q}(p)| > \log(r_{i,q}/l_{i,q+1})+|\log\gamma|.
  \end{equation}
  We need to prove that
  \begin{equation}
    \log |\hat y_{i,q+1}/\hat y_{i,q}| > \log(r_{i,q}/l_{i,q+1})
  \end{equation}
  uniformly over $\hat\cC$. This follows easily from
  Lemma~\ref{lem:yj-v-yk} for an appropriate choice of $\rho$. The
  only non-trivial case is the last one, which follows if one recalls
  that $\log(r_{i,q}/l_{i,q+1})<10\nu|\log\gamma|$.
  
  The third statement follows from the second: over $p$ all the points
  $y_j$ except $y_i=0$ lie in the domains $\cF_{i,q}$, and assuming
  that they never cross into the boundaries
  $\cF_{i,q}^\gamma\setminus\cF_{i,q}$ this remains true uniformly
  over $\hat\cC$. We proceed to the proof of the second statement. We
  will show that $y_j$ does not belong to the outer boundary of
  $\cF_{i,q}^\gamma\setminus\cF_{i,q}$ (the case of the inner boundary
  is similar). By construction over $p$ one of the following holds:
  \begin{align}
     \log|y_j(p)/\hat y_q(p)|&< \log r_{i,q}-2|\log\gamma|, &  \log|y_j(p)/\hat y_q(p)|&> \log r_{i,q}+3|\log\gamma|.
  \end{align}
  We must prove that one of
  \begin{align}
     \log|y_j/\hat y_q|&< \log r_{i,q}, &  \log|y_j/\hat y_q|&> \log r_{i,q}+|\log\gamma|,
  \end{align}
  holds uniformly over $\hat\cC$. This follows in the same manner as
  the first statement.
\end{proof}

Since the sections $y_j$ do not meet the $\cF_{i,q+}$ and
$\partial\cF_{i,q}$, it follows that each section $y_j\neq y_i$ lies
in a single $\cF_{i,q}$ uniformly over $\hat\cC$. We say that such a
section belongs to the cluster $\cF_{i,q}$.

\begin{Rem}\label{rem:single-cluster}
  If we focus our attention on a single cluster $\cF_{i,q}$ then it is
  convenient, up to an affine transformation over $\hat\cC_{i,q}$, to
  assume that both $y_i=0$ and $\hat y_{i,q}=1$. With $\gamma,\rho$ as
  in Proposition~\ref{prop:clusters} we then have
  $\cF_{i,q}=A^\delta(l_{i,q},r_{i,q})$. For any $y_j$ in the
  $\cF_{i,q}$ cluster our choice of $\rho$ ensures that the second
  option of~\eqref{eq:sijk-fund} holds, and we have
  $\diam(y_j(\hat\cC_{i,q}),\C\setminus\{0,1\})=O_\ell(\nu^3\rho)$.
\end{Rem}

\subsubsection{The real setting}
\label{sec:cover-clustering-real}

Suppose that $\cC^\hrho$ is a real cell and
$Z\subset\cC^\hrho\times\C$ is real, i.e. invariant under
$\vz\to\bar\vz$. In this case we would like to construct the fibers
$\cF_{i,q},\cF_{i,q+}$ to be real as well. However, the construction
above produces fibers whose centers and radii are given in terms of
the sections $y_j$, which are generally not real. We now modify this
construction to produce real fibers. We fix $p\in\R\hat\cC$.

The pullbacks $\hat\cC,\hat Z$ of $\cC,Z$ by $R_{\nu!}$ are also
real. Since $\hat Z$ is a cover, each section $y_j$ is
either always real or always non-real on $\R\hat\cC$. We denote the
former sections by $\Sigma_\R$ and the latter by $\Sigma_\C$. If
$y_j\in\Sigma$ then by the symmetry of $\hat Z$ there exists a
symmetric section $y_{\bar j}\in\Sigma$ defined by
$y_{\bar j}(\vz):=\overline{y_j(\bar\vz)}$.

\begin{Def}
  Let $y_j\in\Sigma_\C$. Making a real affine transformation we pass
  to a chart where $y_j(p)=i,y_{\bar j}(p)=-i$. Then the pair
  $y_j,y_{\bar j}$ is called \emph{admissible} if $D(1/10)$ contains
  none of the points $y_k(p)$ for $y_k\in\Sigma$.
\end{Def}
An \emph{admissible center} is a map of the form $(y_j+y_{\bar j})/2$
where $y_j\in\Sigma_\R$ or $y_j,y_{\bar j}\in\Sigma_\C$ is an
admissible pair. Every admissible center is real on $\R\hat\cC$. We
denote the set of admissible centers by $\{c_1,\ldots,c_t\}$.
\begin{Lem}\label{lem:select-admissible}
  Let $y_j\in\Sigma_\C$ and by a real affine transformation pass to a
  chart where $y_j(p)=i,y_{\bar j}(p)=-i$. Then $D(1/5)$ contains
  $c_l(p)$ for an admissible center $c_l$.
\end{Lem}
\begin{proof}
  If $y_j,y_{\bar j}$ is admissible then the claim is
  obvious. Otherwise there exists a section $y_k(p)\in D(1/10)$. If
  $y_k\in\Sigma_\R$ we can take $c_l=y_k$. Otherwise, rescaling to
  make $y_k(p)=i$ we see that it will be enough to prove the claim for
  $y_k$. Repeating this reduction at most $\nu$ times if necessary
  finishes the proof.
\end{proof}

The motivation for the notion of admissibility comes from the
following lemma.

\begin{Lem}\label{lem:admissible-collision}
  Suppose $1/\rho=\poly_\ell(\nu)$. For every section $y_j$ and
  admissible center $c_l$ we have $c_l\neq y_j$ uniformly over
  $\hat\cC$, unless $c_l=y_j\in\Sigma_\R$.
\end{Lem}
\begin{proof}
  If $c_l=y_k\in\Sigma_\R$ for some $k$ then the claim follows from
  the fact that $\hat Z$ is a cover whose sections are pairwise
  distinct over any point of $\hat\cC$. Assume therefore that
  $c_l=(y_k+y_{\bar k})/2$ for $y_k\in\Sigma_\C$. We make a real
  affine transform to pass to a chart where $y_k(p)=i$. Suppose that
  at some point $\tilde p\in\hat\cC$ we have $y_j=(y_k+y_{\bar k})/2$,
  i.e. $s_{j,k,\bar k}(\tilde p)=-1$. Then by~\eqref{eq:sijk-fund} we
  have $|s_{j,k,\bar k}(p)+1|=O_\ell(\nu^3\rho)$. For a suitable
  choice of $\rho$ this readily implies that $|y_j(p)|<1/10$,
  contradicting the admissibility of $c_l$.
\end{proof}

Lemma~\ref{lem:admissible-collision} implies that we can define maps
$\tilde s_{l,j,k}$ as follows
\begin{equation}
  \tilde s_{l,j,k}:\hat\cC_{l,j,k}^{\he{\nu_{l,j,k}\cdot\rho}}\to\C\setminus\{0,1\}, \qquad \tilde s_{i,j,k}=\frac{c_l-y_j}{c_l-y_k}.
\end{equation}
Fix an admissible center $c_l$. We will cluster the sections $y_j$
into annuli according to their relative distances from $c_l$ in
analogy with the complex construction. As before, we make a real
affine transformation and assume that $c_l=0$.

We define the intervals $I_{l,q}$ in the same way as in the complex
setting. However, a small variation is needed in the choice of
$\hat y_{l,q}$, which must be real on $\R\hat\cC$ to maintain the real
structure of our construction. Let $y_j$ be one of the sections with
$y_j(p)\in I_{l,q}$. We define
$\hat y_{l,q}:=\sqrt{y_jy_{\bar j}}$. Clearly $\hat y_{l,q}$ is real
on $\R\hat\cC$. To see that it is univalued on $\hat\cC$ note that
$y_j$ and $y_{\bar j}$ have the same associated monomial by symmetry,
hence $y_jy_{\bar j}$ has an even associated monomial and admits a
univalued square root.

With $\hat y_{l,q}$ chosen as above we continue the construction as in
the complex case. To verify that all the arguments remain valid we
need an analog of Lemma~\ref{lem:yj-v-yk} where either $y_j$ or $y_k$
(or both) are replaced by a cluster center
$\hat y_{l,q}:=\sqrt{y_hy_{\bar h}}$. This follows essentially from
the same lemma applied to $y_h$ and to $y_{\bar h}$, and we leave the
details for the reader. As a consequence we see that
Proposition~\ref{prop:clusters} and Remark~\ref{rem:single-cluster}
continue to hold, with the fibers $\cF_{i,q}$ and $\cF_{i,q+}$ now
real on $\R\hat\cC$.

\section{Cellular Weierstrass Preparation Theorem}
\label{sec:weierstrass}

In this section we state and prove a cellular analog of the
Weierstrass preparation theorem.

\begin{Def}\label{def:weierstrass-cell}
  Let $\gamma\in(0,1)$ and $\cC\odot\cF^\gamma$ be a cell and let
  $F\in\cO_b(\cC\odot\cF^\gamma)$. We say that $\cC\odot\cF^\gamma$ is
  a \emph{Weierstrass cell} with \emph{gap} $\gamma$ for $F$ if:
  \begin{itemize}
  \item $F$ vanishes identically on $\cC\odot\cF^\gamma$, or
  \item $F$ is non-vanishing on $\cC\odot*$ if $\cF=*$, or
  \item $F$ is non-vanishing on $\cC\odot(\cF^\gamma\setminus\cF)$ if
    $\cF=D,D_\circ,A$.
  \end{itemize}

  If $\hat\cC$ is a cell and $F\in\cO_b(\hat\cC)$ we say that a
  cellular map $f:\cC\odot\cF^\gamma\to\hat\cC$ is Weierstrass with
  gap $\gamma$ for $F$ if $\cC\odot\cF^\gamma$ is a Weierstrass cell
  with gap $\gamma$ for $f^*F$.
\end{Def}

In our applications it is convenient to use $\gamma$-extensions
$\cF^\gamma$ rather than the $\he{\cdot}$-extensions in the definition
of the gap of $\cF$, even when the base cell $\cC$ is considered with
a $\hrho$ extension $\cC^\hrho$.

\begin{Thm}[Cellular Weierstrass Preparation Theorem (WPT)]\label{thm:wpt}
  Let $\rho,\sigma>0$. Let $\cC^\hrho$ be a (real) cell and
  $F\in\cO_b(\cC^\hrho)$ a (real) function. Then there exist
  $N$ (real) Weierstrass maps
  $f_j:\cC_j^\hsigma\odot\cF_j^\gamma\to\cC^\hrho$ for $F$ with gap
  $\gamma<1$ such that $\cC\subset\cup_j f_j(\cC_j\odot\cF_j)$.

  If $\cC^\hrho,F$ vary in a definable family $\Lambda$ then one may take 
  $N=\poly_\Lambda(\rho,1/\sigma)$, $\gamma=\gamma_\Lambda<1$
and the maps $f_j$ from a
  single definable family. If $\cC^\hrho,F$ are algebraic of complexity
  $\beta$ then one may take $N=\poly_\ell(\beta,\rho,1/\sigma)$,
  $\gamma=1-1/\poly_\ell(\beta)$ and the maps $f_j$ algebraic of complexity 
  $\poly_\ell(\beta)$.
\end{Thm}

\begin{Rem}\label{rem:wpt-cover}
  The cellular WPT is analogous to the classical WPT in the following
  sense. Let $\cC\odot\cF^\gamma$ be a Weierstrass cell for
  $F\in\cO_b(\cC\odot\cF^\gamma)$ with $\cF$ of type $D,D_\circ,A$ and
  suppose that $F$ does not vanish identically. Then
  \begin{equation}
    \pi : \cC\odot\cF^\gamma \cap \{F=0\} \to \cC
  \end{equation}
  defines a proper ramified covering map: the condition that $F$ does
  not vanish in $\cC\odot(\cF^\gamma\setminus\cF)$ guarantees that the
  zeros of $F$ cannot leave $\cF$. The zeros of $F$ in
  $\cC\odot\cF^\gamma$ therefore agree with the zeros of a
  ``Weierstrass polynomial''
  \begin{equation}
    P(\vz,w) = \prod_{\eta\in\pi^{-1}(\vz)} (w-\eta) = w^\nu+\sum_{j=0}^{\nu-1} c_j(\vz) w^j
  \end{equation}
  with $c_j(\vz)\in\cO_b(\cC)$, where boundedness follows from
  boundedness of $\cF$ and holomorphicity follows by the removable
  singularity theorem.
\end{Rem}

A key difference between the classical Weierstrass preparation
theorem and the cellular analog is that the cellular version does
not require transforming the coordinates to general position.

\begin{Ex}
  Consider the unit polydisc
  $\cP_3\subset\cP_3^{1/10}\subset\C^3_{x,y,z}$ and the function
  $F=zx-y$. Then the classical Weierstrass preparation with respect to
  the $w$-direction is impossible. On the other hand a covering of
  $\cP_3$ by Weierstrass cells with gap $1/2$ for $F$ is given by the
  cells $*\odot*\odot D(1)$ where $F\equiv0$;
  $*\odot D_\circ(1)\odot D_\circ(1)$ and
  $D_\circ(1)\odot A(2x,4)\odot D(1)$ where $F$ has no zeros; and
  $D_\circ(1)\odot D(3x)\odot D(4)$ where $F$ has a single zero in
  each fiber.
\end{Ex}

We will prove the WPT and CPT by simultaneous induction on
$\ell:=\ell(\cC)$ and $\dim\cC$. The cases $\ell=0$ and $\dim\cC=0$
are trivial. We now prove the WPT for a cell $\cC\odot\cF$ of length
$\ell+1$ assuming that the CPT and WPT are true for cells of smaller
length or equal length and smaller dimension. The case $\cF=*$ reduces
to the CPT for $\cC$ so we assume $\cF$ is $D,D_\circ,A$. The proof of
the CPT is postponed to~\secref{sec:proof-cpt}.

We will give two separate proofs: one in the algebraic case, and one
in the analytic case. This is the only part of the proof of the main
theorems where our arguments significantly diverge for these two
cases. 

\subsection{Proof in the algebraic case}
\label{sec:weierstrass-alg-proof}

\subsubsection{Algebraic discriminants}
We need the following simple lemma on discriminants.

\begin{Lem}\label{lem:discriminant-alg}
  Let $\cC$ be a cell of length $\ell+1$ and $F\in\cO_b(\cC)$, both
  algebraic of complexity $\beta$. Suppose that $F$ does not vanish
  identically on $\cC$. Then there exist:
  \begin{itemize}
  \item A polynomial $P\in\C[\vz_{1..\ell+1}]$ of complexity $\poly_\ell(\beta)$, not
    identically vanishing on $\cC$, and satisfying
    $\{F=0\}\subset\{P=0\}$.
  \item A polynomial $\cD\in\C[\vz_{1..\ell}]$ of complexity
    $\poly_\ell(\beta)$, not identically vanishing on $\cC_{1..\ell}$,
    such that the projection 
    \begin{equation}
      \pi:(\cC_{1..\ell}\times\C)\cap\{P=0\}\to\cC_{1..\ell}, \qquad \pi(\vz_{1..\ell+1})=\vz_{1..\ell}
    \end{equation}
    is proper covering map outside $\{\cD=0\}$.
  \end{itemize}
  If $F$ is real then $P,\cD$ can be chosen real as well.
\end{Lem}
\begin{proof}
  We may assume without loss of generality that the type of $\cC$ does
  not contain $*$, since any such coordinate can be ignored for the
  statement of the lemma. Then $\cC\subset\C^{\ell+1}$ is open.

  By definition the graph of $F$ is contained in some irreducible
  algebraic variety $G_F\subset\C^{\ell+1}\times\C_w$. Since $F$ is
  not identically vanishing on $\cC$, the equation $w=0$ cuts $G_F$
  properly in a subvariety of dimension $\ell$, and the projection of
  this variety to $\C^{\ell+1}$ contains $\{F=0\}$ and is contained in
  a proper algebraic hypersurface of degree $\poly_\ell(\beta)$, i.e. in a
  set $\{P=0\}$ where $P\in\C[\vz_{1..\ell+1}]$ is not identically
  vanishing (on $\cC$, since $\cC$ is open). We may without loss of
  generality that $P$ is square-free as a polynomial in
  $\C(\vz_{1..\ell})[\vz_{\ell+1}]$, and in particular has no multiple
  roots for a generic value of $\vz_{1..\ell}$. Then the classical
  discriminant $\cD'$ of $P$ satisfies the conditions of the lemma,
  except that some zeros of $P$ may still escape to infinity outside
  $\cD'=0$. To eliminate this possibility we define $\cD$ to be the
  product of $\cD'$ with the leading coefficient of $P$ with respect to
  $\vz_{\ell+1}$.

  For the final statement, if $F$ is real then its graph is invariant
  under conjugation, and the same is then also true for $G_F$. The
  polynomial $P$ is then also real by construction.
\end{proof}

\subsubsection{Proof of the algebraic WPT}

We now proceed to the proof of the WPT. Suppose $\cF=A(r_1,r_2)$ (the
cases $D(r),D_\circ(r)$ are similar). Applying the refinement theorem
(Theorem~\ref{thm:cell-refinement}), we may assume that $\rho$ is
already as small as we wish as long as
$1/\rho=\poly_\ell(1/\sigma,\beta)$.

There is no harm in replacing $F$ by the polynomial $P$ obtained from
Lemma~\ref{lem:discriminant-alg} applied to the function
$F\cdot\vz_{\ell+1}\cdot(\vz_{\ell+1}-r_1)\cdot(\vz_{\ell+1}-r_2)$. In other words we
may assume without loss of generality that $F$ is a polynomial in the
$\vz_{\ell+1}$ variable and that it vanishes when $\vz_{\ell+1}$ is
either $0$, $r_1(\vz_{1..\ell})$ or $r_2(\vz_{1..\ell})$. We let $\cD$
denote the corresponding discriminant.

We apply the CPT to $\cC$ with $\cD$ and let
$f_j:\cC^\hrho_j\to\cC^\hrho$ denote the resulting cellular cover. If
$f_j(\cC_j^\hrho)$ is contained in $\{\cD=0\}$ then since cellular maps
preserve dimension $\dim\cC_j\le\dim\{\cD=0\}\le\dim\cC-2$. In this case
we set
\begin{align}
  \hat\cC_j&:=\cC_j\odot f_j^*\cF & \hat f_j&:=(f_j,\id):\hat\cC_j^\hrho\to\cC^\hrho
\end{align}
and inductively apply the WPT to $\hat\cC_j$ and $\hat f_j^*F$. We
obtain Weierstrass maps
$f_{j,k}:\cC^\hsigma_{j,k}\odot\cF^\gamma\to\hat\cC_j^\hrho$ for $F$,
and the compositions
$f_{j,k}\circ\hat f_j:\cC^\hsigma_{j,k}\odot\cF^\gamma\to\cC^\hrho$
are Weierstrass maps for $F$ which cover $f(\cC_j)\odot\cF$.

It remains to consider the case that $f_j(\cC_j^\hrho)$ is disjoint
from $\{\cD=0\}$. In the same way as before, it will suffice to prove
the WPT for $\hat\cC_j$ and $\hat f_j^*F$ where $\hat\cC_j$ and
$\hat f_j^* F$ are defined similarly. We return now to the original
notation replacing this pair by $\cC,F$.  We note that $F$ is still
polynomial in $\vz_{\ell+1}$ (though perhaps not in the other
coordinates) and the projection
\begin{equation}
  \pi:(\cC^\hrho\times\C)\cap\{F=0\}\to\cC^\hrho, \qquad \pi(\vz_{1..\ell+1})=\vz_{1..\ell}
\end{equation}
is a proper covering map. It will suffice to prove the WPT under these
conditions. We are now in a position to use the constructions
of~\secref{sec:cover-clustering}. Note that
$\nu=\poly_\ell(\beta)$. We will choose $\gamma=1-1/\poly_\ell(\beta)$
(the precise choice will be determined later).

Recall that we may choose $\rho$ small as long as
$1/\rho=\poly_\ell(1/\sigma,\beta)$. We take
$\hrho<\hsigma\cdot\he{\hat\rho}$ where $\hat\rho$ is chosen in such a
way that Proposition~\ref{prop:clusters} holds over $\cC^\hsigma$. The
zero map is a section of $\pi$ which we denote by $y_0$. Since
$r_1,r_2$ are sections of $\pi$ they belong to certain clusters around
$y_0=0$, say with indices $q_1\le q_2$. We may also assume that
$r_1=\hat y_{0,q_1}$ and if $q_1\neq q_2$ then $r_2=\hat y_{0,q_2}$.

If $q_1=q_2=q$ then we define $\tilde\cF=\cF_{0,q}$. If $q_1<q_2$ then we define
\begin{equation}\label{eq:wpt-final-fiber-alg}
  \tilde\cF:= A(l_{0,q_1}r_1,r_{0,q_2}r_2),
\end{equation}
i.e. we take the left endpoint of the $\cF_{0,q_1}$ cluster and the
right endpoint of the $\cF_{0,q_2}$ cluster. Since $r_1,r_2$ are
univalued over $\cC$ this is actually a fiber over $\cC$, which
contains $A(r_1,r_2)$ and satisfies the Weierstrass condition with gap
$\gamma$ by Proposition~\ref{prop:clusters}. Recall that we may
assume $\hrho<1/2$. Then the inclusion
$\tilde\cF^\gamma\subset\cF^\hrho$ follows
from~\eqref{eq:cluster-edge-bounds} if we choose $\gamma$ satisfying
$\gamma^{5\nu}<1/2$ where $\nu:=\deg\pi$. Since
$\nu=\poly_\ell(\beta)$ one can indeed choose $\gamma$ satisfying this
condition with $\gamma=1-1/\poly_\ell(\beta)$.

\subsubsection{The real setting}

If $\cC$ and $F$ are real then $P$ is also real, and consequently the
coverings constructed by inductive applications of the CPT can be
taken to be real. After these reductions, the fiber constructed
in~\eqref{eq:wpt-final-fiber-alg} is clearly real as well.

\subsection{Proof in the analytic case}

Before giving the proof of the WPT in the analytic case we develop
some general results concerning the Laurent coefficients of definable
families of holomorphic functions.

\subsubsection{Laurent domination in definable families}
\label{sec:laurent-domination}

We will study the following property of the Taylor/Laurent
coefficients of a holomorphic function.

\begin{Def}[Taylor domination]\label{def:domination}
  A holomorphic function $f:D(r)\to\C$ with $r>0$ is said to possess
  the $(p,M)$ Taylor domination property\footnote{note that we use a
    slightly simplified form of the definition given in
    \protect{\cite{by:domination}}.} \cite{by:domination} if its
  Taylor expansion $f(z)=\sum a_k(z-z_0)^k$ satisfies the estimate
  \begin{equation}
    |a_k| r^k < M \max_{j=0,\ldots,p} |a_j| r^j, \qquad k=p+1,p+2,\ldots 
  \end{equation}
  Similarly, for $r_2>r_1>0$ a holomorphic function $f:A(r_1,r_2)\to\C$
  is said to possess the $(p,M)$ Laurent domination property if its
  Laurent expansion $f(z)=\sum a_k(z-z_0)^k$ satisfies the estimates
  \begin{equation}
    \begin{aligned}
      |a_k| r_2^k &< M \max_{j=-p,\ldots,p} |a_j| r_2^j, \qquad
      k=p+1,p+2,\ldots \\
      |a_k| r_1^k &< M \max_{j=-p,\ldots,p} |a_j| r_1^j, \qquad
      k=-p-1,-p-2,\ldots
    \end{aligned}  
  \end{equation}
\end{Def}

We will need the following lemma on Laurent expansions.

\begin{Lem}\label{lem:fourier-bounds}
  Let $S:=S(1)$ and $\{f_\lambda:S^{\delta^2}\to\C\}_\lambda$ a
  definable family of holomorphic functions for some
  $0<\delta<1$. Then there exists $B$ such that for every $\lambda$ we
  have
  \begin{equation}
    \norm{f_\lambda}_{S^\delta} \le B\norm{f_\lambda}_S.
  \end{equation}
  If we write $f_\lambda(z)=\sum a_k(\lambda)z^k$ then there exists
  $p\in\N$ and $m>0$ such that for all $\lambda$ we have
  \begin{equation}
    |a_j(\lambda)| > m \norm{f_\lambda}_S\qquad \text{for some } j=j(\lambda)\in\{-p,\ldots,p\}.
  \end{equation}
\end{Lem}
\begin{proof}
  By Lemma~\ref{lem:voorhoeve-estimate}, the Voorhoeve index of
  $f_\lambda$ along any circle in $S^{\delta^2}$ is uniformly bounded
  over $\lambda$ and over the circle. Fix $\lambda$ and assume without
  loss of generality (up to rotation) that the maximum of $f_\lambda$
  on $S^\delta$ is attained at $\delta i$ or $\delta^{-1} i$. Let $C$
  be the circle with diameter $[\delta i,\delta^{-1}i]$ and fix some
  disc $D$ with $C\subset D$ and $\bar D\subset S^{\delta^2}$. The
  Voorhoeve index of $f_\lambda$ over $\partial D$ is uniformly
  bounded over $\lambda$, and it follows from
  \cite[Theorem~3]{ky:rolle}\footnote{The result there is stated for
    $K$ with nonempty interior, but actually holds for $K$ of the form
    $S\cap D$, see \protect{\cite[Section~4.1]{ky:rolle}}.}  that
  $B_{S\cap D,D}(f_\lambda)$ is uniformly bounded by some constant
  $B_1$, where
  \begin{equation}
    B_{K,U}(f) = \log\max_{z\in\bar U}|f(z)| - \log\max_{z\in K}|f(z)|.
  \end{equation}
  Thus
  \begin{equation}
    \norm{f_\lambda}_{S^\delta} = \max_{z\in\bar D} |f(z)| \le e^{B_1} \norm{f_\lambda}_S
  \end{equation}
  proving the first part with $B=e^{B_1}$.

  From the first part and the Cauchy estimate it follows that
  \begin{equation}
    |a_k(\lambda)| \le B \delta^{|k|} \norm{f_\lambda}_S.
  \end{equation}
  Then
  \begin{align}
    \norm{f_\lambda}_S &\le \sum_{k\in\Z} |a_k(\lambda)| =
    \sum_{k\in\{-p,\ldots,p\}} |a_k(\lambda)|+\sum_{k\not\in\{-p,\ldots,p\}} |a_k(\lambda)| \le \\
    &\sum_{k\in\{-p,\ldots,p\}} |a_k(\lambda)|+\frac{2B\delta^{p+1}}{1-\delta} \norm{f_\lambda}_S
  \end{align}
  and the result now follows with $p$ such that
  $\frac{2B\delta^{p+1}}{1-\delta}<1/2$ and $m=1/(4p+2)$.
\end{proof}

As a direct corollary of Lemma~\ref{lem:fourier-bounds} we obtain the
following result. We remark that a similar statement for families of
discs appeared (in a slightly different form and with a different
proof) in \cite{cpw:params}. The disc case also follows fairly
directly from a classical theorem of Biernacki
\cite{biernacki}. However the annulus case does not seem to follow in
a similar manner.

\begin{Cor}\label{cor:uniform-indices}
  Let $\cF=\{\cF_\lambda\}$ be a definable family of discs or annuli and let
  $f_\lambda:\cF_\lambda^\e\to\C$ be a definable family of
  holomorphic functions, for some $0<\e<1$. Then the functions
  $f_\lambda$ have the $(p,M)$ Laurent domination property in $\cF_\lambda$
  for some uniformly bounded $p,M$.
\end{Cor}
\begin{proof}
  We prove the annuli case, the disc case being simpler. Since the
  claim is invariant under rescaling, we may assume without loss of
  generality that $\cF_\lambda$ is of the form $A(r_1,1)$. We need to
  prove
  \begin{equation}
    |a_k| < M \max_{j=-p,\ldots,p} |a_j| , \qquad k=p+1,p+2,\ldots
  \end{equation}
  And indeed by the Cauchy estimates and
  Lemma~\ref{lem:fourier-bounds} with $\delta^2=\e$ we have
  \begin{equation}
    |a_k(\lambda)| \le B \norm{f_\lambda}_s \delta^{k} \le \frac B m |a_j(\lambda)| \qquad \text{for
      some } j\in\{-p,\ldots,p\}.
  \end{equation}
  For $k=-p-1,\ldots$ we proceed similarly, assuming now that
  $\cF_\lambda$ is of the form $A(1,r_2)$.
\end{proof}

We record a simple consequence of the Taylor domination property.

\begin{Prop}\label{prop:residue-domination}
  Let $A:=A(r_1,r_2)$ with $0<r_1<r_2$ and $0<\delta<1$. Let
  $f:A^\delta\to\C$ be a holomorphic function with $(p,M)$ Taylor
  domination. Write
  \begin{equation}
    f(z)=\sum_{j=-p}^p a_j z^j + R_p(z).
  \end{equation}
  Then for $z\in A$ we have
  \begin{equation}\label{eq:residue-ineq}
    |R_p(z)| < \frac{2\delta}{1-\delta} M \max_{j=-p,\ldots,p} |a_j z^j|.
  \end{equation}
\end{Prop}
\begin{proof}
  The Taylor domination property in $A^\delta$ gives
  \begin{equation}
    \begin{aligned}
      |a_k| r_2^k &< \delta^{k-p} M \max_{j=-p,\ldots,p} |a_j| r_2^j , \qquad
      k=p+1,p+2,\ldots \\
      |a_k| r_1^k  &< \delta^{-k-p} M \max_{j=-p,\ldots,p} |a_j| r_1^j, \qquad
      k=-p-1,-p-2,\ldots
    \end{aligned}  
  \end{equation}
  and the same clearly holds with $r_1,r_2$ replaced by $z\in A$.
  Summing over $k=p+1,\ldots,\infty$ and $k=-p-1,\ldots,-\infty$ we
  obtain~\eqref{eq:residue-ineq}.
\end{proof}

\subsubsection{Proof of the analytic WPT}

We may assume without loss of generality that $\cC\odot\cF$ admits an
extension slightly wider than $\hrho$. Then $(\cC\odot\cF)^\hrho$ is
subanalytic, and in particular the family of all discs or annuli of
the form $\{\vz_{1..\ell}\}\times\tilde\cF$ with
$\{\vz_{1..\ell}\}\times\tilde\cF^{1/2}\subset(\cC\odot\cF)^\hrho$ is
definable. Applying Lemma~\ref{cor:uniform-indices} we find $p,M$
such that $F$ possesses the $(p,M)$ Laurent domination property on
every fiber $\{\vz_{1..\ell}\}\times\tilde\cF$ as above.

Applying the refinement theorem (Theorem~\ref{thm:cell-refinement}),
we may assume that $\rho$ is already as small as we wish as long as
$1/\rho=\poly_\ell(1/\sigma)$. Since the maps constructed in the
refinement theorem are cellular translate maps, every disc/annulus
$\{\vz_{1..\ell}\}\times\tilde\cF^{1/2}\subset(\cC\odot\cF)^\hrho$ in
a refined cell maps to a disc/annulus satisfying the same requirement
in our original cell $\cC$. Thus we may assume that after refinement
our function $F$ still possesses the $(p,M)$ Laurent domination
property on every fiber $\{\vz_{1..\ell}\}\times\tilde\cF$ as
above. Note that crucially $(p,M)$ does not depend on our choice of
$\hrho$.

Suppose $\cF=A(r_1,r_2)$ (the cases $D(r),D_\circ(r)$ are similar).
If $F$ vanishes identically on $\cC\odot\cF$ then there is nothing to
prove, so suppose otherwise. Rescaling $\cF$ by $r_2$ we may assume
that $r_2=1$ and $|r_1|<1$. We write a Laurent expansion
\begin{equation}\label{eq:F-laurent}
  F(\vz_{1..\ell+1}) = \sum_{k=-\infty}^\infty a_k(\vz_{1..\ell}) \vz_{\ell+1}^k
\end{equation}
where $a_k\in\cO_b(\cC^\hrho)$. We apply the CPT to the collection
$a_{-p},\ldots,a_p$ on $\cC^\hrho$. In the same way as
in~\secref{sec:weierstrass-alg-proof} we may reduce to the case where
every $a_k$ is either identically zero or non-vanishing on
$\cC^\hrho$. Note also that since this step only reparametrizes the
base without affecting the fiber, the $(p,M)$ Laurent domination
property still holds uniformly for every $\vz_{1..\ell}\in\cC^\hrho$.

Let $\Pi\subset\{-p,\ldots,p\}$ denote the indices $k$ such that
$a_k\neq0$. For $(j,k)\in\Pi^2,j\neq k$, we define
\begin{equation}
  r_{jk}(\vz_{1..\ell}) = \sqrt[k-j]{\frac{a_j(\vz_{1..\ell})}{a_k(\vz_{1..\ell})}},
\end{equation}
that is, $S(|r_{jk}|)$ is the circle in $\vz_{\ell+1}$ where the
$j$-th and $k$-th terms of~\eqref{eq:F-laurent} are of the same
modulus. Note that $r_{jk}$ is multivalued. We let $\Sigma$ denote the
set consisting of $1,2$ and the pairs $(j,k)\in\Pi^2$, $j\not=k$, and
set $\mu:=2+(2p+1)p\ge\#\Sigma$.  

We show that the zeros of $F$ can only occur in concentric annuli of
bounded width around the radii $r_{jk}$. It is more convenient to
state this in the logarithmic scale. We set $s_\alpha:=\log|r_\alpha|$
for every $\alpha\in\Sigma$ and note that $s_\alpha$ is single valued
on $\cC_{1..\ell}$.

\begin{Lem}\label{lem:norootsfaraway}  
  There exists a constant $B=B(M,p)>0$ such that for any
  $\vz_{1..\ell+1}\in\cC$ satisfying
  \begin{equation}\label{eq:noroots-condition}
    |\log |\vz_{\ell+1}|-s_\alpha(\vz_{1..\ell})| > B \qquad\text{for }\alpha\in\Sigma.
  \end{equation}
  we have $F(\vz_{1..\ell+1})\neq0$.
\end{Lem}
\begin{proof}
  Assume that~\eqref{eq:noroots-condition} holds for some unspecified
  constant $B$ which will be chosen later. Set
  $r:=|\vz_{\ell+1}|$. For $(j,k)\in\Pi^2,j\neq k$, we have
  \begin{equation}
    \left\vert\log \frac{|a_j \vz_{\ell+1}^j|}{|a_k 
        \vz_{\ell+1}^k|}\right\vert=
    \left|(j-k)[-\log(|r_{jk}|+\log r]\right| > B.
  \end{equation}
  In particular there exists one term $j_0\in\Pi$ such that
  $|a_{j_0} \vz_{\ell+1}^{j_0}|$ is maximal and we have
  \begin{equation}
    |a_{k} \vz_{\ell+1}^{k}|< e^{-B}|a_{j_0} \vz_{\ell+1}^{j_0}|,
    \qquad k\in\{-p,\ldots,p\}\setminus\{j_0\}.
  \end{equation}
  Set
  \begin{equation}
    R_p(\vz_{1..\ell+1}) := F(\vz_{1..\ell+1})-\sum_{k=-p}^p a_j(\vz_{1..\ell})\vz_{\ell+1}^k.
  \end{equation}
  Since $1,2\in\Sigma$ we have
  $e^B r_1(\vz_{1..\ell})<r<e^{-B}r_2(\vz_{1..\ell})$. Then
  Proposition~\ref{prop:residue-domination} implies (with
  $\delta=e^{-B}$) that
  \begin{equation}
    |R_p(\vz_{1..\ell})| < \frac{2M}{e^B-1}|a_{j_0} \vz_{\ell+1}^{j_0}|.
  \end{equation}
  In particular choosing $B$ such that
  $2pe^{-B}+\frac{2M}{e^B-1}<1$ we see that
  \begin{equation}
    |R_p(\vz_{1..\ell})|+\sum_{k\in\Pi\setminus\{j_0\}} |a_{k}
    \vz_{\ell+1}^{k}| < |a_{j_0} \vz_{\ell+1}^{j_0}| 
  \end{equation}
  so $F(\vz_{1..\ell+1})\neq0$ as claimed.
\end{proof}

Our goal will be to construct an annulus $\cA\subset\cF^\hrho$
concentric with $\cF$ which contains $\cF$ and such that
$\cA^{1/2}\setminus\cA$ remains at logarithmic distance at least $B$
from each of the $|r_\alpha|$, uniformly over $\cC^\hsigma$. In light
of Lemma~\ref{lem:norootsfaraway} this will be a Weierstrass cell with
gap $\gamma=1/2$ for $F$.

Apply the CPT to the cell $\cC^\hrho$ and the hypersurfaces
$\{r_\alpha=r_\beta\}$ for $(\alpha,\beta)\in\Sigma^2,\alpha\neq\beta$
(recall that $r_\alpha$ are multivalued so formally we raise both
sides to a power and clear denominators to obtain a bounded analytic
equation). Once again we may reduce to the case that each of these
equations is satisfied either identically or nowhere in $\cC^\hrho$.

We let $Z$ denote the union of the graphs over $\cC^\hrho$ of the zero
function, which we denote by $y_0$, and of all $r_\alpha$ for
$\alpha\in\Sigma$. By the condition above $\pi:Z\to\cC^\hrho$ is a
proper covering map. We are now in a position to use the constructions
of~\secref{sec:cover-clustering} with the value of $\gamma$ given by
$\tilde\gamma=\delta^2/2$ where $\delta:=e^{-B}$.

Recall that we may choose $\rho$ small as long as
$1/\rho=\poly_\ell(1/\sigma)$. We take
$\hrho<\hsigma\cdot\he{\hat\rho}$ where $\hat\rho$ is chosen in such a
way that Proposition~\ref{prop:clusters} holds over
$\cC^\hsigma$. Since $r_1,r_2$ are sections of $\pi$ they belong to
certain clusters around $y_0=0$, say with indices $q_1\le q_2$. We may
also assume that $r_1=\hat y_{0,q_1}$ and if $q_1\neq q_2$ then
$r_2=\hat y_{0,q_2}$.

If $q_1=q_2$ then we define $\tilde\cA=\cF_{0,q}^\delta$. If $q_1<q_2$ then we define
\begin{equation}\label{eq:wpt-final-fiber-analytic}
  \tilde\cA:= A^\delta(l_{0,q_1}r_1,r_{0,q_2}r_2),
\end{equation}
i.e. we take the left endpoint of the $\cF_{0,q_1}$ cluster and the
right endpoint of the $\cF_{0,q_2}$ cluster and apply
$\delta$-extension (recall $\delta=e^{-B}$). Since $r_1,r_2$ are
univalued over $\cC$ this is actually a fiber over
$\cC$. Proposition~\ref{prop:clusters} and our choice
$\tilde\gamma=\delta^2/2$ ensures that $\cA^{1/2}\setminus\cA$ remains
at logarithmic distance at least $B$ from each of the $|r_\alpha|$,
uniformly over $\cC^\hsigma$ as was our goal. The inclusion
$\tilde\cA^{1/2}\subset\cF^\hrho$ follows
from~\eqref{eq:cluster-edge-bounds} provided that we choose $\hrho$
small enough (note that $B$ and hence $\delta,\tilde\gamma$ did not
depend on our choice of $\hrho$).

\subsubsection{The real setting}

If $\cC$ and $F$ are real then the Laurent coefficients $a_j$ are also
real, and consequently the coverings constructed by inductive
applications of the CPT can be taken to be real. After these
reductions, the fiber constructed in~\eqref{eq:wpt-final-fiber-analytic}
is clearly real as well.

\subsubsection{Uniformity over families}

Uniformity of the number of cells over definable families follows
readily from the proof. Indeed the indices $p,M$ have already been
shown to be uniformly bounded over families. The fact that all
Weierstrass cells can be chosen from a single definable family follows
exactly as in the case of the refinement theorem
(Theorem~\ref{thm:cell-refinement}).

\subsubsection{Analytic discriminants}

Later we will also need the following lemma on analytic discriminants.

\begin{Lem}\label{lem:discriminant-analytic}
  Let $\cC^\hrho$ be a cell and let $Z\subset\cC^\hrho\times D(r)$ be
  an analytic subset, for some $r>0$, such that the projection
  $\pi:Z\to\cC^\hrho$ is proper. Then there exists
  $0\neq\cD\in\cO_b(\cC)$ such that $\pi$ is a covering map outside
  $\{\cD=0\}$. If $\cC,Z$ are real then $\cD$ can be chosen real as
  well. If $\cC^\hrho,Z$ vary in a definable family then $\cD$ can also be
  chosen to vary in a definable family.
\end{Lem}
\begin{proof}
  According to \cite[Theorem~III.B.21]{gr:analytic} the map $\pi$ is
  an \emph{analytic cover} and in particular it is a $\lambda$-sheeted
  covering map, for some $\lambda\in\N$, outside a set
  $A\subset\cC^\hrho$ which is \emph{negligible} in the sense of
  \cite[Definition~III.B.2]{gr:analytic}. Then letting
  \begin{equation}
    P_Z(\vz,w) := \prod_{\eta\in\pi^{-1}(\vz)} (w-\eta)
  \end{equation}
  we obtain a monic polynomial of degree $\lambda$ with holomorphic
  bounded coefficients in $\cC^\hrho\setminus A$, which by
  \cite[Definition~III.B.2]{gr:analytic} extend to holomorphic bounded coefficients
  in $\cC^\hrho$. Then $\cD$ can be taken to be the classical
  discriminant of $P_Z$ with respect to $w$. If $\cC,Z$ are real then
  $P_Z$ and hence $\cD$ are also real. It is clear that this
  construction can be made uniform over definable families.
\end{proof}

\section{Proofs of the CPT and CPrT}

In this section we finish the proof of the CPT and CPrT for a cell
$\cC\odot\cF$. We proceed by induction on the length and dimension. By
the note following Theorem~\ref{thm:wpt}, we may already assume that
the WPT holds for $\cC\odot\cF$.

\subsection{Proof of the CPT}
\label{sec:proof-cpt}

We will prove the CPT in a slightly weakened form, replacing the
prepared maps $f_j$ by arbitrary cellular maps. We later prove that
this weaker form implies the CPrT, which in turn directly implies the
stronger form of the CPT. We describe the proof for the complex
version of the CPT, and at the end indicate the changes required for
the real version. To avoid cluttering the text, we state our proof for
the algebraic version of the CPT with polynomial estimates in the
complexity $\beta$. The subanalytic version where these polynomial
estimates are replaced by uniformity over definable families is
obtained in a completely analogous manner.

\subsubsection{Reduction to the case of one function ($M=1$)}
\label{sec:cpt-single-F}

We claim that it is enough to prove the CPT for a single function
$F$. We suppose that this is already proved and prove the result for
an arbitrary collection $F_1,\ldots,F_M$. We may suppose that none of
the $F_j$ vanish identically on $\cC$. Let $F:=F_1\cdots F_M$ and
apply the CPT to $\cC,F$ to obtain a cellular cover
$f_j:\cC_j^\hsigma\to(\cC\odot\cF)^\hrho$ compatible with $F$. Fix
some $f_j$. If $f_j(\cC_j^\hsigma)$ lies outside the zeros of $F$ then
it is already compatible with $F_1,\ldots,F_M$.  Otherwise it lies in
the zeros of $F$, and since cellular maps preserve dimension we have
$\dim\cC_j<\dim\cC\odot\cF$. By induction we obtain cellular maps
$f_{jk}:\cC_{jk}^\hsigma\to\cC_j^\hsigma$ which are compatible with
$f_j^*F_1,\ldots,f_j^*F_M$. Then the compositions $f_j\circ f_{jk}$
are compatible with $F_1,\ldots,F_M$ and cover $f_j(\cC_j)$, and taken
together this gives a cellular cover of $\cC\odot\cF$ with
$\poly_\ell(\beta,1/\sigma,\rho)$ maps as required.

\subsubsection{Reduction to large $\sigma$ and small $\rho$}
\label{sec:cpt-refine-sigma-rho}

Applying the refinement theorem (Theorem~\ref{thm:cell-refinement}) to
$\cC$ we may suppose that it already admits a
$\he{\hat\rho}$-extension for
$\hat\rho^{-1}=\poly_\ell(\rho,\beta)$. Below we will assume that
$\rho$ is already as small as we wish subject to this
asymptotic. Similarly, it is enough to prove the CPT with any
$\hat\sigma=\poly_\ell(\beta)$ since we may afterwards apply the
refinement theorem to the resulting cells to obtain cells with
$\sigma$-extensions. Below we will assume that $\sigma$ is already as
large as we wish subject to this asymptotic.

\subsubsection{Reduction to a Weierstrass cell for $F$}

We apply the WPT to cover $\cC\odot\cF$ by cells of the form
$f_j:\cC_j^\hrho\odot\cF_j^\gamma\to(\cC\odot\cF)^\hrho$ which are Weierstrass
for $F$ with gap $\gamma=1-1/\poly_\ell(\beta)$. It is enough to prove
the CPT for each of these cells separately, i.e. we may assume that
$F\in\cO_b(\cC^\hrho\odot\cF^\gamma)$ does not vanish in
$\cC^\hrho\odot(\cF^\gamma\setminus\cF)$. If $\cF$ is of type $*$ the
CPT reduces to the CPT for $\cC_{1..\ell-1}$. Suppose $\cF=A(r_1,r_2)$
(the cases $D(r),D_\circ(r)$ are similar).

\subsubsection{Reduction to a proper covering map}

Since $\cC^\hrho\odot\cF$ is a Weierstrass cell for $F$, the zero
locus of $F$ in this cell is proper ramified cover of $\cC$ under the
projection $\pi_{1..\ell}:\cC^\hrho\times\cF^\gamma\to\cC^\hrho$ (see
Remark~\ref{rem:wpt-cover}). In the analytic case, by
Lemma~\ref{lem:discriminant-analytic} there exists
$\cD\in\cO_b(\cC^\hrho)$ not identically vanishing on $\cC$ such that
the restriction of $\pi_{1..\ell}$ to the zero locus of $F$ is a
proper covering map outside $\{\cD=0\}$. In the algebraic case, by
Lemma~\ref{lem:discriminant-alg} we see that such $\cD$ can be taken
to be a polynomial of complexity $\poly_\ell(\beta)$.

We apply the CPT to $\cC$ and the function $\cD$ and let
$f_j:\cC^\hrho_j\to\cC^\hrho$ denote the resulting cellular cover. If
$f_j(\cC_j^\hrho)$ is contained in $\{\cD=0\}$ then since cellular
maps preserve dimension $\dim\cC_j\le\dim\{\cD=0\}\le\dim\cC-2$. In
this case we set
\begin{align}
  \hat\cC_j&:=\cC_j\odot f_j^*\cF & \hat f_j&:=(f_j,\id):\hat\cC_j^\hrho\to\cC^\hrho
\end{align}
and note that $f_j^*F\in\cO_b(\hat\cC_j^{\he{\hat\rho}})$ for
$\hat\rho=\poly_\ell(\rho,\beta)$. We inductively apply the CPT to
$\hat\cC_j$ with its $\hat\rho$-extension and $\hat f_j^*F$. We obtain
a cellular cover
$f_{j,k}:\cC^\hsigma_{j,k}\to\hat\cC_j^{\he{\hat\rho}}$ compatible
with $F$, and the compositions
$f_{j,k}\circ\hat f_j:\cC^\hsigma_{j,k}\to\cC^\hrho$ are compatible
with $F$ and cover $f(\cC_j)\odot\cF$.

It remains to consider the case that $f_j(\cC_j^\hrho)$ is disjoint
from $\{\cD=0\}$. In the same way as before, it will suffice to prove
the CPT for $\hat\cC_j$ and $\hat f_j^*F$. We return now to the
original notation replacing this pair by $\cC,F$ where we may now
assume that the discriminant locus is empty. We denote by $Z$ the
union of $(\cC^\hrho\odot\cF^\gamma)\cap\{F=0\}$ and the graphs of
$r_1,r_2$ and the zero function over $\cC^\hrho$. Then $\pi:=(\pi_{1..\ell})\rest Z$ is
a proper unramified covering map: for the zero locus of $F$ this is
true by non-vanishing of the discriminant; for the remaining three
graphs we only need to verify that
$0,r_1(\vz_{1..\ell}),r_2(\vz_{1..\ell})$ are pairwise disjoint and
mutually disjoint from the zeros of $F$ for every
$\vz_{1..\ell}\in\cC^\hrho$. This follows from the Weierstrass
condition since $r_1,r_2$ lie in $\cF^\gamma\setminus\cF$ (and $0$ is
never in $\cF$).

We are now in a position to use the constructions
of~\secref{sec:cover-clustering}. We denote the degree of $\pi$ by
$\nu$, and note that $\nu=\poly_\ell(\beta)$.

\subsubsection{Covering the zeros}
\label{sec:cpt-zero-cover}

By Lemma~\ref{lem:y_j-monodromy}, each section $y_j$ of $\pi$ lifts to
a univalued map $y_j:\hat\cC_j^{\nu_j\cdot\rho}\to\C$. For every
$y_j\neq r_1,r_2,0$ the image of this map lies in
$(\cC^\hrho\odot\cF^\gamma)\cap\{F=0\}$, i.e. it gives a cell
compatible with $F$. The collections of all of these cells cover the
zeros of $F$ in $\cC\odot\cF$. Note that we may also include the
analogous maps for $r_1,r_2$ as they never meet the zeros of $F$, but
this is not essential as they do not belong to $\cF$. On the other
hand we may not include the $y_0=0$ map as its image lies outside
$\cF^\gamma$.

The remaining (and far more non-trivial) task is to cover the
complement of this set of zeros.

\subsubsection{The Voronoi cells associated with $y_0=0$}
\label{sec:voronoi-cells-y0}

For the remainder of the proof we fix a point $p\in\hat\cC$.  We may
assume that $\rho$ is small enough that
Proposition~\ref{prop:clusters} holds, not only over $\cC$ but over
$\cC^\hsigma$. The zero map is a section of $\pi$ which we denote by
$y_0$. We will construct a collection of cells which we call the
\emph{Voronoi cells} of $y_0$.

Since $r_1,r_2$ are sections of $\pi$ they belong to certain clusters
around $y_0=0$, say with indices $q_1\le q_2$. For
$q=q_1,\ldots,q_2-1$ the fibers $\cF^\gamma_{0,q+}$ over
$\hat\cC_{0,q}$ contain none of the zeros $y_j$ and are contained in
$\cF^\gamma$. Therefore the maps
$\hat\cC_{0,q+}\odot\cF_{0,q+}\to\cC\odot\cF$ admit $\gamma$-extensions
compatible with $F$. These give our first Voronoi cells.

Now fix $q\in\{q_1,q_1+1,\ldots,q_2\}$ and consider the cluster
$\cF_{0,q}$. Up to an affine transformation as in
Remark~\ref{rem:single-cluster} we may assume that $\hat y_{0,q}=1$
and $\cF_{0,q}=A^\delta(l_{0,q},r_{0,q})$. Recall that
$\gamma=1-1/\poly_\ell(\beta)$, and therefore we have
$\cF_{0,q}\subset A(1/2,2)$. Recall also that
\begin{equation}\label{eq:0-cluster-y_j}
  \diam(y_j(\hat\cC),\C\setminus\{0,1\})=O_\ell(\nu^3\rho)
\end{equation}
for any of the $y_j$ belonging to $\cF_{0,q}$.

Let $\alpha>0$ be such that the discs of radius $\alpha$ centered at
$y_j(p)$ belongs to $\cF_{0,q}^\gamma(p)$. Clearly we can choose
$\alpha^{-1}=\poly_\ell(\beta)$. Let $U_p(\alpha)$ be the set obtained
from $\cF_{0,q}(p)$ by removing all of these discs. We can cover
$U_p(\alpha)$ be $\poly_\ell(\beta)$ discs $D_k$ centered at
$U_p(\alpha)$ such that $D_k^{1/2}$ does not meet the discs of
radius $\alpha/2$ centered at $y_j(p)$. These give our remaining Voronoi
cells. We claim that $\hat\cC_{0,q}\odot D^{1/2}_k$ is compatible
with $F$. Indeed, over $p$ the discs $D_k$ remain at distance
$\alpha/2$ from the zeros, and by~\eqref{eq:0-cluster-y_j} the points
$y_j$ do not move enough to meet $D_i$ for $\rho$ sufficiently
small. By the same reasoning we obtain the following fundamental
property of the Voronoi cells.

\begin{Lem}\label{lem:voronoi-y0}
  The Voronoi cells associated with $y_0$ cover every point
  $\vz_{1..\ell+1}\in\cC\odot\cF$ such that
  \begin{equation}
    \dist(\vz_{\ell+1},y_0) < (2\alpha)^{-1}\dist(\vz_{\ell+1},y_j)
  \end{equation}
  for every $y_j$.
\end{Lem}

\subsubsection{The Voronoi cells associated with $y_i$}

Let $y_i$ be one of the non-zero sections of $\pi$. We will construct
a collection of \emph{Voronoi cells} for $y_i$ by analogy with $y_0$:
the only additional difficulty is making sure that all cells remain in
$\cF^\gamma$. Our goal will be to construct cells with the following
property.

\begin{Lem}\label{lem:voronoi-yi}
  The Voronoi cells associated with $y_i$ cover every point
  $\vz_{1..\ell+1}\in\cC\odot\cF$ such that
  \begin{equation}\label{eq:voronoi-yi-1}
    \dist(\vz_{\ell+1},y_i) \le 2\alpha\dist(\vz_{\ell+1},y_0)
  \end{equation}
  and
  \begin{equation}\label{eq:voronoi-yi-2}
    \dist(\vz_{\ell+1},y_i) \le (2\alpha)^{-1}\dist(\vz_{\ell+1},y_j)
  \end{equation}
  for every $y_j\neq y_0,y_i$.
\end{Lem}

We make an affine transformation such that $y_i=0$. Since $y_i\in\cF$,
we know that in these coordinates $D(r y_0)\subset\cF^\gamma$ for
$r\sim1-\gamma=1/\poly_\ell(\beta)$. In the notations
of~\secref{sec:cover-clustering}, let $\cF_{i,q_0}$ be the last
cluster whose outer boundary at $p$ is smaller than $ry_0(p)$. With
our choice of $\gamma$ the logarithmic width of each cluster can be
assumed to be bounded by $\log 2$. Therefore the inner boundary of
$\cF_{i,q_0+1}$ at $p$ is at least $ry_0(p)/2$. If the outer boundary
of $\cF_{i,q_0}$ is smaller than $\gamma^2ry_0(p)/2$ then we also set
\begin{equation}
  \tilde\cF_{i,q_0+}=A^\delta(r_{i,q_0}\hat y_{i,q_0},\gamma r y_0/2).
\end{equation}
See Figure~\ref{fig:cluster-near-boundary} for an illustration of this
construction. One can check in the same way as in
Proposition~\ref{prop:clusters} that when defined,
$\tilde\cF_{i,q_0+}$ is a well defined annulus over $\hat\cC_{i,q_0+}$
and $\tilde\cF_{i,q_0+}^\gamma\subset\cF^\gamma$ contains no zeros.

\begin{figure}
  \centering
  \includegraphics[width=0.6\textwidth]{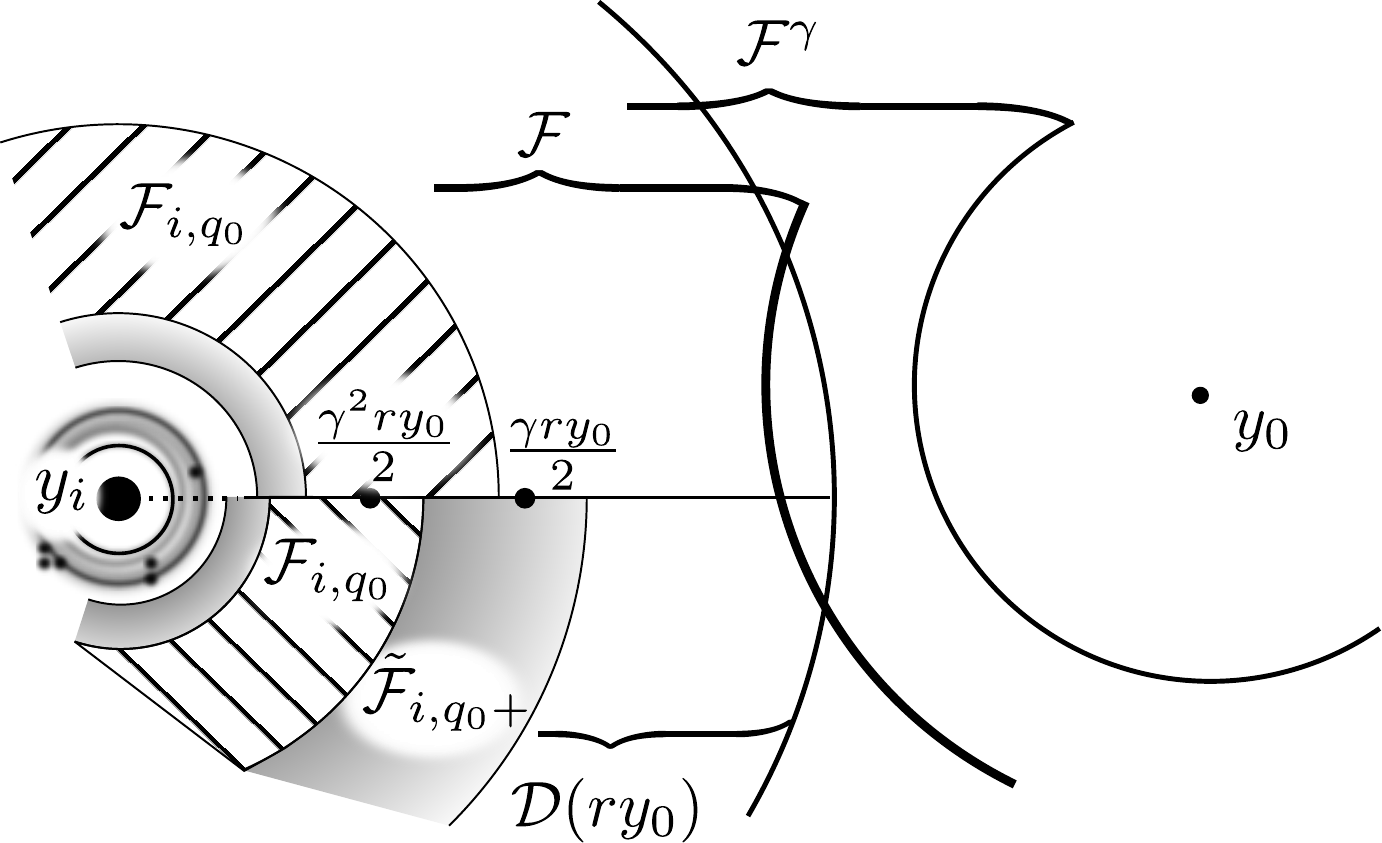}
  \caption{Construction of the $\tilde\cF_{i,q_0+}$ fiber near the
    boundary.}
  \label{fig:cluster-near-boundary}
\end{figure}

The fibers $\cF_{i,q},\cF_{i,q+}$ for $q=1,\ldots,q_0-1$, the fiber
$\cF_{i,q_0}$, and $\tilde\cF_{i,q_0+}$ if it is defined cover
$D(\gamma^2 r y_0/2)$ uniformly over $\hat\cC$. We can choose our
$\alpha=1/\poly_\ell(\beta)$ in such a way that this disc contains all
points satisfying~\eqref{eq:voronoi-yi-1}. It will be enough to
construct the Voronoi cells covering the points in these fibers which
also satisfy~\eqref{eq:voronoi-yi-2}. This is done in a manner
completely analogous to the Voronoi cells of $y_0$, except that in
place of $\cF_{0,q_1+},\ldots,\cF_{0,(q_2-1)+}$ we use
$\cF_{i,0+},\ldots,\cF_{i,(q_0-1)+}$ and $\tilde\cF_{i,q_0+}$ if it is
defined; and instead of $\cF_{0,q_1},\ldots,\cF_{0,q_2}$ we use
$\cF_{i,1},\ldots,\cF_{i,q_0-1}$.

\subsubsection{Conclusion}

We claim that the Voronoi cells for $y_0$ and the other sections
$y_j$ cover $\cC\odot\cF\setminus\{F=0\}$. Indeed, let
$\vz\in\cC\odot\cF$. If
$\dist(\vz_{\ell+1},y_0) < (2\alpha)^{-1}\dist(\vz_{\ell+1},y_j)$ for
every $y_j$ then $\vz$ is covered by the Voronoi cells of
$y_0$. Otherwise it is covered by the Voronoi cell of $y_j$ for the
section $y_j$ closest to $\vz_{\ell+1}$. Since the roots of $\{F=0\}$
were already covered in~\secref{sec:cpt-zero-cover}, this concludes
the proof of the CPT in the complex case.

\subsubsection{The real case of the CPT}

The proof of the CPT in the real case proceeds in the same manner up
to~\secref{sec:cpt-zero-cover}. At this point one should cover only
the sections $y_j$ that are real, ensuring that we indeed obtain real
cells. Since the sections $y_j$ are locally unramified, each of them
is either purely real or purely non-real on $\R_+\cC$ so we can indeed
choose those $y_j$ that are real to obtain a covering of the zeros of
$F$ in $\R_+(\cC\odot\cF)$.

The next step is the construction of the Voronoi cells. To ensure that
the constructed cells are real, one should replace the construction
of~\secref{sec:cover-clustering} by that
of~\secref{sec:cover-clustering-real} as we explain below.

The section $y_0=0$ is real and therefore forms an admissible
center. The Voronoi cells for $y_0$ are constructed as in the complex
case~\secref{sec:voronoi-cells-y0}, except that the discs $D_k$ are
now chosen with real centers such that their positive parts cover
$\R_+U_p(\alpha)$. The points $\vz_{\ell+1}$ that are left uncovered
over $p$ are those that have distance at most $\alpha$ to some section
$y_k$. If $\vz_{\ell+1}\in\R$ then by
Lemma~\ref{lem:select-admissible} such points also have distance at
most $\alpha$ to some admissible center $c_l=(y_j+y_{\bar j})/2$. By
analogy with Lemma~\ref{lem:voronoi-y0} we obtain the following lemma.

\begin{Lem}\label{lem:voronoi-y0-real}
  The Voronoi cells associated with $y_0$ cover every point
  $\vz_{1..\ell+1}\in\R_+(\cC\odot\cF)$ such that
  \begin{equation}
    \dist(\vz_{\ell+1},y_0) < (2\alpha)^{-1}\dist(\vz_{\ell+1},c_l)
  \end{equation}
  for every admissible center $c_l$.
\end{Lem}

For the remaining Voronoi cells, we construct them centered around the
admissible centers $c_l$ rather than the sections $y_j$. The
construction is analogous, replacing again the covering of
$U_p(\alpha)$ by a real covering of $\R_+U_p(\alpha)$. We similarly
obtain the following lemma.

\begin{Lem}\label{lem:voronoi-yi-real}
  The Voronoi cells associated with $c_l$ cover every point
  $\vz_{1..\ell+1}\in\R_+(\cC\odot\cF)$ such that
  \begin{equation}\label{eq:voronoi-yi-1-real}
    \dist(\vz_{\ell+1},c_l) \le 2\alpha\dist(\vz_{\ell+1},y_0)
  \end{equation}
  and
  \begin{equation}\label{eq:voronoi-yi-2-real}
    \dist(\vz_{\ell+1},c_l) \le (2\alpha)^{-1}\dist(\vz_{\ell+1},c_m)
  \end{equation}
  for every $c_m\neq y_0,y_l$.
\end{Lem}

The proof is then concluded in the same manner.

\subsection{Proof of the CPrT}
\label{sec:proof-cprt}

In this section we will prove the CPrT using the WPT and CPT. We will
need a simple remark on the structure of the maps constructed in these
two theorems.

\begin{Rem}\label{rem:wpt-cpt-translates}
  The maps constructed in the CPT can be assumed to be translates in
  the final variable, as the reader may easily verify by examining the
  inductive proof. Similarly, the maps constructed in the WPT can be
  assumed to be translates in the final two variables: the CPT and
  refinement theorem (Theorem~\ref{thm:cell-refinement}) are applied
  in the base to give a translate in the next-to-last variable,
  whereas in the last variable the map is the identity.
\end{Rem}

The main inductive step for the proof of the CPrT is contained in the
following lemma.

\begin{Lem}\label{lem:cprt-step}
  Let $f:\cC^\hrho\to\hat\cC$ be a (real) cellular map. Then there
  exists a (real) cellular cover $\{g_j:\cC_j^\hrho\to\cC^\hrho\}$ such that each $f\circ g_j$ is prepared
  \emph{in the last variable}.

  If $\cC^\hrho,\hat\cC,f$ vary in a definable family $\Lambda$ then the size of the
  cover is $\poly_\Lambda(\rho)$, and the maps $g_j$
  can be chosen from a single definable family. If $\cC^\hrho,\hat\cC,f$
  are algebraic of complexity $\beta$ then the cover has size
  $\poly_\ell(\beta,\rho)$ and complexity $\poly_\ell(\beta)$.
\end{Lem}
\begin{proof}
  We describe the proof for the complex case, and at the end indicate
  the changes required for the real version. Uniformity over families
  is obtained by a straightforward generalization of the argument.

  Let $\cC:=\cC_{1..\ell}\odot\cF$. Recall that the last coordinate of
  $f$ has the form $P(\vz_{1..\ell};\vz_{\ell+1})$ where $P$ is a
  monic polynomial in $\vz_{\ell+1}$ (say of degree $\mu$) with
  coefficients holomorphic in $\cC_{1..\ell}^\hrho$. If $\cF=*$ then
  we can just replace $\vz_{\ell+1}$ by zero in the expression above,
  so $f$ itself is already prepared. So assume $\cF$ is of type
  $D,D_\circ,A$.

  Let $\Sigma\subset\cC$ denote the set critical points of $P$ with respect to
  $\vz_{\ell+1}$, i.e.
  \begin{equation}
    \Sigma := \left\{ \pd{P}{\vz_{\ell+1}}(\vz_{1..\ell+1})=0 \right\}.
  \end{equation}
  Note that since $P$ is a monic polynomial of positive degree in
  $\vz_{\ell+1}$ the hypersurface $\Sigma$ has zero dimensional fibers
  over $\cC_{1..\ell}$.
  
  We apply the CPT to $\cC$ and $\Sigma$ to obtain a cellular cover
  $\{g_j:\cC_j^\hrho\to\cC^\hrho\}$ compatible with $\Sigma$. If the
  image of $g_j$ is contained in $\Sigma$ then, since the fibers of
  $\Sigma$ over $\cC_{1..\ell}$ are zero dimensional, $\cC_j$ has type
  ending with $*$ and in this case $f\circ g_j$ is already prepared as
  noted above. It remains to consider the case that the image of $g_j$
  is disjoint from $\Sigma$. Recall from
  Remark~\ref{rem:wpt-cpt-translates} that $g_j$ is a translate in its
  last variable $\vw_{\ell+1}$. Then we have
  \begin{equation}
    \pd{}{\vw_{\ell+1}}(P\circ g_j) = \pd{P}{\vz_{\ell+1}}(g_j(\vw_{1..\ell+1}))\neq0
  \end{equation}
  since the image of $g_j$ is disjoint from $\Sigma$. It will now be
  enough to prove the lemma with $f$ replaced by each $f\circ g_j$ as
  above. In other words we may assume without loss of generality that
  $\Sigma=\emptyset$.

  Recall that $P$ is bounded in absolute value by some constant $N$ on
  $\cC^\hrho$. Write $D:=D(N)$ and consider the cell
  \begin{equation}
    \tilde\cC:=\cC_{1..\ell}\odot D\odot\cF
  \end{equation}
  with coordinates $\vz_{1..\ell},w,\vz_{\ell+1}$, which also admits a
  $\hrho$-extension. Let $\Gamma\subset\tilde\cC^\hrho$ be the hypersurface
  given by
  \begin{equation}
    \Gamma := \{ w=P(\vz_{1..\ell};\vz_{\ell+1}) \}.
  \end{equation}
  We apply the WPT to $\tilde\cC^\hrho$ and $\Gamma$ to obtain a
  cellular cover
  $\{g_\alpha:\cC_\alpha^{\he{\rho/\mu}}\odot\cF_\alpha^\gamma\to\tilde\cC^\hrho\}$
  of $\tilde\cC$ by Weierstrass cells for $\Gamma$ (i.e. for the
  function $w-P(\vz_{1..\ell};\vz_{\ell+1})$).

  Suppose first that
  $g_\alpha(\cC_\alpha\odot\cF_\alpha)\subset\Gamma$. Then since
  $\Gamma$ has zero-dimensional fibers over $\cC_{1..\ell}\odot D$
  (because $P$ is monic) we have $\cF_\alpha=*$ and
  \begin{equation}\label{eq:psi-subord}
    \psi_\alpha(\vz_{1..\ell},w,*) := g_\alpha(\vz_{1..\ell},w,*) = (\cdots, w+\phi_\alpha(\vz_{1..\ell}),\zeta_\alpha(\vz_{1..\ell},w))
  \end{equation}
  where $\psi_\alpha:\cC_\alpha\odot*\to\Gamma$ is a cellular
  map admitting a $\he{\rho/\mu}$-extension. Here we used the fact that
  $g_\alpha$ is a translate in its next to last variable. For later
  purposes we set $\nu(\alpha)=1$.

  Suppose now that
  $g_\alpha(\cC_\alpha\odot\cF_\alpha)\not\subset\Gamma$. Then by the
  definition of a Weierstrass cell it follows that $g_\alpha^*\Gamma$
  forms a cover of $\cC_\alpha$. Since we are assuming
  that $\Sigma=0$, this is a covering map of degree
  $\nu(\alpha)\le\mu$. Then $g_\alpha^*\Gamma$ admits $\nu(\alpha)$
  multivalued sections. Each section
  $s_{\alpha,j}:\cC_\alpha\odot*\to g_\alpha^*\Gamma$ takes
  the form
  \begin{equation}
    s_{\alpha,j}(\vz_{1..\ell},w,*) = (\vz_{1..\ell},w,\hat\zeta_{\alpha,j}(\vz_{1..\ell},w))
  \end{equation}
  where $\hat\zeta$ is ramified of order at most
  $\nu(\alpha,j)\le\nu(\alpha)$ (see
  Lemma~\ref{lem:y_j-monodromy}). We note that in the algebraic case,
  $s_{\alpha,j}$ is algebraic of degree $\poly_\ell(\beta)$ since it
  is the section of a hypersurface of complexity
  $\poly_\ell(\beta)$.

  We set
  $\cC_{\alpha,j}\odot* := (\cC_\alpha\odot*)_{\times\nu(\alpha,j)}$.
  Pulling back $s_{\alpha,j}$ by
  $R_{\nu(\alpha,j)}:\cC_{\alpha,j}\odot*\to\cC_\alpha\odot*$ we
  obtain a univalued cellular map
  $s_{\alpha,j}\circ R_{\nu(\alpha,j)}:\cC_{\alpha,j}\odot*\to g_\alpha^*\Gamma$
  and finally composing with $g_\alpha$ we obtain a cellular map
  \begin{equation}
    \psi_{\alpha,j}:=g_\alpha\circ s_{\alpha,j}\circ R_{\nu(\alpha,j)}:\cC_{\alpha,j}\odot*\to\Gamma
  \end{equation}
  of the form
  \begin{equation}\label{eq:psi-cover}
    \psi_{\alpha,j}(\vz_{1..\ell},w,*) = (\cdots,w^{\nu(\alpha,j)}+\phi_{\alpha,j}(\vz_{1..\ell}),\zeta_{\alpha,j}(\vz_{1..\ell},w)).
  \end{equation}
  Here again we used the fact that $g_\alpha$ is a translate in its
  next to last variable. By Lemma~\ref{prop:ext-v-cover} the map
  $\psi_{\alpha,j}$ admits a $\hrho$-extension.

  Let $\beta$ be one of the indices $\alpha$ or $(\alpha,j)$ for the
  maps $\psi$ constructed
  in~\eqref{eq:psi-subord},\eqref{eq:psi-cover}. Set
  $f_\beta:=((\psi_\beta)_{1..\ell},\zeta_\beta):\cC_\beta^\hrho\to\cC^\hrho$. Then
  $f\circ f_\beta$ is prepared in its final variable: indeed, since
  $\psi_\beta$ maps into $\Gamma$ we have
  \begin{equation}
    w^{\nu(\alpha)}+\phi_\beta(\vz_{1..\ell})=P((\psi_\beta)_{1..\ell}(\vz_{1..\ell}),\zeta_\beta(\vz_{1..\ell},w))
    = (f\circ f_\beta)_{\ell+1}
  \end{equation}
  for every $(\vz_{1..\ell},w)\in\cC_\beta^\hrho$. It remains to show
  that the union over $\beta$ of $f_\beta(\cC_\beta)$ covers
  $\cC$. But this is clear, since $\psi_\beta(\cC_\beta\odot*)$ covers
  $\tilde\cC\cap\Gamma$ by construction and the projection of
  $\tilde\cC\cap\Gamma$ to $\vz_{1..\ell+1}$ equals $\cC$ by our
  choice of $D$.

  We now consider the real case. We proceed in the same manner up to
  the construction of the sections $\psi_{\alpha,j}$. Note that some
  of these sections may not be real. However, since the critical locus
  $\Sigma$ is empty it follows that on $\R_+\cC_{\alpha,j}$ the
  section $\psi_{\alpha,j}$ is either always real or always
  non-real. Since we are only interested in covering $\R_+\Gamma$ we
  may replace the full set of sections $\psi_{\alpha,j}$ by those
  sections that are real on $\R_+\cC_{\alpha,j}$. With this
  modification we obtain a real cover of $\cC$ as required.
\end{proof}

We are now ready to finish the proof of the CPrT by induction: suppose
that the CPrT is already proved for cells of length $\ell$, and we
will prove it for a cell $\cC:=\cC_{1..\ell}\odot\cF$. By
Lemma~\ref{lem:cprt-step} we may assume that the map $f$ is already
prepared in its final variable. By the inductive hypothesis we can
find a cellular cover $g_j:\cC_j^\hrho\to\cC_{1..\ell}^\hrho$ such
that $f_{1..\ell}\circ g_j$ is prepared. Then
$\hat g_j:=g_j\odot\id:(\cC_j\odot(g_j^*\cF))^\hrho\to\cC^\hrho$ is a
cellular cover for $\cC$ with $f\circ\hat g_j$ is prepared as
required.

\section{Analysis on complex cells}

In this section we prove some basic estimates for holomorphic
functions in complex cells. We also introduce the notion of a
\emph{quadric} cell which allows for some finer estimates and is used
in the construction of smooth parametrizations
in~\secref{sec:smooth-param}.

We fix the notation for the remainder of this section. We let
$\cC=\cF_1\odot\cdots\odot\cF_\ell$ denote a complex cell. As a matter
of normalization, we assume that $\cF_j$ is of the form
$*,D(1),D_\circ(1)$ or $A(r_j,1)$. Every cell is isomorphic to a cell
of this form by an appropriate rescaling map. We say that such a cell
$\cC$ is \emph{normalized}. We fix $0<\delta<1$ and assume that $\cC$
admits a $\delta$-extension.

\subsection{Logarithmic derivatives}

The following proposition is a cellular analog of the classical fact
that a logarithmic derivative of a holomorphic function admits at most
first order poles.

\begin{Prop}\label{prop:log-derivative-bound}
  Let $0<\delta<1/8$ and let $f\in\cO_b(\cC^\delta)$ be
  non-vanishing. Then for $\vz\in\cC$ we have
  \begin{equation}
    \abs{\frac{1}{f}\frac{\partial f}{\partial \vz_i}} \le O_f(\abs{\vz_i}^{-1})
    \qquad \text{for } i=1,\ldots,\ell.
  \end{equation}
   If $\cC,f$ are algebraic of complexity $\beta$ then
  \begin{equation}
    \abs{\frac{1}{f}\frac{\partial f}{\partial \vz_i}}\le \poly_\ell(\beta)\cdot \abs{\vz_i}^{-1}
    \qquad \text{for } i=1,\ldots,\ell.
  \end{equation}
\end{Prop}
\begin{proof}
  By the monomialization lemma (Lemma~\ref{lem:monomial}) we have
  $\log f=\sum_k m_k\log \vz_k+\log U$ where $g=\log U$ is holomorphic
  and bounded in $\cC^{1/4}$. The derivative of the first term with
  respect $\vz_i$ is bounded by $|m_i\vz_i|^{-1}$, and in the
  algebraic case $|m_i|=\poly_\ell(\beta)$. It remains to give a
  similar estimate for $\pd{g}{\vz_i}$. Since we only care about the
  derivatives we may, after a translate, assume that the image of $g$
  contains $0$.  Then by the monomialization lemma in the algebraic
  case $\norm{g}_{C^{1/4}}\le\poly_\ell(\beta)$.
  
  Write a decomposition $g=\sum_{\vsigma} g_\vsigma(\vz^{[\vsigma]})$
  as in Corollary~\ref{cor:laurent-disc-decmp},
  \begin{equation}
    g_\vsigma\in\cO_b(\cP^{1/2}), \qquad \norm{g_\vsigma}_{\cP^{1/2}} = O_\ell(\norm{g}_{\cC^{1/4}}).
  \end{equation}
  We estimate the derivative by  
  \begin{equation}
    \abs{\pd{g_\vsigma}{\vz_i}}\le
    \sum_{j=1}^\ell \abs{(g_\vsigma)'_j(\vz^{[\vsigma]}) \pd{(\vz_j^{[\vsigma_j]})}{\vz_i}}.
  \end{equation} 
  By the Cauchy formula for $g_\vsigma$ in $\cP^{1/2}$ we see that
  $\norm{(g_\vsigma)'_j}_\cP=O_\ell(\norm{g_\vsigma}_{\cP^{1/2}})$,
  which is $\poly_\ell(\beta)$ in the algebraic case.

  For $\vsigma_j=1$, the derivative
  $\pd{(\vz_j^{[\vsigma_j]})}{\vz_i}$ is $0$ if $i\neq j$ and $1$ if
  $i=j$. For $\vsigma_j=-1$, the derivative
  $\pd{(\vz_j^{[\vsigma_j]})}{\vz_i}$ is $0$ for $i>j$, $-r_i/\vz_i^2$
  for $i=j$ and $\vz_j^{-1}\pd{r_j}{\vz_i}$ for $i<j$. But
  \begin{align}
    \abs{-\frac{r_i}{\vz_i^2}}&=\abs{\frac{r_i}{\vz_i}\cdot\frac1{\vz_i}}\le|\vz_i|^{-1}, &
    \abs{\vz_j^{-1}\pd{r_j}{\vz_i}}&=\abs{\frac{r_j}{\vz_j}\cdot\frac1{r_j} \pd{r_j}{\vz_i}}\le K|\vz_i|^{-1}
  \end{align}
  where the final inequality follows by induction on $\ell$ and
  $K=\poly_\ell(\beta)$ in the algebraic case. From the above we
  conclude that $\abs{\pd{g}{\vz_i}}<K'|\vz_i|^{-1}$ for an appropriate
  constant $K'$ as claimed.
\end{proof}

\subsection{Quadric cells}
\label{sec:quadric-cells}

We say that the $\cC$ cell is \emph{quadric} if the radii $r_j$ (for
$\cF_j$ of annulus type) have a univalued square root, $r_j=\rho_j^2$
on $\cC_{1..j-1}$. This can also be stated by saying that the
associated monomial of $r_j$ has even degrees. We note that if $\cC$
is quadric then $\cC^\delta$ is quadric as well. We denote by
$\Qua\cC$ the \emph{positive quadrant} defined by replacing each fiber
$\cF_j=A(r_j,1)$ in the definition of $\cC$ by $A(\rho_j,1)$. We
denote by $\Qua^\delta\cC$ the cell obtained by taking
$\tilde\cF_j=A(\rho_j,\delta^{-1})$. Recall the notations of
Definition~\ref{def:nu-cover}.

\begin{Prop}\label{prop:quadric-nu-cover}
  Let $\cC$ be a cell admitting a $\delta$-extension.  Set
  $\vnu:=(2^{\ell-1},\ldots,1)$. Then $\cC_{\times\vnu}$ is quadric
  and admits a $\delta^{1/2^{\ell-1}}$-extension.
\end{Prop}
\begin{proof}
  If $(\alpha_1,\ldots,\alpha_{j-1})$ are the degrees of the
  associated monomial of $r_j$ then the associated monomial of the
  $j$-th fiber of $\cC_{\times\vnu}$ has degrees
  $2^{j-\ell}(2^{\ell-1}\alpha_1,\ldots,2^{\ell-j+1}\alpha_{j-1})$,
  all of which are even as claimed. The existence of a
  $\delta^{1/2^{\ell-1}}$-extension is a simple exercise.
\end{proof}

Suppose $\cC$ is a quadric cell and $\vsigma\in\{-1,1\}^\ell$, with
negative entries allowed only for indices $j$ such that $\cF^j$ is an
annulus. We define the inversion map $I_\vsigma:\cC\to\cC$ by the
identity on coordinates $\vz_j$ with $\vsigma=1$ and by $r_j/\vz_j$ on
coordinates $\vz_j$ with $\vsigma=-1$. Using inversions we can prove
the following.

\begin{Prop}\label{prop:quadric-cover}
  Let $\cC$ be a (real) cell admitting a $\delta$-extension. There
  exists a collection of (real) quadric normalized cells $\cC_j$ and
  (real) cellular maps $\{f_j:\cC_j^\delta\to\cC^\delta\}$ such that
  $f_j(\Qua\cC_j)$ covers $\cC$ (and $f_j(\R_+\Qua\cC_j)$ covers
  $\R_+\cC$). If $\cC$ is algebraic of complexity $\beta$ then
  $\cC_j,f_j$ are algebraic of complexity $\poly_\ell(\beta)$.
\end{Prop}
\begin{proof}
  By the refinement theorem (Theorem~\ref{thm:cell-refinement}) we may
  assume that $\cC$ admits a $\delta^{2^\ell}$-extension. Applying
  Proposition~\ref{prop:quadric-nu-cover} we see that
  $\cC_{\times\vnu}$ admits a $\delta^2$-extension, is quadric, and
  after rescaling may also be assumed to be normalized. It will
  suffice to prove the claim for $\cC_{\times\vnu}$, so henceforth we
  replace $\cC$ by $\cC_{\times\vnu}$.

  We essentially want to use the collection of all the inversion maps
  $I_\vsigma$ to cover $\cC$ by $I_\vsigma(\Qua\cC)$ and $\R_+\cC$ by
  $I_\vsigma(\R_+\Qua\cC)$. The minor technical issue is that this
  does not cover the equators $\{\vz_j=\rho_j\}$. To avoid this
  problem we use a slightly larger cell $\tilde\cC$. Namely, we
  replace each fiber $\cF_j=A(r_j,1)$ by
  $\tilde\cF_j=A(\delta r_j,1)$. Since $\cC$ admits a
  $\delta^2$-extension $\tilde\cC$ admits a $\delta$-extension. We can
  now cover $\cC$ by the union of $I_\vsigma(\Qua\tilde\cC)$ (and similarly
  for the real part).
\end{proof}

The following proposition is our principal motivation for introducing
the notion of a quadric cell.

\begin{Lem}\label{lem:QC-derivative-bound}
  Let $\delta<1/4$ and let $\cC^\delta$ be a quadric cell and
  $f\in\O_b(\cC^\delta)$. Then
  \begin{equation}
    \norm{\pd f{\vz_j}}_{\Qua\cC} \le O_\ell(\norm{f}_{\cC^\delta}\cdot\delta)
    \qquad \text{for }j=1,\ldots,\ell.
  \end{equation} 
\end{Lem}
\begin{proof}
  We assume without loss of generality that $\norm{f}_{\cC^\delta}=1$.
  By Corollary~\ref{cor:laurent-disc-decmp},
  $f=\sum_\vsigma f_\vsigma(\vz^{[\vsigma]})$ with
  $\norm{f_\vsigma}_{\cP^{2\delta}}=O_\ell(1)$. It will suffice to prove
  the claim for each of these summands, so fix
  $\vsigma\in\{-1,1\}^\ell$.

  If $\vsigma_j=1$ then
  \begin{equation}
    \pd{}{\vz_j} f_\vsigma(\vz^{[\vsigma]}) = (f_\vsigma)'_j(\vz^{[\vsigma]}) +
    \sum_{k>j,\vsigma_k=-1} (f_\vsigma)'_k(\vz^{[\vsigma]})\cdot \frac{1}{\vz_k}\cdot\pd{r_k}{\vz_j}.
  \end{equation}
  By the Cauchy formula
  \begin{equation}
    \norm{(f_\vsigma)'_j}_\cP=O_\ell(\delta \norm{f_\vsigma}_{\cP^{2\delta}})=O_\ell(\delta).
  \end{equation}
  Also, since $\rho_k=\sqrt{r_k}$ is holomorphic of norm at most $1$
  in $\cC_{1..k-1}^\delta$ we have
  \begin{equation}
    \norm{\frac{1}{\vz_k} \pd{r_k}{\vz_j}}_{\Qua\cC} = \norm{ 2 \pd{\sqrt{r_k}}{\vz_j} \frac{\sqrt{r_k}}{\vz_k}}_{\Qua\cC}\le
    O_\ell(\delta)
  \end{equation}
  where we used induction over $\ell$ and the fact that
  $\abs{\sqrt{r_k}/\vz_k}<1$ in $\Qua\cC$.  Combining these estimates we see
  that
  $\norm{\pd{}{\vz_j} f_\vsigma(\vz^{[\vsigma]})}_{\Qua\cC}=O_\ell(\delta)$.

  The case $\vsigma_j=-1$ is similar. We have
  \begin{equation}
    \pd{}{\vz_j} f_\vsigma(\vz^{[\vsigma]}) = -(f_\vsigma)'_j(\vz^{[\vsigma]})\cdot\frac{r_j}{\vz_j^2} +
    \sum_{k>j,\vsigma_k=-1} (f_\vsigma)'_k(\vz^{[\vsigma]})\cdot \frac{1}{\vz_k}\cdot\pd{r_k}{\vz_j},
  \end{equation}
  which can be estimated in the same manner noting that
  $\norm{\frac{r_j}{\vz_j^2}}_{\Qua\cC}\le 1$ by definition.
\end{proof}

\subsection{Straightening the positive quadrant}
\label{sec:quadrant-straighten}

Let $\cC$ be a normalized quadric cell of length $\ell$ admitting a
$\delta\le1/100$ extension. To simplify the notations we assume that
$\cC$ contains no fibers of type $D$ (we can replace each fiber of
type $D$ by fibers of type $D_\circ$ and $*$). To further simplify the
notation we also assume that $\cC$ contains no fibers of type $*$ (if
it does then one should add additional single coordinates to $B$
defined below in order to preserve the cellular structure of the map).

Let $B:=(0,1)^{\times\ell}$.  We define a map $h:B\to\R_+\cC^\delta$
inductively by
\begin{equation}\label{eq:h-def}
  h_i(\vu) = \vu_i+R_i(\vu_{1..i-1}), \qquad R_i:=\sqrt{r_i\circ h_{1..i-1}}.
\end{equation}
Here $r_i$ denotes the inner radius of $\cF_i$ if it is of type $A$,
and $0$ if it is of type $D_\circ$. Since $\cC$ is normalized it
follows that $\R_+\Qua\cC\subset h(B)\subset \cC^\delta$. We think of
$h$ as a straightening of the positive quadrant. Our goal will be to
show that $h$ can be analytically continued to a polysector
\begin{equation}\label{eq:polysector}
  \sec_\ell(\e) := \{|\Arg u_i|<\e, |u_i|<2 : i=1,\ldots,\ell\}
\end{equation}
for some positive $\e>0$. This will be used
in~\secref{sec:smooth-param} to construct parametrizations of
$\R_+\cC$ with control on derivatives.

\begin{Prop}\label{prop:h-extension}
  There exists $\e>0$ such that $h$ can be analytically extended to
  $\sec_\ell(\e)$ and we have
  \begin{equation}\label{eq:h-domains}
    h(\sec_\ell(\e))\subset\Qua^{1/4}\cC\cap\{|\Arg\vz_j|\le\pi/4:j=1,\ldots,\ell\}.
  \end{equation}
  Moreover for any $j$ such that $\cF_j$ is an annulus we have
  \begin{equation}\label{eq:R_j-C1}
    \norm{\frac{\vu_i}{R_j} \pd{R_j}{\vu_i}}_{\sec_\ell(\e)} \le M_\ell,\quad i=1,\ldots,\ell
  \end{equation}
  for some $M_\ell>0$, i.e. $\log R_j$ is a $C^1$-bounded function of
  $\log\vu$. If $\cC$ is algebraic of complexity $\beta$ then
  $\e^{-1},M_\ell=\poly_\ell(\beta)$.
\end{Prop}
\begin{proof}
  By induction on $\ell$ we may assume that $h_{1..\ell-1}$ extends to
  $\sec_{\ell-1}(\e)$ for a sufficiently small $\e$, that
  \begin{equation}\label{eq:h-domains-ind}
    h_{1..\ell-1}(\sec_{\ell-1}(\e))\subset\Qua^{1/4}\cC_{1..\ell-1}\cap\{|\Arg\vz_j|\le\pi/4:j=1,\ldots,\ell-1\},
  \end{equation}
  and that \eqref{eq:R_j-C1} already holds for $R_{1..\ell-1}$ with an
  appropriate constant $M_{\ell-1}$ (as a shorthand notation we write
  $R_{1..\ell-1}$ to mean only those $R_j$ which are defined,
  i.e. such that $\cF_j$ is an annulus). Our first goal is to
  prove~\eqref{eq:R_j-C1} for $R_\ell$ assuming $\cF_\ell$ is an
  annulus.

  Since $\log R_{1..\ell-1}$ is a $C^1$-bounded function of
  $\log\vu_{1..\ell-2}$ we may, taking $\e<\pi/(4M_{\ell-1})$,
  suppose that on $\sec_{\ell-1}(\e)$ we have
  \begin{equation}\label{eq:R_i-arg-bound}
    |\Arg \vu_i|, |\Arg R_i| \le \pi/4, \qquad i=1,\ldots,\ell-1.
  \end{equation}
  Indeed $R_i$ is real on $\R\sec_{i-1}(\e)$, and since every point of
  $\sec_{i-1}(\e)$ lies at logarithmic distance at most $\e$ from this
  real part, it follows that $\log R_i$ and in particular $\Arg R_i$
  is at distance at most $M_{\ell-1}\cdot\e$ from zero. We note that
  this implies
  \begin{equation}\label{eq:u_i-vs-z_i}
    |\vu_i|,|R_i| \le |\vu_i+R_i|=|\vz_i|, \qquad i=1,\ldots,\ell-1.
  \end{equation}
  Now computing the left hand side of~\eqref{eq:R_j-C1} with $j=\ell$
  we have
  \begin{equation}\label{eq:R_ell-der}
    \frac{\vu_i}{R_\ell} \pd{R_\ell}{\vu_i} =
    \vu_i \left(\frac1{\sqrt{r_\ell}}\pd{\sqrt r_\ell}{\vz_i}\circ h_{1..\ell-1}\right)
    + \vu_i\sum_{i<j<\ell} \pd{R_j}{\vu_i} \left(\frac1{\sqrt{r_\ell}}\pd{\sqrt{r_\ell}}{\vz_j}\circ h_{1..\ell-1}\right).
  \end{equation}
  By~\eqref{eq:h-domains-ind} the image
  $h_{1..\ell-1}(\sec_{\ell-1}(\e))$ is contained in
  $\cC_{1..\ell-1}^{1/4}$. Applying
  Proposition~\ref{prop:log-derivative-bound} and
  using~\eqref{eq:u_i-vs-z_i} we have for some $K>0$
  \begin{gather}
    \abs{\vu_i \big(\frac1{\sqrt{r_\ell}}\pd{\sqrt r_\ell}{\vz_i}\circ h_{1..\ell-1}\big)}\le \abs{\frac{K\vu_i}{\vz_i}}\le K, \\
    \abs{\frac1{\sqrt{r_\ell}}\pd{\sqrt{r_\ell}}{\vz_j}\circ h_{1..\ell-1}}\le \frac{K}{|\vz_j|} \le \frac{K}{|R_j|}.
  \end{gather}
  Plugging into~\eqref{eq:R_ell-der} and using the
  inductive~\eqref{eq:R_j-C1} we get
  \begin{equation}
    \norm{\frac{\vu_i}{R_\ell} \pd{R_\ell}{\vu_i}}_{\sec_{\ell-2}(\e)}\le K\left(1+\sum_{i<j<\ell} \abs{\frac{\vu_i}{R_j}\pd{R_j}{\vu_i}}\right)
    \le K\big(1+(\ell-i-1) M_{\ell-1}\big)
  \end{equation}
  as claimed. In the algebraic case $K=\poly_\ell(\beta)$ and by
  induction we have indeed $M_\ell=\poly_\ell(\beta)$.

  We now pass to proving that $h$ can be analytically extended to
  $\sec_\ell(\e)$, satisfying~\eqref{eq:h-domains}. By assumption
  $h_{1..\ell-1}$ is already extended and
  satisfies~\eqref{eq:h-domains-ind}. If $\cF_\ell$ is an annulus (the
  $D_\circ$ case is similar but easier) what remains to be verified is
  that
  \begin{equation}\label{eq:R_l-domain-cond}
    R_l(\vu_{1..\ell-1}) +\{ |\Arg\vu_\ell|<\e,|\vu_\ell|<2 \} \subset \{ |R_\ell|<|\vz_\ell|<4 \}
  \end{equation}
  for $\vu_{1..\ell-1}\in \sec_{\ell-1}(\e)$. Having
  established~\eqref{eq:R_j-C1} for $R_\ell$, we can now assume
  that~\eqref{eq:R_i-arg-bound} holds for $R_\ell$ as well (for a
  sufficiently small $\e$). Then~\eqref{eq:R_l-domain-cond} follows by
  an elementary geometric argument illustrated in
  Figure~\ref{fig:SellEpsilon}.
  \begin{figure}
    \centering
    \includegraphics[width=0.6\textwidth]{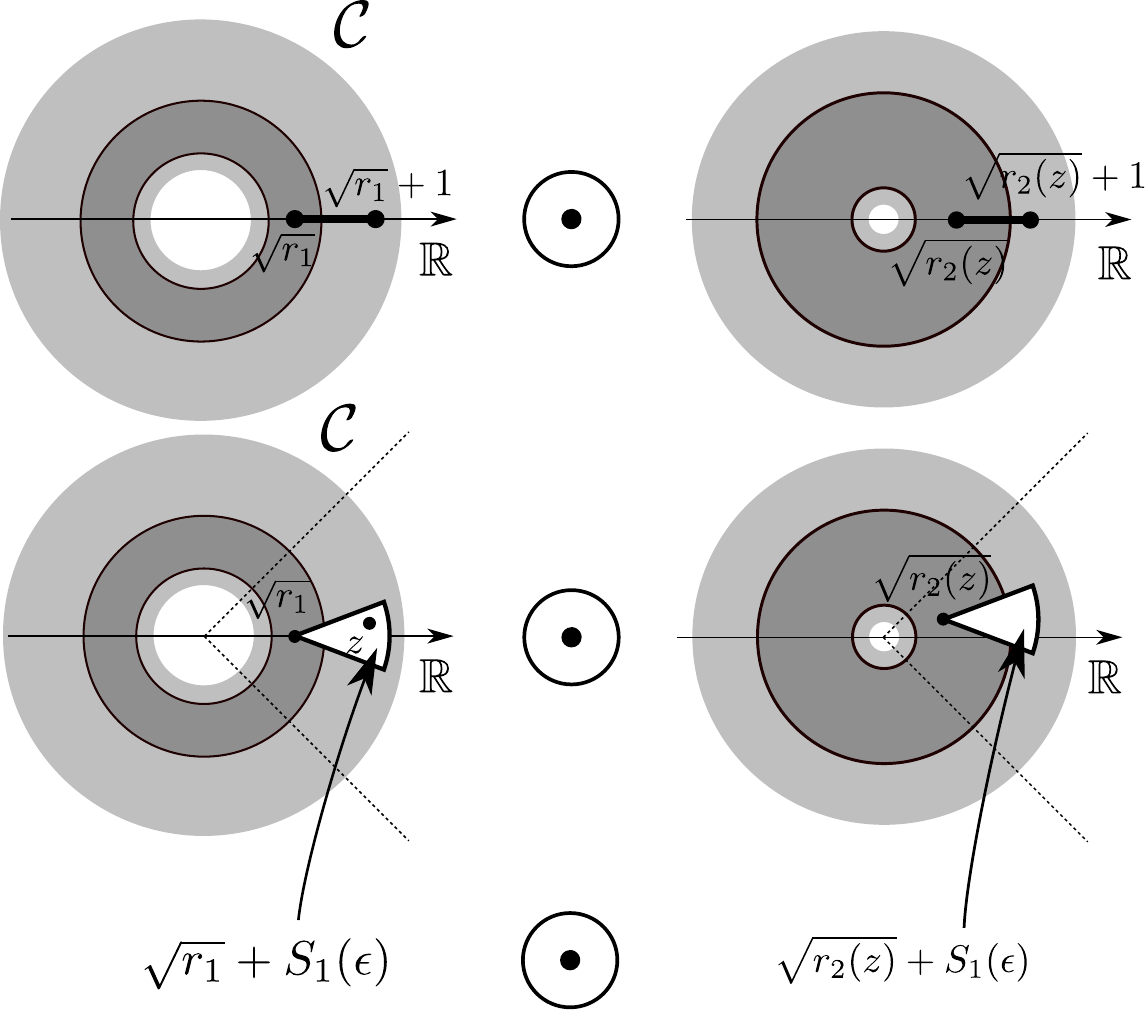}
    \caption{Analytic continuation of $h$ to $\sec_\ell(\e)$.}
    \label{fig:SellEpsilon}
  \end{figure}
\end{proof}

\begin{Cor}\label{cor:h-jacobian-bound}
  Let $\sec_\ell(\e)$ be as in Proposition~\ref{prop:h-extension}. Then
  the Jacobian of $h$ satisfies
  \begin{equation}
    \norm{\left[\pd{h_i}{\vu_j}\right] - I_{\ell\times\ell}}_{\sec_\ell(\e)} \le O_\ell(\delta).
  \end{equation}
\end{Cor}

\begin{proof}
  By induction on $\ell$ we may assume that the Jacobian of
  $h_{1..\ell-1}$ already satisfies the condition. It remains to
  produce a bound for each $\pd{h_\ell}{\vu_j}$ with $j<\ell$. We have
  \begin{equation}
    \pd{h_\ell}{\vu_j} = \pd{\big(\sqrt{r_\ell}\circ h_{1..\ell-1})}{\vu_j},
  \end{equation}
  and since the norm of the Jacobian of $h_{1..\ell-1}$ is $O_\ell(1)$
  it remains to show that the derivatives
  $\pd{\sqrt r_\ell}{\vz_k}\circ h_{1..\ell-1}$ are bounded by $O_\ell(\delta)$
  for $k=1,\ldots,\ell-1$. Since
  $h_{1..\ell-1}(\sec_{\ell-1}(\e))\subset\Qua^{1/4}\cC_{1..\ell-1}$ and $\sqrt{r_\ell}$ is
  bounded by $1$ in $\cC^\delta_{1..\ell-1}$, this follows from
  Lemma~\ref{lem:QC-derivative-bound}.
\end{proof}

The following is a direct corollary of
Lemma~\ref{lem:QC-derivative-bound} and
Corollary~\ref{cor:h-jacobian-bound}.

\begin{Cor}\label{cor:Phi-C1-norm}
  Let $\sec_\ell(\e)$ be as in Proposition~\ref{prop:h-extension}. Let
  $F\in\cO_b(\cC^\delta)$ and set $\Phi:=F\circ h$. Then
  \begin{equation}
    \norm{\pd{\Phi}{\vu_i}}_{\sec_\ell(\e)}< O_\ell(\norm{F}\cdot\delta) \qquad i=1,\ldots,\ell.
  \end{equation}
\end{Cor}

\section{Smooth parametrization results}
\label{sec:smooth-param}

In this section we use complex cells to prove the various refinements
of the Yomdin-Gromov algebraic lemma described
in~\secref{sec:intro-smooth}. All of these results follow fairly
directly from a parametrization result involving functions with a
holomorphic continuation to a complex sector, which we describe first.

\subsection{Sectorial parametrizations of subanalytic sets}
\label{sec:sectorial-param}

We say that $B\subset\R^\ell$ is a cube if it is a direct product of
intervals $(0,1)$ and singletons $\{0\}$. We will say that a map
$h:B\to\R^\ell$ is \emph{cellular} if $h_j$ depends only on
$\vx_{1..j}$. The \emph{sectorial cube} $B(\e)\subset\C^\ell$
corresponding to $B$ is the direct product where we replace each
interval $(0,1)$ by the sector $\sec(\e):=\{|\Arg z|<\e,|z|<2\}$. We
define cellular maps $h:B(\e)\to\C^\ell$ similarly.

\begin{Thm}\label{thm:sectorial-param}
  Let $A\subset[0,1]^\ell$ be subanalytic. There exists a collection
  of cubes $B_1,\ldots,B_N$ and injective cellular maps
  $\phi_j:B_j\to A$ such that $A=\cup_j\phi_j(B_j)$. Moreover, there
  exists $\e>0$ such that each $\phi_j$ extends to a holomorphic map
  $\phi_j:B_j(\e)\to\cP_\ell^{1/2}$ with $C^1$-norm $O_\ell(1)$.

  If $A$ is semialgebraic of complexity $\beta$ then we have
  $N,1/\e=\poly_\ell(\beta)$.
\end{Thm}

\begin{proof}
  Set $\delta=1/100$ as in~\secref{sec:quadrant-straighten}.  We begin
  by applying Corollary~\ref{cor:cpt-semialg} in the semialgebraic
  case or Corollary~\ref{cor:cpt-subanalytic} in the subanalytic case
  with $\hrho<1/2$ and $\hsigma<\delta$, followed by
  Proposition~\ref{prop:quadric-cover} to construct a collection of
  quadric normalized cells $\cC_j$ and real maps
  $f_j:\cC_j^\delta\to\cP_\ell^{1/2}$ such that
  $f_j(\R_+\cC^\delta)\subset A$ and $A=\cup_j f_j(\R_+\Qua\cC_j)$. By
  construction each $f_j$ is a composition of a prepared cellular map,
  a rescaling map,   a covering map and an inversion map. Since each of these maps is
  injective on the positive real part we see that $f_j$ is injective
  on $\R_+\cC_j$.
  
  Let $h_j:B_j(\e)\to\cC^\delta$ be the straightening map
  constructed for $\cC_j$ in~\secref{sec:quadrant-straighten} and set
  $\phi_j:=f_j\circ h_j$. The map $h_j$ is cellular and injective by
  construction, so $\phi_j$ is cellular and injective on $B_j$.  We
  have
  \begin{equation}
    A= \bigcup_j f_j(\R_+\Qua\cC_j) \subset \bigcup_j \phi_j(B_j) \subset A
  \end{equation}
  Finally, by Corollary~\ref{cor:Phi-C1-norm} the $C^1$-norm of
  $\phi_j$ is bounded by $O_\ell(1)$.
\end{proof}
 
Later we will use the Cauchy estimate to pass from the sectorial
parametrizations of Theorem~\ref{thm:sectorial-param} to
parametrizations with bounded higher-order derivatives. Toward this
end the following simple combinatorial lemma will be useful.

\begin{Prop}\label{prop:Phi-decomp}
  Let $B(\e)$ be a sectorial cube and let $\Phi:B(\e)\to\C$ have
  $C^1$-norm $1$. There is a decomposition
  \begin{equation}
    \Phi(\vu) = \sum_{I\subset\{1,\ldots,\ell\}} \Phi_I(\vu)
  \end{equation}
  where $\Phi_I:B(\e)\to\C$ depends on $\vu_i:i\in I$ only, vanishes on
  $\cup_{i\in I}\{\vu_i=0\}$ and has $C^1$-norm bounded by
  $O_\ell(1)$. In particular for every $I\neq\emptyset$ we have
  \begin{equation}
    |\Phi_I(\vu)|\le O_\ell(\min_{i\in I}|\vu_i|), \qquad \text{for
    }\vu\in B(\e).
  \end{equation}
\end{Prop}
\begin{proof}
  We proceed by reverse induction on $j$, defined as the minimal
  $i\in\{1,\ldots,\ell\}$ such that $\Phi$ depends on $\vu_i$ and does
  not vanish identically on $\{\vu_i=0\}$. If the condition defining
  $j$ is empty then we may take $\Phi=\Phi_I$ where $I$ denotes the
  set of indices that $\Phi$ depends on. Otherwise write
  \begin{equation}
    \Phi=(\Phi-R)+R, \qquad R(\vu) = \Phi(\vu_1,\ldots,\vu_{j-1},0,\vu_{j+1},\ldots,\vu_\ell).
  \end{equation}
  Since the first summand vanishes on $\{\vu_j=0\}$ and the second
  summand does not depend on $\vu_j$ we may finish for each of them by
  induction (the $C^1$-norm of the summands are majorated by $2$ and
  $1$ respectively).
\end{proof}

\subsection{The algebraic reparametrization lemmas}

We give more detailed statements of the algebraic lemmas in the
$C^r$-smooth and mild contexts, and show that these statements imply
the statements appearing in~\secref{sec:intro-smooth}.  Our result in
the $C^r$ case is as follows.

\begin{Thm}\label{thm:Cr}
  Let $X\subset[0,1]^\ell$ be subanalytic. There exists $A>0$, cubes
  $B_1,\ldots,B_C$ such that for any $r\in\N$ there exist subanalytic
  injective cellular maps $\phi_j:B_j\to X$ whose images cover $X$,
  with
  \begin{equation}\label{eq:Cr-norm-bound}
    \norm{D^\valpha \phi_j} \le \valpha!(A r)^{|\valpha|}, \quad
    \text{for } |\valpha|\le r.
  \end{equation}
  If $X$ is semialgebraic of complexity $\beta$ then
  $A,C=\poly_\ell(\beta)$ and the maps $\phi_j$ are semialgebraic of
  complexity $\poly_\ell(\beta,r)$.
\end{Thm}

The cellular structure and injectivity of the parametrizing maps
implies that the result of Theorem~\ref{thm:Cr} is automatically
uniform over families. Indeed, suppose $X\subset\R^{n+m}$ is a bounded
subanalytic set viewed as a family $\{X_p\}$ of subsets of
$\R^m$. Construct $\phi_1,\ldots,\phi_C$ as above. For any $p\in\R^n$
and any $j=1,\ldots,C$, the fiber $(\phi_j)_{1..n}^{-1}(p)$ is either
empty or a cube $B_{p,j}\subset\R^m$. The $\phi_j$ of the latter type
restrict to parametrizing maps $\phi_{p,j}$ for $X_p$ with the same
$C^r$-norm bound~\eqref{eq:Cr-norm-bound} (now in $m$ variables). One
can subdivide the domains of the $\phi_{p,j}$ into smaller cubes to
obtain a parametrization with unit $C^r$ norms as follows. Let
$\mu:=\dim X_p$. Subdividing $B_{p,j}$ into $N^\mu$ subcubes $B_{p,j,k}$
of side length $1/N$ and rescaling back to the unit cube gives maps
$\phi_{p,j,k}:B_{p,j,k}\to\R^m$ with
\begin{equation}
  \frac{\norm{D^\valpha \phi_{p,j,k}}}{\valpha!} \le \left(\frac{A r}N\right)^{|\valpha|}, \quad
  \text{for } |\valpha|\le r.
\end{equation}
Taking $N=Ar$ one obtains unit norms above. This gives the $C^r$
statements of~\secref{sec:intro-smooth}.

Our result in the mild case is as follows.

\begin{Thm}\label{thm:mild}
  Let $X\subset[0,1]^\ell$ be subanalytic. There exists $A<\infty$,
  cubes $B_1,\ldots,B_C$ and $(A,2)$-mild injective cellular maps
  $\phi_j:B_j\to X$ whose images cover $X$.

  If $X$ is semialgebraic of complexity $\beta$ then one may take
  $A,C=\poly_\ell(\beta)$.
\end{Thm}

As in the case of Theorem~\ref{thm:Cr}, the result of
Theorem~\ref{thm:mild} is automatically uniform over families due to
the cellular structure and the injectivity of the maps.

\subsection{Proofs of the smooth parametrization results}
\label{sec:param-proofs}

Let $X$ be as in Theorem~\ref{thm:Cr} or Theorem~\ref{thm:mild}. By
Theorem~\ref{thm:sectorial-param} there exists a collection of cubes
$B_1,\ldots,B_C$ and injective cellular maps $\phi_j:B_j\to X$ such
that $X=\cup_j\phi_j(B_j)$, and each $\phi_j$ extends to a holomorphic
map $\phi_j:B_j(\e)\to\C^\ell$ with $C^1$-norm $O_\ell(1)$. Moreover
in the algebraic case $C,1/\e=\poly_\ell(\beta)$. We will always
assume $\e<1$. It will be enough to prove the parametrization results
for each $\phi_j(B_j)$ separately, so we fix a pair
$(B,\phi):=(B_j,\phi_j)$ and proceed to consider the problem of
parametrizing $\phi(B)$.

Let $\Phi$ denote one of the coordinates of $\phi$ and recall that
$\Phi$ has $C^1$-norm $O_\ell(1)$. To control higher derivatives we
will consider the map $\tilde\phi:=\phi\circ S$ where $S$ is a
``flattening'' bijection $S:B\to B$. We will see that an appropriate
choice of $S$ (one for the $C^r$ version and another for the mild
version) allows one to control the higher derivatives of $\Phi\circ S$
using the Cauchy estimates.

To simplify the notation, we will suppose below that $B$ is of the
form $(0,1)^\ell$ (in general the direct product may also contain
copies of $*$, but these obviously have no impact on parametrization
questions). We will take $\Phi$ to be any function of $C^1$-norm $O_\ell(1)$
in $B(\e)$.

\subsubsection{Proof of the $C^r$-parametrization result}
\label{sec:params-proofs-cR}

Consider the bijection (see Figure~\ref{fig:Cr})
\begin{equation}
P_r:B\to B, \qquad P_r(\vw_{1..\ell})=(\vw_1^r,\ldots,\vw_\ell^r).
\end{equation}
In the notations of~\secref{sec:param-proofs}, Theorem~\ref{thm:Cr}
follows from the following lemma.

\begin{Lem}
  The function $\Phi\circ P_r$ has bounded $C^r$-norm,
  \begin{equation}
    \abs{(\Phi\circ P_r)^{(\valpha)}(\vw)}\le  \valpha! \left(O_\ell(r/\e)\right)^{|\valpha|}
    \qquad\text{for }|\valpha|\le r.
  \end{equation}
\end{Lem}
\begin{proof}
  Since $\Phi$ extends holomorphically to $B(\e)$, the
  composition $\Phi\circ P_r$ extends holomorphically to
  \begin{equation}
    \Omega:=B(\e/r)\cap\{|\zeta_i|<2^{1/r}, i=1,\ldots,\ell\}.
  \end{equation}
  In particular for every $\vw\in B$ it contains the polydisc
  $D_\vw=\{\abs{\xi_i-\vw_i}\le (\sin\frac{\e}{r}) \vw_i\}$. Denote by
  $\cS(D_\vw)$ its skeleton.

  Using Proposition~\ref{prop:Phi-decomp} we see that 
  \begin{equation}
    (\Phi\circ P_r)^{(\valpha)}(\vw)
    =\sum_I (\Phi_I\circ P_r)^{(\valpha)}(\vw)
    =\sum_{\supp{\valpha}\subset I} (\Phi_I\circ P_r)^{(\valpha)}(\vw),
  \end{equation}
  so it is enough to consider $(\Phi_I\circ P_r)^{(\valpha)}(\vw)$
  only for those summands with $\supp{\valpha}\subset I$. Applying the
  Cauchy formula we have
  \begin{multline}\label{eq:Cauchy-Cr}
    \abs{(\Phi_I\circ P_r)^{(\valpha)}(\vw)} =
    \abs{(2\pi i)^{-\ell}\valpha!\int_{\cS(D_\vw)}
      \frac{\Phi_I\circ P_r(\xi)}{\prod (\xi_i-\vw_i)^{\valpha_i+1} }d\xi_1\dots d\xi_\ell}\\
    \le\valpha!\prod 
       \left(\vw_i\sin\frac{\e}{r}\right)^{-\valpha_i}\max_{\xi\in\cS(D_\vw)}\abs{\Phi_I\circ
       P_r(\xi)}.
  \end{multline}
  Denote $\vw_{\min}:=\min_{i\in I} \vw_i$. By
  Proposition~\ref{prop:Phi-decomp} we have
  \begin{equation}
    \abs{\Phi_I\circ P_r(\xi)}\le O_\ell(1)\cdot \min_{i\in I}\abs{\xi_i^r}
    \le O_\ell(1)\cdot \abs{\vw_{\min}^r},
  \end{equation}
  where in the last inequality we used
  $\left(1+\sin\tfrac{\epsilon}{r}\right)^r<e$. Then~\eqref{eq:Cauchy-Cr}
  implies
  \begin{equation}
    \abs{(\Phi_I\circ P_r)^{(\valpha)}(\vw)}\le 
    O_\ell(1) \left(\sin\frac{\epsilon}{r}\right)^{-\abs{\valpha}}\valpha!
    \vw_{\min}^{r-\abs{\valpha}}\le \valpha! \left(O_\ell(r/\e) \)^{|\valpha|}
  \end{equation}
  as long as $|\valpha|\le r$.
\end{proof}

\begin{figure}
  \centering
  \includegraphics[width=\textwidth]{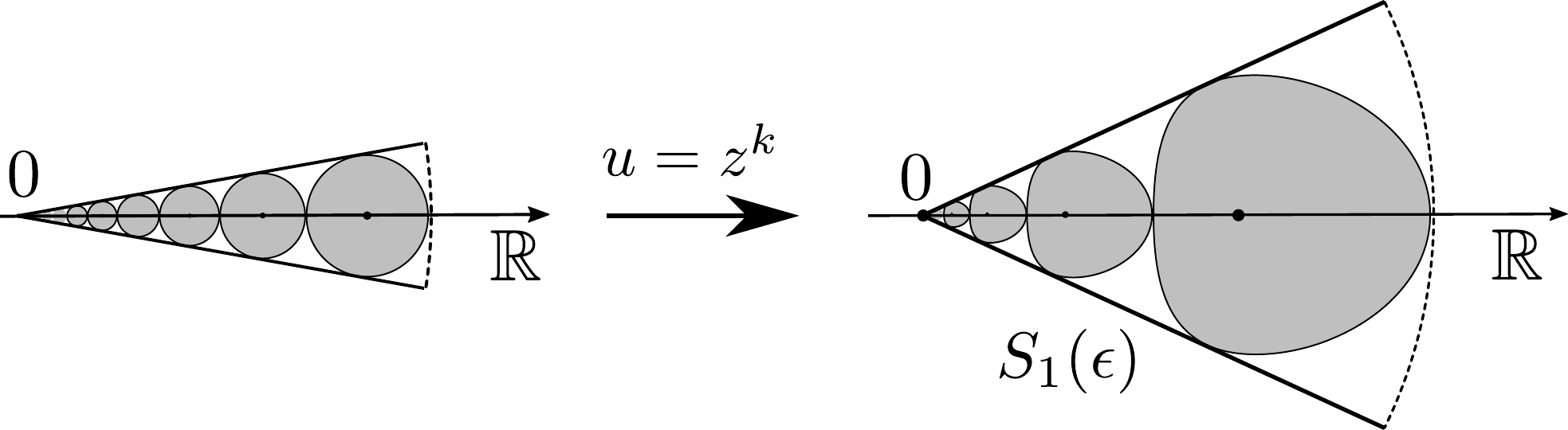}
  \caption{Reparametrization of $B(\e)$ in the $C^r$ case.}
  \label{fig:Cr}
\end{figure}

\subsubsection{Proof of the mild parametrization result}

Consider the bijection (see Figure~\ref{fig:mild})
\begin{equation}
  \Exp:B\to B, \qquad
  \Exp(\vw_{1..\ell})=\left(\exp\left(1-\frac1{\vw_1}\right),\dots,\exp\left(1-\frac1{\vw_\ell}\right)\right).
\end{equation}
In the notation of~\secref{sec:param-proofs}, Theorem~\ref{thm:mild}
follows from the following lemma.

\begin{Lem}
  The function $\Phi\circ\Exp$ is $(A,2)$-mild,
  \begin{equation}
    \abs{(\Phi\circ\Exp)^{(\valpha)}(\vw)}\le  \valpha!(A|\valpha|^2)^{|\valpha|}
    \qquad\text{for } \valpha\in\N^\ell
  \end{equation}
  with $A=O_\ell(1/\e)$.
\end{Lem}
\begin{proof}
  Since $\Phi$ extends holomorphically to $B(\e)$, the
  composition $\Phi\circ\Exp$ extends holomorphically to the domain
  $\Omega:=\Exp^{-1}B(\e)$ and it is easy to check that for any $\vw\in B$,
  \begin{equation}
    D_\vw:=\{\abs{\xi_i-\vw_i}\le \frac 1 2\e \vw_i^2\}\subset\Omega.
  \end{equation}
  Denote by $\cS(D_\vw):=\{\abs{\xi_i-\vw_i}= \frac 1 2\e \vw_i^2\}$
  the skeleton of $D_\vw$.

  Using Proposition~\ref{prop:Phi-decomp} in the same way as
  in~\secref{sec:params-proofs-cR} we see that it is enough to
  consider $(\Phi_I\circ\Exp)^{(\valpha)}(\vw)$ only for those
  summands with $\supp{\valpha}\subset I$. Applying the Cauchy formula
  we have
  \begin{multline}\label{eq:Cauchy-mild}
    \abs{(\Phi_I\circ\Exp)^{(\valpha)}(\vw)} =
    \abs{(2\pi i)^{-\ell}\valpha!\int_{\cS(D_\vw)}
      \frac{\Phi_I\circ\Exp(\xi)}{\prod (\xi_i-\vw_i)^{\valpha_i+1} }d\xi_1\dots d\xi_\ell}\\
    \le\valpha!\prod 
    \left(\vw_i^2 \frac\e2\right)^{-\valpha_i}\max_{\xi\in\cS(D_\vw)}\abs{\Phi_I\circ
      \Exp(\xi)}.
  \end{multline}
  By Proposition~\ref{prop:Phi-decomp} we have 
  \begin{multline}
    \abs{\Phi_I\circ\Exp(\xi)}\le
    O_\ell(1) \min_{i\in I}\abs{\exp\left(1-\xi_i^{-1}\right)} \\
    \le O_\ell(1) \exp(-\max_{i\in I}\Re \xi_i^{-1}) 
    \le O_\ell(1) \exp(-\max_{i\in I} \vw_i^{-1}),
  \end{multline}
  where in the last inequality we used the fact that for $w<\e^{-1}$,
  the mapping $\xi\mapsto1/\xi$ maps the circle
  $\{\abs{\xi-w}\le \frac 1 2\e w^2\}$ inside the circle
  $\{\abs{\xi-w^{-1}}<2\}$.
  
  Denote $\vw_{\min}=\min_{i\in I} \vw_i$. Then~\eqref{eq:Cauchy-mild}
  implies
  \begin{equation}
    \begin{aligned}
      \abs{(\Phi_I\circ\Exp)^{(\valpha)}(\vw)}&\le 
      \valpha! O_\ell(1) (2/\e)^{|\valpha|} \vw_{\min}^{-2\abs{\valpha}}e^{-\frac{1}{\vw_{\min}}} \\
      &\le \valpha! O_\ell\left((8/\e)^{|\valpha|}\right) {\abs{\valpha}}^{2\abs{\valpha}}e^{-2\abs{\valpha}}
    \end{aligned}
  \end{equation}
  where the final estimate is a simple maximization exercise (with the
  maximum attained at $\vw_{\min}=\frac{1}{2\abs{\valpha}}$).
\end{proof}

\begin{figure}
  \centering
  \includegraphics[width=\textwidth]{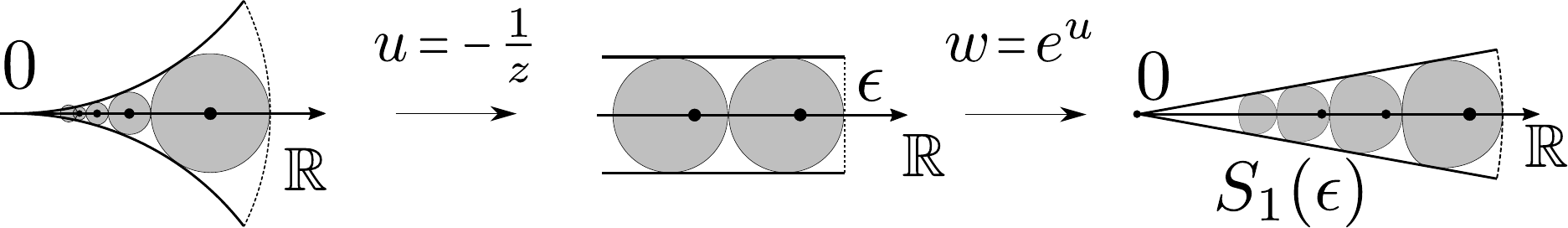}
  \caption{Reparametrization of $B(\e)$ in the mild case.}
  \label{fig:mild}
\end{figure}

\subsection{Uniform parametrizations in $\R_\an$}
\label{sec:sec:uniform-param-Ran}

In this section we give a proof of
Proposition~\ref{prop:hyp-param}. We begin by proving
Lemma~\ref{lem:log-length}.

\begin{proof}[Proof of Lemma~\ref{lem:log-length}]
  It is enough to prove that $\log\xi$ has derivative bounded by
  $4\pi p^2$ in $[-1,1]$. Note that $\log\xi$ is well defined and
  $p$-valent in $D(2)$. Let $r:=|(\log\xi)'(z_0)|$ for some
  $z_0\in D(1)$. By the $p$-valent analog of the Koebe 1/4-theorem
  $\log\xi(z_0+D(1))$ contains a disc of radius at least
  $r/(4p)$, see \cite[Theorem~5.1]{hayman:book}. On the
  other hand, it cannot contain a disc of radius larger than
  $\pi p$ since then $\xi=e^{\log\xi}$ would not be
  $p$-valent. Thus $r\le 4\pi p^2$ as claimed.
\end{proof}

We now proceed to the proof of Proposition~\ref{prop:hyp-param}. To
avoid confusion we will denote the order of derivatives in the
statement of the proposition by $k$ instead of $r$, which we reserve
here to denote a radius. As a first step we will reduce to the case of
a single complex cell. Let
$G_F\subset[0,1]_{e,t}^2\times[0,1]_{x,y}^2$ denote the graph of
$F(e,t)=F_e(t)$. We also denote $x(e,t)=x_e(t)$ and $y(e,t)=y_e(t)$.
According to Theorem~\ref{cor:cpt-subanalytic} and
Remark~\ref{rem:cpt-semi-extra} there exists a collection of prepared
cellular maps $f_j:\cC_j^{1/4}\to\cP_4^{1/2}$ such that
$f_j(\R_+\cC_j)$ covers $G_F$ and
$f_j(\R_+\cC_j^{1/2})\subset G_F$. By Remark~\ref{rem:cpt-semi-extra}
we may also assume that $f_j$ is compatible with the function $e$. Let
$\cC$ denote one of the cells $\cC_j$ and $f:=f_j$. It will be enough
to prove that
\begin{equation}\label{eq:log-length}
  \log [f(\R_+\cC)\cap\{e=e_0\}]
\end{equation}
has bounded Euclidean length independent of $e_0\neq0$.

We first note that since $\cC$ maps into a graph, its type must end
with $**$. We proceed to consider the first two fiber types. If the
type of $\cC$ begins with $*$ then $f(\R_+\cC)$ meets only one line
$e=e_0$ and the length of~\eqref{eq:log-length} is bounded by the
(finite) log-length of the hyperbola $xy=e_0$ in $[0,1]^2$. Similarly
if the type of $\cC$ begins with $D$ or $A$ then since $f_j$ is
compatible with $e$ we conclude that $f_j(\R_+\cC)$ meets only lines
$e=e_0$ with $e_0$ bounded away from $0$, and the same argument
holds. The only non-trivial case is therefore when the type of $\cC$
begins with $D_\circ$ and $f$ maps this punctured disc to a disc
centered at $e=0$.

Let $\cC$ be of the form $\cC=D_\circ(R)\odot \cF\odot*\odot *$. If
$\cF=*$, then $f(\R_+\cC)$ meets, for every $e_0$, only one point in
$xy=e_0$. Otherwise we may write $f$ in the form
\begin{equation}
  (e,t,x,y) = f(\e,\tau,*,*) = 
  (\e^m,\tau^n+\phi(\e),\xi(\e,\tau),\eta(\epsilon,\tau)).
\end{equation}
If $\cF=D(r)$, then the functions $\xi(\e,\tau),\eta(\epsilon,\tau)$
are $p$-valent in $\tau\in D^{1/2}(r)$, for some $p$ uniform over
$\e\in D_\circ(R)$ by subanalyticity of $f$ restricted to $\cC^{1/2}$
(see Proposition~\ref{prop:cell-subanalytic}). The claim then follows
from Lemma~\ref{lem:log-length}. If $\cF=D_\circ(r)$, then by
removable singularity theorem both $\xi(\e,\tau)$ and
$\eta(\epsilon,\tau)$ can be extended to $D_\circ(R)\odot D(r)$, with
extensions satisfying the same relation $\xi\cdot\eta=\e^m$ for
$\e\in D_\circ(R)$. In particular, both $\xi(\e,\tau)$ and
$\eta(\epsilon,\tau)$ do not vanish, and we can proceed as above.

The remaining case is $\cF=A(r_1,r_2)$. We now make three
reductions. First, we replace $f$ with
\begin{equation}
(e,t,x,y) = g(\e,\tau,*,*) =\left(\e^m, \tau^n/r_2^n(\e),\xi(\e,\tau), 
\eta(\e,\tau)\right)
\end{equation}
defined on the same cell. This amounts to an affine transformation of
the coordinate $t$. Evidently, the image of $\xi,\tau$ is
unchanged, and the constant $M_k$ in \eqref{eq:derivatives bounded} is
replaced by $\abs{r_2^{nk}(\e)}M_k$, which is again uniformly bounded
in $\e$. Second, we replace $\cC$ with
$\tilde{\cC}=D_\circ(R)\odot A(r,1)\odot*\odot*$ for
$r:=r_1(\e)/r_2(\e)$, and replace $g$ with the composition
\begin{equation}
(e,t,x,y) = \tilde{g}(\e,\tau,*,*)=g\circ\pi 
=\left(\e^m, \tau^n,\xi(\e,\tau r_2(\e)), \eta(\e,\tau r_2(\e))\right),
\end{equation}
where $\pi:\tilde{\cC} \to\cC$ is the mapping
$(\e,\tau,*,*)\to (\e,\tau r_2(\e),*,*)$. This only amounts to a
reparametrization of the set $g(\cC)$ leaving both the image of
$\xi,\eta$ and the bounds \eqref{eq:derivatives bounded} unchanged.
Finally by a similar rescaling of $\e$ we may assume that $R=1$.
Without loss of generality we return to our original notation,
assuming now that $\cC=D_\circ(1)\odot A(r,1)\odot*\odot*$ and
$f(\e,\tau,*,*)=\left(\e^m, \tau^n,\xi(\e,\tau), \eta(\e,\tau)\right)$.

Using Corollary~\ref{cor:laurent-disc-decmp}, decompose $\eta$ as
\begin{equation}
\eta(\e,\tau)=a(\e,\tau)+b(\e,\tfrac{r}{\tau}),
\end{equation}
where both $a,b$ are holomorphic and bounded in a neighborhood of
$\cP_2^{1/2}$. We have
$\partial_t =\frac{\tau^{1-n}}{n} \partial_\tau$ and a straightforward
induction shows that if $\partial_t^k\eta$ is bounded then
$\partial_\tau^k\eta$ is bounded as well. Since $\partial_\tau^k a$ is
bounded uniformly in $\e$ by the Cauchy estimates for $a$, we conclude
that $\partial_\tau^k b(\e, \tfrac{r}{\tau})$ is also bounded
uniformly in $\e$. We claim that this is impossible unless
$b\equiv b(\e)$ or $r\sim 1$. This finishes the proof: in the former
case $\eta$ itself extends as a holomorphic function to $\cP_2^{1/2}$
and we return to the case of functions on discs treated above using
Lemma~\ref{lem:log-length}; in the latter case, we may (for instance
following the proof of the refinement theorem,
Theorem~\ref{thm:cell-refinement}) cover $D_\circ(1)\odot A(r,1)$ by
finitely many cells of the type $D_\circ(1)\odot D(1)$ returning again
to the disc case.

It remains to prove the claim above. Let $v=\tfrac{r}{\tau}$, so
$b(\e,v)$ is holomorphic and bounded in a neighborhood of
$\cP^{1/2}$. Suppose
\begin{align}
  r(\e)&=\e^m(c+\dots),& &m>0,c\neq0 \\
  b(\e, v)&=\sum_{\ell\ge0} \e^\ell b_\ell(v),& &b_\ell\in\cO_b(\D^{1/2}),
\end{align}
and let $\ell_0$ be the first index such that $b_{\ell_0}$ is non-constant.
A simple computation using $\partial_\tau = -\frac{v^2}{r}\partial_v$ gives 
\begin{equation}
  \partial_\tau^k b(\e, v)=\sum_{\ell\ge\ell_0-mk} \e^\ell \tilde{b}_{\ell,k}(v),
\end{equation}
where
\begin{equation}
  \tilde b_{\ell_0-mk,k} = \left(-\tfrac{v^2}{c}\partial_v\right)^k b_{\ell_0}(v)\not\equiv0
\end{equation}
and the final inequality follows since $b_{\ell_0}$ is non-constant,
and the operator $(v^2\partial_v)^k$ does not kill any non-constant
holomorphic functions as can be seen by examining its action on Taylor
expansions. For $k$ such that $\ell_0-mk<0$ we see that
$\partial_\tau^kb(\e, v)$ has a pole in $\e=0$ for any $v$ such that
$\tilde b_{\ell_0-mk,k}(v)\neq0$ and is therefore unbounded as
$\e\to0$, contradicting our assumption.

\subsection{Proof of Theorem~\ref{thm:analytic-entropy}}
\label{sec:analytic-entropy-proof}

We indicate the proof with the notations of \cite{bly}. Applying the
refined algebraic lemma to the graph of the polynomial map
$P:[0,1]^l\to[0,1]^l$ in \cite[Section~2.1.2]{bly} one obtains the
estimate $C_{r,l,m}=\poly_{m,l}(r)$. This implies that the sequence
$M_k:=k^{k^2}$ is $(l,m)$-admissible for any $l,m\in\N$. Then
\cite[Corollary~B]{bly} applies to any map $f$ which satisfies
$\norm{f}_r < C r^{r^2}$ for some constant $C$ and every $r\ge0$. In
particular, it applies to analytic maps. We also have
\begin{equation}
  G_M(t) = \sup_k \{ k\log k<t \} = \Theta\left(\frac{t}{\log t}\right).
\end{equation}
Thus~\eqref{eq:analytic-entropy} follows from the conclusion of
\cite[Corollary~B]{bly}.

\appendix

\section{Applications in dynamics (by Yosef Yomdin)}
\label{appendix:yomdin}

\subsection{Non-dynamical parametrization results}
\label{Sec:Non.Dyn}

In this section we present a non-dynamical result, based on the
refined algebraic lemma. This result reduces the problem of bounding
dynamical local entropy (considered in~\secref{sec:Dyn.Appl} below) to
a ``combinatorial'' analysis of the complexity of long compositions of
smooth functions. Set $Q:=[0,1]$. Our basic unit of complexity is
defined as follows.

\begin{Def}[$k$-complexity unit ($k$-cu)]
  A $k$-complexity unit is a $C^k$ map $\phi:Q^l\to\R^m$ with
  $\norm{\phi}_k\le1$. We say that $\{\phi_q:Q^l\to\R^m\}$ is a $k$-cu
  cover of a set $A\subset\R^m$ if $A$ is contained in the union of
  the images of $\phi_q$.
\end{Def}

\begin{Rem}\label{rem:alg-lemma-closed}
  In this appendix it is more convenient to use $Q=[0,1]$ in place of
  $(0,1)$. We remark that that in the formulation of the refined
  algebraic lemma (Lemma\ref{thm:refined-alg-lemma}), if the
  semialgebraic set $X$ is closed then one can use parametrizing maps
  $\phi_j:Q^\mu\to X$ with unit $C^k$-norms in $Q^\mu$. To see this,
  apply the refined algebraic lemma to obtain a $C^{k+1}$
  parametrization of $X$ by maps $\phi_j:(0,1)^\mu\to X$ and note that
  each such map extends to a $C^k$ map on $Q^\mu$ by elementary
  calculus, and $\phi_j(Q^\mu)\subset X$ since $X$ is closed.
\end{Rem}

Theorem~\ref{thm:non.dyn} below can be considered as a kind of
``Taylor formula'' for $C^k$-parametrizations: it measures the
complexity of the graph of a $C^k$-map in terms of a covering by
$k$-cu's; exactly as in the usual Taylor formula, the estimates depend
only on the $k$-th derivative. A result of this type appears (in a
weaker form) in \cite[Theorem~2.1]{yomdin:entropy}. In
\cite{gromov:gy} it is improved and split into ``Main lemma'' and
``Main corollary''. Following \cite{gromov:gy}, we split below
Theorem~2.1 of \cite{yomdin:entropy} into Theorem~\ref{thm:non.dyn}
and Corollary~\ref{cor:non.dyn}, and incorporate other improvements
due to Gromov. Below we denote by $G_g$ the graph of a map $g$.

\begin{Thm}\label{thm:non.dyn}
  Let $g:Q^l\to\R^m$ be a $C^k$-mapping such that
  $\max_{x\in Q^l}\norm{d^kg}\le M.$ Then $G_g\cap Q^{l+m}$ admits a
  $k$-cu cover of the form $\{(\phi_q,g\circ\phi_q)\}$ of size
  $\poly_{l,m}(k)M^{l/k}$.
\end{Thm}

\begin{proof}
  Let $S:=\{x\in Q^l:\norm{g(x)}\le1\}$. We subdivide $Q^l$ into
  sub-cubes $Q^l_j$ of size $\gamma=(kM^{1/k})^{-1}$ and parametrize
  each $Q^l_j$ by an affine mapping $\eta_j:Q^l\to Q^l_j$. The number
  $N_1$ of the sub-cubes is $k^lM^{l/k}$. Since the $k$-th derivative
  under this parametrization is multiplied by $\gamma^k$, we get for
  $\theta_j:=g\circ \eta_j$ the bound $\norm{d^k\theta_j}\le k^{-k}$.

  Let $P_j$ be the Taylor polynomial of $\theta_j$ of degree $k-1$ at
  the center of $Q^l$. By the Taylor remainder formula we have
  \begin{equation}\label{eq:Taylor.appr.norm}
    \norm{d^l\theta_j-d^lP_j}\leq \frac{1}{(k-l)!k^k}, \qquad l=0,\ldots,k.
  \end{equation}
  Put $S_j=\{x\in Q^l: \ ||P_j(x)||\le \frac{3}{2}\}.$
  By~\eqref{eq:Taylor.appr.norm} we have $|\theta_j-P_j|< \frac{1}{2},$
  and hence the sets $\eta_j(S_j)$ for $j=1,\ldots,N_1$ cover $S$.

  In the next step, which is central for all the construction, we
  apply the refined algebraic lemma (see
  Remark~\ref{rem:alg-lemma-closed}) to the graph of $P_j$ over $S_j$
  with smoothness order $k$. We get $k$-cu's $\psi_{ij}:Q^l\to S_j$
  for $i=1,\ldots,N_2$ with the images of $\psi_{ij}$ covering $S_j$,
  such that $P_j\circ \psi_{ij}$ are also $k$-cu's. Moreover
  $N_2\le\poly_{l,m}(k)$. Now we estimate the derivatives
  $(\theta_j\circ \psi_{ij})^{(s)}$ for $s=1,\ldots,k$, comparing them
  with the derivatives of $P_j\circ \psi_{ij}$ via the Fa\`a di Bruno
  formula:
  \begin{equation}\label{eq:Faa.di.Bruno1}
    (F\circ G)^{(s)}= \sum_{l=1}^s B_s^l(G',G'',\ldots,G^{(s-l+1)})\cdot F^{(l)}\circ G,
  \end{equation}
  with $B^l_s$ the so-called Bell polynomials, satisfying in
  particular
  \begin{equation}
    B_s:=\sum_{l=1}^s B^l_s(1,\ldots,1)<s^s.
  \end{equation}
  This formula is linear with respect to $F$, and applying it to
  $F=\theta_j-P_j$ and $G=\psi_{ij}$ we get for $s=1,\ldots,k$,
  \begin{multline}
    |(\theta_j\circ \psi_{ij})^{(s)}- (P_j\circ \psi_{ij})^{(s)}|= |((\theta_j-P_j)\circ \psi_{ij})^{(s)}|\leq \\
    \sum_{l=1}^s B_s^l(\psi_{ij}',\psi_{ij}'',\ldots,\psi_{ij}^{(k-l+1)})\cdot |(\theta_j-P_j)^{(l)}\circ \psi_{ij}| \le 
    \frac{B_s}{k^k}\le 1,
  \end{multline}
  by the bounds~\eqref{eq:Taylor.appr.norm}, and because $\psi_{ij}$
  are $k$-cu's. Since $P_j\circ \psi_{ij}$ is also a $k$-cu we
  conclude that $|(\theta_j\circ \psi_{ij})^{(s)}|\le2$. Therefore,
  another subdivision of $Q^l$ into sub-cubes of size $\frac{1}{2}$
  and the corresponding affine parametrizations of these sub-cubes
  reduces the derivative bounds to $1$. We denote by
  $\{\psi_{i,j,p}\}$ the collection of
  $N_1N_2 2^l=\poly_{l,m}(k)M^{l/k}$ maps obtained in this way. Set
  $\phi_{i,j,p}:=\eta_j\circ\psi_{i,j,p}$. Then $\phi_{i,j,p}$ is a
  $k$-cu since $\psi_{i,j,p}$ is, and
  $g\circ\phi_{i,j,p}=\theta_j\circ\psi_{i,j,k}$ is a $k$-cu as shown
  above.
\end{proof}

Next we modify Theorem \ref{thm:non.dyn} to make it suitable to apply
to $n$-fold iterations
$f^{\circ n}=f\circ f \circ \ldots \circ f$. The main reason to
separate Theorem~\ref{thm:non.dyn} from Corollary~\ref{cor:non.dyn}
below is that in the latter it is not enough to assume that
$||d^kg||\le M$ only for the highest derivative: we need
$||d^sg||\le M$ for all $s=1,\ldots,k$.

\begin{Cor}\label{cor:non.dyn}
  Let $g:Q^l\to\R^m$ be a $C^k$-mapping such that
  $\max_{x\in B'}\norm{d^sg}\le M$ for $s=1,\ldots,k$. Let
  $\sigma:Q^l\to Q^l$ be a $k$-cu and put $h=g\circ\sigma$.  Then
  $G_h\cap Q^{l+m}$ admits a $k$-cu cover of the form
  $\{(\phi_q,h\circ\phi_q)\}$ of size $\poly_{l,m}(k)M^{l/k}$.
\end{Cor}
\begin{proof}
  By the Fa\`a di Bruno formula~\eqref{eq:Faa.di.Bruno1} we have for
  the $k$-th derivative of $h$
  \begin{equation}\label{eq:Faa.di.Bruno2}
    |h^{(k)}|=|(g\circ \sigma)^{(k)}|= |\sum_{l=1}^k B_k^l(\sigma',\sigma'',\ldots,\sigma^{(k-l+1)})\cdot g^{(l)}\circ \sigma| < k^kM.
  \end{equation}
  Now we apply Theorem~\ref{thm:non.dyn} to $h$, and get the required
  covering by $k$-cu's.
\end{proof}

\subsection{Applications to volume growth and entropy}
\label{sec:Dyn.Appl}

Let $Y$ be a compact $m$-dimensional $C^\infty$-smooth Riemannian
manifold, and let $f:Y\to Y$ be a smooth (at least $C^{1+\mu}$, with
$\mu>0$) mapping. The topological entropy $h(f)$ was defined
in~\secref{sec:intro-dynamics}. Below we consider another
``entropy-type'' dynamical invariant called the \emph{volume growth}.

\subsubsection{Volume growth}
\label{appendix:sec:volume-growth}

For $\sigma^l:Q^l\to Y$ a $C^\infty$-mapping, and for $\cS\subset Q^l$
a measurable subset, let $v(\sigma^l,\cS)$ denote the $l$-dimensional
volume of the image of $\sigma^l(\cS)$ in $Y$. Put
$v(f,\sigma^l,\cS,n)=v(f^{\circ n}\circ \sigma^l,\cS)$, and put
\begin{align}\label{eq:vol.growth}
  v(f,\sigma^l,\cS)&= \limsup_{n\to \infty} \ \frac{1}{n} \log v(f,\sigma^l,\cS,n), &
  v(f,\sigma^l)&=v(f,\sigma^l, Q^l).
\end{align}
Finally, we put
\begin{align}
  v_l(f)&=\sup_{\sigma^l}v(f,\sigma^l), & v(f)&=\max_{l} v_l(f)
\end{align}
where the supremum is taken over all $C^\infty$ maps
$\sigma^l:Q^l\to Y$. An important inequality obtained by Newhouse in
\cite{newhouse:entropy-volume} connects the volume growth to the
topological entropy,
\begin{equation}\label{eq:Newhouse}
  h(f)\le v(f).
\end{equation}
The inverse inequality is not always true, but the question of its
validity is important in many dynamical applications. In particular,
as it was explained in~\secref{sec:intro-yomdin-thm}, this inverse
inequality implies Shub's entropy conjecture. In order to analyze its
validity we define a $\delta$-local volume growth at $x\in Y$ as
\begin{equation}
  v_l(f,\delta,x)=\sup_{\sigma^l} \ v(f,\sigma^l,B^n_{\delta}(x))
  = \sup_{\sigma^l} \limsup_{n\to \infty} \ \frac{1}{n} \log v(f,\sigma^l,B^n_{\delta}(x),n)
\end{equation}
where $B^n_{\delta}(x)$ is the $\delta$-ball centered at $x$ in the $n$-th
iterated metric $d_n$. The quantity $v_l(f,\delta,x)$ measures the
exponential rate of volume growth of the part of the images of
$f^{\circ n}\circ \sigma^l$ which always stay at a distance at most
$\delta$ from the orbit of $f^{\circ n}(x)$ for $n\to\infty$. Finally,
we put
\begin{align}
  v_l(f,\delta)&=\sup_{x\in Y} v_l(f,\delta,x), & v^0_l(f)&=\lim_{\delta\to 0} v_l(f,\delta).
\end{align}
The $\delta$-tail entropy can be also be defined in an analogous
manner setting $h^*(f,\delta,x)=h(f,B^n_{\delta}(x))$ and
$h^*(f,\delta):=\sup_{x\in Y} h^*(f,\delta,x)$. There is a version of
Newhouse's inequality~\eqref{eq:Newhouse} for these invariants.

\subsubsection{Local volume growth and entropy}

The following partial inverse to~\eqref{eq:Newhouse} is
almost immediate:
\begin{equation}
  v_l(f)\le h(f,\delta)+v_l(f,\delta).  
\end{equation}
Indeed, $h(f,\delta)$ counts the minimal number of the balls
$B^n_{\delta}(x)$ covering $Y$, while $v_l(f,\delta)$ bounds the
volume growth on each of these balls. As $\delta\to 0$ we obtain the
bound
\begin{equation}
  v_l(f)\le h(f)+v^0_l(f).  
\end{equation}
Thus the inverse to the Newhouse inequality is satisfied if
$v^0_l(f)=0$. Our first goal is to establish the following result.

\begin{Thm}[\protect{\cite{yomdin:entropy}}]\label{thm:yomdin-Cinf-volume}
  If $f:Y\to Y$ is $C^\infty$-smooth then $v_l^0(f)=0$.
\end{Thm}

\subsubsection{Local $C^k$-complexity growth}

From now on we concentrate on bounding from above the local volume
growth $v_l(f,\delta,x)$ for a given $x\in Y$. For this purpose we
define yet another entropy-like invariant $h_{l,k}(f,\delta,x)$, which
measures the ``local $C^k$-complexity growth'' in a
$\delta$-neighborhood of the orbit of $x$
(cf. \cite{gromov:gy,yomdin:entropy-analytic}):
\begin{equation}\label{Ck.local.comp}
  h_{l,k}(f,\delta,x)=\sup_{\sigma^l}\limsup_{n\to \infty}\frac{1}{n}\log N_k(f,\sigma^l,\delta,x,n),
\end{equation}
where $\sigma^l:Q^l\to Y$ varies over all $C^k$-smooth mappings and
$N_k(f,\sigma^l,\delta,x,n)$ is the minimal number of $k$-cu's
$\psi_j:Q^l\to Q^l$, whose images cover
$(\sigma^l)^{-1}(B^n_{\delta}(x))$, and for which
$f^{\circ n}\circ \sigma^l \circ \psi_j$ are $k$-cu's.

An almost immediate fact is that
\begin{equation}\label{eq:vol.Ck.ineq}
  v_l(f,\delta,x)\le h_{l,k}(f,\delta,x).
\end{equation}
Indeed, since $f^{\circ n}\circ \sigma^l \circ \psi_j$ are $k$-cu's,
the $l$-volume of their image does not exceed a certain constant
depending only on $l,m$ and hence
\begin{equation}
  v_l(f,\sigma^l,B^n_{\delta}(x),n)\le O_m(N_k(f,\sigma^l,\delta,x,n))
\end{equation}
As $n\to \infty$ in~\eqref{eq:vol.growth} and~\eqref{Ck.local.comp}
the asymptotic constant disappears. Thus it remains to bound from
above $h_{l,k}(f,\delta,x)$.

Below the derivatives of $f:Y\to Y$, and their norms, are defined via
the local coordinate charts on $Y$.

\begin{Prop}\label{prop:bound.hk}
  Assume that for certain positive constants $L,M$ the inequalities
  \begin{align}\label{eq:deriv.bd}
    \norm{df}&\le L,& \max_{s=2,\ldots,k}\norm{d^sf}^{\frac{1}{s-1}}&\le M
  \end{align}
  are satisfied. Then for each $x\in Y$ and for each
  $\delta\le M^{-1}$ we have
  \begin{equation}\label{eq:bound.hk1}
    h_{l,k}(f,\delta,x)\le \frac{l\log L}{k} + \log\poly_m(k).
  \end{equation}
\end{Prop}

\begin{proof}
  Using the local coordinate charts at the points
  $x,f(x),f^{\circ 2}(x),\ldots$ of the orbit of $x$ under $f$, we
  replace iterations of $f$ with compositions of a non-autonomous
  sequence $F$ of mappings $f_i:B^m_1\to\R^m$ with $f_i(0)=0$, where
  $B^m_1$ is the unit ball in $\R^m$, centered at the origin.

  Fix $\delta\le M^{-1}$, and consider the concentric ball
  $B^m_\delta\subset B^m_1.$ Let $\eta: B^m_1\to B^m_\delta$ be a
  linear contraction $\eta(x)=\delta x$. We consider instead of the
  sequence $F$ of mappings $f_i$ a sequence $\bar F$ of
  $\delta$-rescaled mappings
  $\bar f_i=\eta^{-1}\circ f_i\circ \eta:B^m_1\to {\mathbb R}^m.$ For
  the derivatives we have $d^s\bar f_i=\delta^{s-1}d^s f_i$, and by
  the assumptions we get
  \begin{equation}
    ||d\bar f_i||\le L, \ \max_{s=2,\ldots,k}||d^s \bar f_i||\le 1.
  \end{equation}
  Thus $\delta$-rescaling ``kills'' all the derivatives of $f_i$,
  starting with the second one. The first derivative (having a
  dynamical meaning) does not change.

  We have to estimate the number
  $\bar N_{k,n}:=N_k(\bar F,\sigma^l,1,0,n),$ which, by an evident
  change of notations, is the minimal number of $k$-cu's
  $\psi_j:Q^l\to Q^l$, whose images cover
  $(\sigma^l)^{-1}(\bar B^n_{1}(0))$, and for which
  $\bar f_n\circ \bar f_{n-1}\circ \ldots \circ \bar f_1 \circ \sigma^l \circ \psi_j$
  are $k$-cu's.

  We will prove via induction by $n$ that
  \begin{equation}\label{eq:induction.ineq}
    \bar N_{k,n}\leq L^{\frac{nl}{k}}C^n
  \end{equation}
  where $C=\poly_m(k)$ is the constant from
  Corollary~\ref{cor:non.dyn}.  Assume that~\eqref{eq:induction.ineq}
  is valid for $n$. To prove it for $n+1$ we use
  Corollary~\ref{cor:non.dyn}.  We apply it to $g=\bar f_{n+1}$ and to
  each $k$-cu $\nu$ of the form
  \begin{equation}
    \nu=\bar f_n\circ \bar f_{n-1}\circ \ldots \circ \bar f_1 \circ \sigma^l \circ \psi_j,
  \end{equation}
  obtained in the $n$-th step. For each $\nu$ we get at most
  $N=CL^{\frac{l}{k}}$ new $k$-cu's $\phi_q: Q^l\to Q^l$ whose images
  cover $(\sigma^l)^{-1}(\bar B^{n+1}_{1}(0))$, and for which the
  compositions $\bar f_{n+1}\circ \nu\circ \phi_q$ are also
  $k$-cu's. By the induction assumption, the number $\bar N_{k,n}$ of
  the $k$-cu's $\nu$ is at most $L^{\frac{nl}{k}}C^n$. Therefore
  \begin{equation}
    \bar N_{k,n+1}\le \bar N_{k,n}N\le \bar N_{k,n}CL^{\frac{l}{k}}\le L^{\frac{(n+1)l}{k}}C^{n+1}.  
  \end{equation}
  This completes the induction, proving~\eqref{eq:induction.ineq} for
  $n\in\N$. Taking $\log$ and dividing by $n$ finishes the proof.
\end{proof}

To finish the proof of Theorem~\ref{thm:yomdin-Cinf-volume} we should
eliminate the extra term $\log\poly_m(k)$
in~\eqref{eq:bound.hk1}. We do this by replacing $f$ by its iterate
$f^{\circ q}$ for some sufficiently large $q$.

Let $f$ be as in Proposition~\ref{prop:bound.hk}. For each $q\in\N$
the first derivative of $f^{\circ q}$ is bounded by $L^q$. Denote by
$M_q$ a constant satisfying
\begin{equation}
  \max_{s=2,\ldots,k}||d^sf^q||^{\frac{1}{s-1}}\le M_q.
\end{equation}
The constants $M_q$ can be estimated from $L,M$ using Fa\`a di Bruno
formula (see, e.g. \cite{bly}). We do not provide here an explicit
expression, since below, for analytic $f$, we get it in a much easier
way. Using the equality
$h_{l,k}(f^{\circ q},\delta,x)=qh_{l,k}(f,\delta,x)$ and applying
Proposition~\ref{prop:bound.hk} to $f^{\circ q}$ we obtain the
following result, which contains Theorem~\ref{thm:yomdin-Cinf-volume}.

\begin{Cor}\label{cor:bound.hk1}
  For each $\delta\le1/M_q$ we have
  \begin{equation}\label{eq:bound.hk2}
    h_{l,k}(f,\delta,x)\le \frac{l\log L}{k} + \frac{\log\poly_m(k)}{q}.    
  \end{equation}
  In particular, for $q\to\infty$ and $\delta\to0$,
  \begin{align}\label{eq:lim.delta}
    \lim_{\delta\to0} h_{l,k}(f,\delta,x)&\le \frac{l\log L}{k}, & v^0_l(x)&=0.
  \end{align}
\end{Cor}

\subsubsection{Volume growth for analytic maps}

Now we show, following \cite{bly} how the
Corollary~\ref{cor:bound.hk1} can be used to provide an explicit bound
for the local volume growth and local entropy. Setting
$q(k)=O_m(1)\cdot \frac{k \log k}{l\log L}$ we immediately obtain the
following.

\begin{Cor}\label{cor:bound.hk2}
  For each $\delta \leq \delta(k)= \frac{1}{M_{q(k)}}$ we have
  \begin{equation}\label{eq:bound.hk3}
    h_{l,k}(f,\delta,x)\le \frac{2l\log L}{k}.
  \end{equation}
\end{Cor}

For $C^\infty$ functions $f$ Corollary~\ref{cor:bound.hk2} reduces the
problem of estimating the asymptotic behavior of $h_{l,k}(f,\delta)$
as $\delta\to0$ to the problem of estimating the growth of $M_{q(k)}$,
i.e. the growth of the high-order derivatives of $f^{\circ q(k)}$. We
find it explicitly for analytic $f$, referring the reader to
\cite{bly}, where a much more general setting is presented.

Let $Y\subset {\mathbb R}^p$ be a compact $m$-dimensional real
analytic submanifold, and let $f:Y\to Y$ be a real analytic
mapping. We assume that $f$ is extendible to a complex analytic
mapping $\tilde f:U\to {\mathbb C}^p$ in a neighborhood $U$ of $Y$ in
${\mathbb C}^p$ of size $\rho$. More accurately, for each point
$x\in Y$ the complex polydisc at $x$ of radius $\rho$ is contained in
$U$. We assume that the norm $||\tilde f||$ is bounded by $D$ in $U$,
and the norm of its first derivative $||d\tilde f||$ is bounded there
by $L.$ The following theorem immediately implies the statement about
volume growth in Theorem~\ref{thm:analytic-entropy} since the volume
of the image of a $k$-cu is universally bounded; the statement about
tail entropy follows similarly but we omit the details.

\begin{Thm}\label{thm:anal.tail.ent}
  Let $f$ be as above and $l=1,\ldots,m$. For every $\delta>0$ there
  exists $k\in\N$ such that for every $x\in Y$ we have
  \begin{equation}
    h_{l,k}(f,\delta,x)\le O_m(\log L)\cdot \frac{\log|\log\delta|}{|\log \delta|}.
  \end{equation}
  In particular,
  \begin{equation}
    v_l(f,\delta)\le O_m(\log L)\cdot \frac{\log|\log\delta|}{|\log \delta|}.
  \end{equation}
\end{Thm}

\begin{proof}
  Since, by the assumptions, $||d\tilde f||\le L$, this complex
  mapping can expand distances at most $L$ times. Therefore, the
  $q$-th iterate $f^{\circ q}$ is extendible to at least a complex
  neighborhood $U_q$ of $Y$ of the size $\frac{\rho}{L^q}$, and it is
  bounded there by $D$. Applying Cauchy formula, we conclude that for
  each $x\in Y$ we have
  \begin{equation}\label{eq:compl.bound}
    \norm{d^s\tilde f^{\circ q}} \le D s! \left(\frac{L^q}{\rho}\right)^s, \qquad s=0,1,2,\ldots    
  \end{equation}
  Therefore
  \begin{equation}
    \norm{d^s\tilde f^{\circ q}}^{1/s}\le D^{1/s}\cdot s\cdot \frac{L^q}{\rho}, \qquad s=0,1,2,\ldots
  \end{equation}
  In particular, for $s=0,1,2,\ldots,k$ we have
  \begin{equation}
    \norm{d^s\tilde f^{\circ q}}^{1/s} \le \frac{DkL^q}{\rho}
    = 2^{O_m(k\log k)}
  \end{equation}
  for $k$ sufficiently big. Therefore $M_{q(k)}= 2^{O_m(k\log k)}$
  (note that the difference between the $1/s$ and $1/(s-1)$ power is
  absorbed into the asymptotic constant), and
  Corollary~\ref{cor:bound.hk2} applies with
  $\delta=2^{-O_m(k\log k)}$. Solving for fixed $\delta$, we see
  that Corollary~\ref{cor:bound.hk2} applies with
  \begin{equation}
    k=\Omega_m\left(\frac{|\log\delta|}{\log|\log\delta|}\right).
  \end{equation}
  With this $k$, Corollary~\ref{cor:bound.hk2} gives
  \begin{equation}
    h_{l,k}(f,\delta,x)\le \frac{2l\log L}{k}=
    O_m(\log L)\cdot \frac{\log|\log\delta|}{|\log \delta|}.
  \end{equation}
  as claimed.
\end{proof}

\section{Applications in diophantine geometry}

For $p\in\P^\ell(\Q)$ we define $H(p)$ to be $\max_i |\vp_i|$ where
$\vp\in\Z^{\ell+1}$ is a projective representative of $p$ with
$\gcd(\vp_0,\ldots,\vp_\ell)=1$. For $\vx\in\A^\ell(\Q)$ we define its
height to be the height of $\iota(\vx)$ for the standard embedding
\begin{equation}
  \iota:\A^\ell\to\P^\ell,\qquad \iota(\vx_{1..\ell}) = (1:\vx_1:\cdots:\vx_\ell).
\end{equation}
For a set $X\subset\P^\ell(\R)$ we denote
\begin{equation}
 X(\Q,H) := \{ p\in X\cap \P(\Q)^\ell : H(p)\le H\},
\end{equation}
and similarly for $X\subset\R^\ell$.

In \cite{bombieri-pila} Bombieri and Pila introduced a method for
studying the quantity $\#X(\Q,H)$ as a function of $H$ when $X$ is the
graph of a $C^r$ (or $C^\infty$) smooth function $f:[0,1]\to\R^2$. It
turns out that two very different asymptotic behaviors are obtained
depending on whether the graph of $f$ belongs to an algebraic plane
curve. The Yomdin-Gromov algebraic lemma has been used in both of
these directions, to generalize from graphs of functions to more
general sets. In~\secref{sec:bp-complex} we give a complex-cellular
analog of the Bombieri-Pila determinant method. We then present an
application in the algebraic context in~\secref{sec:alg-density} and
in the transcendental context in~\secref{sec:trans-density}.

\subsection{The Bombieri-Pila method for complex cells}
\label{sec:bp-complex}

The Bombieri-Pila method is based on a clever estimate for certain
\emph{interpolation determinants} of a collection of functions, which
is obtained using Taylor expansions. In this section we develop a
parallel theory over complex cells, replacing Taylor expansions by
Laurent expansions.

\subsubsection{Interpolation determinants}

Let $\mu\in\N$. Let $\vf:=(f_j:U\to\R)_{j=1,\ldots,\mu}$ be a tuple of
functions and $\vp:=(p_1,\ldots,p_\mu)\in U^\mu$ a tuple of
points. For minor technical simplifications we will assume that
$f_1\equiv 1$.

We define the \emph{interpolation determinant}
\begin{equation}
  \Delta(\vf;\vp) := \det(f_i(p_j))_{1\le i,j\le\mu}.
\end{equation}
Bombieri-Pila \cite{bombieri-pila} give an estimate for
$\Delta(\vf;\vp)$ assuming that $\vf$ have bounded $C^r$ norm (for
sufficiently large $r$) and $\vp$ are contained in a sufficiently
small ball. In our complex analytic analog the small ball will be
replaced by a cell $\cC$ with a sufficiently wide $\vdelta$-extension,
$\vf$ will be replaced by a tuple holomorphic functions on
$\cC^\vdelta$ and the $C^r$-norm will be replaced by the maximum norm
in $\cC^\vdelta$.

Let $\cC$ be a complex cell and let $m\le\dim\cC$ denote the number of
$D$-fibers. Set
\begin{equation}
  E_m:=\frac{m}{m+1}(m!)^{1/m}.
\end{equation}
Fix $0<\delta<1/2$, and let $\vdelta$ be given by $\delta$ in the
$D$-coordinates and $\delta^{E_m\mu^{1+1/m}}$ in the $A,D_\circ$
coordinates. We suppose that $\cC$ admits a $\vdelta$-extension.
  
\begin{Lem}\label{lem:bp-upper-complex}
  With $\cC$ as above, suppose
  $f_1,\ldots,f_\mu\in\cO_b(\cC^\vdelta)$ with
  $\diam(f_i(\cC^\vdelta),\C)\le M$ and $p_1,\ldots,p_\mu\in\cC$. Then
  \begin{equation}
    |\Delta(\vf;\vp)| \le M^\mu \mu^{O_\ell(\mu)} \delta^{E_m\mu^{1+1/m}-O_\ell(\mu)}.
  \end{equation}
  When $m=0$ the bound above holds\footnote{In fact one can easily
    derive better estimates in this case but this is not needed in
    this paper.} if we replace $m=0$ by $m=1$.
\end{Lem}
\begin{proof}
  Since $f_1\equiv1$ the interpolation determinant is unchanged if we
  replace each $f_j$ by $f_j-\const$ for $j\ge2$. In this way we may
  assume without loss of generality that
  $\norm{f_i}_{\cC^\vdelta}\le M$. We may also assume by rescaling
  that $M=1$.

  Write a Laurent expansion
  \begin{equation}\label{eq:bp-f_i-exp}
    f_i(\vz) = \left(\sum_{\valpha\in\Z^\ell,|\valpha|<E_m \mu^{1+1/m}} c_{i,\valpha} \vz^{[\valpha]}\right)+R_i(\vz)
  \end{equation}
  and note that by Proposition~\ref{prop:norm-Laurent} we have
  $|c_{i,\valpha}|<\vdelta^{\pos\valpha}$ and
  \begin{equation}\label{eq:bp-R_i-estimate}
    \norm{R_i}_{\cC} \le \sum_{|\valpha|\ge E_m\mu^{1+1/m}} \vdelta^{\pos\valpha} \le 
    \sum_{k\ge E_m\mu^{1+1/m}} (2k+1)^\ell \delta^k = O_\ell(\mu^{2\ell}) \delta^{E_m \mu^{1+1/m}}
  \end{equation}
  where the final two estimates are elementary and left to the
  reader. We expand the determinant $\Delta(\vf,\vp)$ multilinearly
  into $\mu^{O_\ell(\mu)}$ determinants $\Delta_I$ where each $f_i$ is
  replaced by one of the summands in~\eqref{eq:bp-f_i-exp}. We may
  consider only those $\Delta_I$ where each normalized monomial
  appears at most once: otherwise the determinant has two identical
  columns. We claim that each such $\Delta_I$ is majorated in $\cC$ by
  $\mu^{O_\ell(\mu)}\delta^{E_m \mu^{1+1/m}}$.

  Consider first the case that $\Delta_I$ contains a residue $R_i$ or
  one of the coefficients $c_{i,\valpha}$ where $\valpha$ contains a
  non-vanishing $A,D_\circ$ coordinate. Expand $\Delta_I$ into $\mu!$
  summands. Each summand contains either a term $R_i(p_j)$ or
  $c_{i,|\valpha|}$ as above, and the estimate then follows
  from~\eqref{eq:bp-R_i-estimate} or from our assumption on
  $\vdelta$.

  The remaining determinants $\Delta_I$ involve only the $m$ variables
  of type $D$, which we assume for simplicity of the notation are
  given by $\vz_{1..m}$. We expand $\Delta_I$ into a $\mu!$
  summands. Each summand is a product
  \begin{equation}
    c_{1,\valpha^1}\vz^{[\valpha^1]}\cdots c_{\mu,\valpha^\mu}\vz^{[\valpha^\mu]}, \qquad \valpha^1,\ldots,\valpha^\mu\in\N^m
  \end{equation}
  with distinct $\valpha^j$. By the estimate on the Laurent
  coefficients this summand is bounded by
  $\delta^{\sum_{i,j}\valpha_i^j}$. Let $L_m(k)$ denote the number of
  monomials of degree $k$ in $m$ variables. The minimal term
  $\delta^q$ will clearly be obtained if we choose $L_m(0)$ of the
  $\valpha^j$s of degree 0, $L_m(1)$ of degree $1$ and so on. Let
  $\nu$ be the largest integer satisfying
  $\sum_{k=0}^\nu L_m(k) \le \mu$. Then
  $q\ge\sum_{k=0}^\nu L_m(k)\cdot k$. Simple computations give
  \begin{equation}
    \begin{gathered}
      L_m(k) = \frac{k^{m-1}}{(m-1)!}+O_m(k^{m-2}) \qquad  \sum_{j=0}^k L_m(j) = \frac{k^m}{m!} + O_m(k^{m-1}) \\
      \mu =\frac{\nu^m}{m!}+O(\nu^{m-1}) \qquad  \nu = (m!\mu)^{1/m}+O(1)        
    \end{gathered}
  \end{equation}
  and finally $q=E_m\mu^{1+1/m}-O(\mu)$ as claimed.
\end{proof}

Lemma~\ref{lem:bp-upper-complex} gives essentially the same estimate
as the classical Bombieri-Pila method, \emph{but with dimension $m$
  instead of $\dim\cC$}. The key point is that for $\vdelta$
corresponding to $\hrho$-extensions with $\rho\ll1$ the
$D_\circ,A$-coordinates of $\vdelta$ have order $2^{-1/\rho}$
(compared with order $\rho$ for the $D$ coordinates), which allows us
to produce cellular covers satisfying the conditions of
Lemma~\ref{lem:bp-upper-complex}. Roughly one may say that the
punctured discs and annuli that we produce are so thin that they may
be ignored for the purposes of the Bombieri-Pila method. More
accurately we have the following version of the refinement theorem
(Theorem~\ref{thm:cell-refinement}).

\begin{Lem}\label{lem:refinement-special}
  Let $\cC^{1/2}$ be a (real) complex cell. Let $\mu\in\N$ and
  $0<\delta<1$. Then there exists a (real) cellular cover
  $\{f_j:\cC^{\vdelta_j}_j\to\cC^{1/2}\}$ of size
  \begin{equation}
    N =
    \begin{cases}
      \poly_\ell(\mu\log(1/\delta))\cdot\delta^{-2\dim\cC} & \text{for complex covers} \\
      \poly_\ell(\mu\log(1/\delta))\cdot\delta^{-\dim\cC} & \text{for real covers}
    \end{cases}
  \end{equation}
  where each $f_j$ is a cellular translate map, and $\vdelta_j$ is
  defined as in Lemma~\ref{lem:bp-upper-complex} (for $\cC_j$).

  If $\cC^{1/2}$ varies in a definable family (and $\delta$ varies under the
  condition $0<\delta<1$) then the cells $\cC_j$ and maps $f_j$ can
  also be chosen from a single definable family. If $\cC$ is algebraic
  of complexity $\beta$ then $\cC_j,f_j$ are algebraic of complexity
  $\poly_\ell(\beta)$.
\end{Lem}
\begin{proof}
  We begin by constructing, using the refinement theorem
  (Theorem~\ref{thm:cell-refinement}), a cellular cover
  $\{g_j:\cC^\hrho_j\to\cC^{1/2}\}$ such that
  $2^{-1/\rho}=O_\ell(\delta^{E_m\mu^{1+1/m}})$ holds for every
  $m=1,\ldots,\dim\cC$. The size of this cover is
  $\poly_\ell(\mu\log(1/\delta))$. We now fix some $g=g_j$. It will be
  enough to construct a cellular cover
  $\{f_j:\cC^{\vdelta_j}_j\to\cC_j^{\hrho}\}$ as in the conclusion of
  the lemma. Since each $D_\circ,A$ coordinate already has the desired
  extension, we need only refine the $D$ fibers to achieve a
  $\delta$-extension. This is elementary: after rescaling each
  $D$-fiber to $D(1)$ we cover $D(1)$ (resp. $(-1,1)$) by
  $\delta^{-2}$ (resp. $\delta^{-1}$) discs (resp. real discs) of
  radius $\delta$.
\end{proof}

\subsubsection{Polynomial interpolation determinants}

Let $\Lambda\subset\N^\ell$ be a finite set and set
$d:=\max_{\valpha\in\Lambda}|\valpha|$ and $\mu:=\#\Lambda$. For minor
technical simplifications we assume ${\mathbf 0}\in\Delta$. Let
$\vf:=(\vf_1,\ldots,\vf_\ell)$ be a collection of functions with a common
domain $U$ and $\vp:=(\vp_1,\ldots,\vp_\mu)\in U^\mu$ a collection of
points. We define the \emph{$\Lambda$-interpolation determinant}
\begin{equation}
  \Delta^\Lambda(\vf;\vp) := \Delta(\vg;\vp), \qquad \vg:=(\vf^\valpha:\valpha\in\Lambda).
\end{equation}
Basic linear algebra shows that for $S\subset U$ the set
$\vf(S)\subset\C^\ell$ belongs to a hypersurface $\{P=0\}$ with
$\supp P\subset\Lambda$ if and only if $\Delta^\Lambda(\vf;\vp)$
vanishes for any $\vp\subset S$ of size $\mu$.

The following is due to Bombieri and Pila \cite{bombieri-pila}.

\begin{Lem}\label{lem:bp-lower}
  Let $H\in\N$ and suppose $H(\vf(\vp_j))\le H$ for any
  $j=1,\ldots,\mu$. Then $\Delta^\Lambda(\vf;\vp)$ either vanishes or
  satisfies
  \begin{equation}
    |\Delta^\Lambda(\vf;\vp)| \ge H^{-\mu d}.
  \end{equation}
\end{Lem}
\begin{proof}
  Let $0<Q_j\le H$ denote the common denominator of $\vf(\vp_j)$. The
  column corresponding to $\vp_j$ in $\Delta^\Lambda(\vf;\vp)$
  consists of rational numbers with common denominator dividing
  $Q_j^d$. Factoring this common denominator from each column we
  obtain a matrix with integer entries whose determinant is either
  vanishing or an integer. In the latter case
  \begin{equation}
    |\Delta^\Lambda(\vf;\vp)| \ge \prod_{j=1}^\mu Q_j^{-d} \ge H^{-\mu d}.
  \end{equation}
\end{proof}

\subsection{Rational points on algebraic hypersurfaces}
\label{sec:alg-density}

\subsubsection{Previous work}
\label{sec:alg-density-history}

Let $X\subset\P^2(\C)$ be an irreducible algebraic curve of degree
$d$. Let $f:I\to X$ be a $C^r$-smooth function with $\norm{f}_r\le1$
and write $X_f:=f(I)$. Pila~\cite{pila:density-Q} proves, using a
generalization of the method of \cite{bombieri-pila}, that for a
sufficiently large $r=r(d)$ one has $\#X_f(\Q,H)=O_{\e,d}(H^{2/d+\e})$
for any $\e>0$.

If one allows $r=r(d,H)$, for instance if $f$ is $C^\infty$ smooth,
then the Bombieri-Pila method can be pushed further to replace
$H^{2/d+\e}$ by $H^{2/d}\log^kH$ for some $k>0$. The study of
$\#X(\Q,H)$ is therefore directly related to the $C^r$ parametrization
complexity of $X$ in the sense of the Yomdin-Gromov algebraic
lemma. Using an explicit parametrization construction Pila
\cite{pila:pems} proved that
$\#X(\Q,H)\le(6d)^{10}4^dH^{2/d}(\log H)^5$.

In \cite{heath-brown:density} Heath-Brown has developed a
non-Archimedean version of the Bombieri-Pila method and used it to
derive the estimate $\#X(\Q,H)=O_{\e,d}(H^{2/d+\e})$. This was later
improved (in arbitrary dimension $\ell$) by Salberger
\cite{salberger:density} to $O_{\ell,d}(H^{2/d}\log H)$ and very
recently by Walsh \cite{walsh:boundd-rational} to
$O_{\ell,d}(H^{2/d})$.

For a general irreducible variety $X\subset\P(\C)^\ell$ of dimension
$m$, Broberg \cite[Theorem~1]{broberg:note} (generalizing a result of
Heath-Brown \cite{heath-brown:density}) has proved that $X(\Q,H)$ is
contained in $O_{\ell,d,\e}(H^{(m+1)d^{-1/m}+\e})$ hypersurfaces
$\{P_i=0\}$, which have degrees $O_{\ell,d,\e}(1)$ and which do not
contain $X$. Intersecting $X$ with each of the hypersurfaces one
obtains varieties of dimension $m-1$, and could potentially proceed by
induction to eventually obtain estimates on $\#X(\Q,H)$. Marmon
\cite{marmon} has recovered this result using the Bombieri-Pila method
by appealing to the Yomdin-Gromov algebraic lemma.

\subsubsection{Proof of Theorem~\ref{thm:improved-marmon}}
\label{sec:alg-density-proof}

In this section we will prove Theorem~\ref{thm:improved-marmon} using
the complex-cellular version of the Bombieri-Pila method. We remark
that a result similar to Theorem~\ref{thm:improved-marmon} could
probably also be proved by following the argument of \cite{marmon} and
replacing the Yomdin-Gromov algebraic lemma by our refined
version.

We will work in the affine setting. Let $\iota_j$ denote the $j$-th
standard chart $\vx_{1..\ell}\to(\cdots:\vx_j:1:\vx_{j+1}:\cdots)$. If
$\vz=(\vz_0:\cdots:\vz_{\ell})\in\P^\ell(\C)$ with $|\vz_j|$ maximal
then in the $\iota_j$-chart $\vz$ corresponds to a point
$\vx\in\bar\cP_\ell$, and moreover $H(\vz)=H(\vx)$. We assume without
loss of generality that $j=0$. Below $\vx$ denotes the affine
coordinates in this chart. It will suffice to count the points of
height $H$ in $\bar\cP_\ell\cap X_0$ where $X_0:=\iota_0^{-1}(X)$.

We fix some degree-lexicographic monomial ordering on
$\C[\vx_{1..\ell}]$ and let $\Lambda:=\N^\ell\setminus\LT(I)$ where
$I\subset\C[\vx_{1..\ell}]$ denotes the ideal of $X_0$ and $\LT(I)$ the
set of leading exponents of $I$ with respect to the fixed
ordering. Set $\Lambda(k):=\Lambda\cap\{|\valpha|\le k\}$. Then
$\mu(k):=\#\Lambda(k)$ is the Hilbert function of $X_0$. This function
eventually equals the Hilbert polynomial of $I$, and by Nesterenko's
estimate \cite{nesterenko:hilbert} we have
\begin{equation}
  \mu(k) \ge \binom{k+m+1}{m+1} - \binom{k+m+1-d}{m+1} = dk^m/m!+\poly_m(d)k^{m-1}.
\end{equation}
We will apply the Bombieri-Pila method with $\Lambda(k)$-interpolation
determinants.

Let $\cC^\vdelta$ be a complex cell as in
Lemma~\ref{lem:bp-upper-complex} and
$f:\cC^\vdelta\to\cP_\ell^{1/2}\cap X_0$
be a cellular map. Let $\vp\subset\cC$ be a $\mu$-tuple of points
satisfying $H(f(\vp_j))\le H$ for $j=1,\ldots,\mu$. Then by
Lemmas~\ref{lem:bp-upper-complex} and~\ref{lem:bp-lower} we have
\begin{equation}\label{eq:bp-comparison-alg}
  H^{-\mu k} \le |\Delta^{\Lambda(k)}(f;\vp)| \le 2^\mu \mu^{O_\ell(\mu)} \delta^{E_m\mu^{1+1/m}-O_\ell(\mu)}
\end{equation}
unless $\Delta^{\Lambda(k)}(f;\vp)=0$. Note that in the right hand
side of~\eqref{eq:bp-comparison-alg}, Lemma~\ref{lem:bp-upper-complex}
gives the estimate with $m$ equal to the number of disc fibers in
$\cC$ (which is at most $\dim\cC$), and this implies the same estimate
with $m=\dim X_0$. Thus unless $\Delta^{\Lambda(k)}(f_{1..\ell};\vp)=0$ we have,
for $k>\poly_m(d)$,
\begin{multline}
  -\log\delta <
  \frac{O_\ell(\log\mu) + k\log H }{E_m\mu^{1/m}+O_\ell(1)} < \\
  \frac{k\log H + O_\ell(\log k)}{E_m\mu^{1/m}+O_\ell(1)} <
  \frac{m+1}{m} \cdot\frac{k\log H+O_\ell(\log k)}{kd^{1/m} (1+\poly_m(d)/k)^{1/m}} < \\
  O_\ell(1)+\frac{m+1}{m}\cdot\frac{\log H}{d^{1/m}} (1+\poly_\ell(d)/k)
\end{multline}
where we used $\mu(k)=\poly_\ell(k)$. Therefore, for
\begin{equation}
  \delta = O_\ell(1)\cdot H^{-\frac{m+1}m d^{-1/m}(1+\poly_\ell(d)/k)}
\end{equation}
we have $\Delta^{\Lambda(k)}(f_{1..\ell};\vp)=0$ for any $\vp$ as
above, and $[f(\cC)](\Q,H)$ is therefore contained in a hypersurface
$\{P=0\}$ satisfying $\supp P\subset\Lambda(k)$. In particular $P$ is
of degree at most $k$ and does not vanish identically on $X_0$.

We use the CPT to construct a real cellular cover of
$\R(\cP_\ell^{1/2}\cap X_0)$ of size $\poly_\ell(d)$ admitting
$1/2$-extensions, and then apply Lemma~\ref{lem:refinement-special} to
further refine each of the cells to satisfy the conditions of
Lemma~\ref{lem:bp-upper-complex}. In this way we obtain
\begin{multline}
  N = \poly_\ell(d) \cdot \delta^{-m}\poly_\ell(\mu\log(1/\delta)) =\\
  \poly(d,k,\log H)\cdot H^{(m+1) d^{-1/m}(1+\poly_\ell(d)/k)}
\end{multline}
real cells, and applying the construction above to each of them
produces the required collection of hypersurfaces.

\subsection{Rational points on transcendental sets}
\label{sec:trans-density}

\subsubsection{Log-sets in diophantine applications of the Pila-Wilkie
theorem}
\label{sec:log-sets-pila}

In some of the most remarkable applications of the Pila-Wilkie
theorem, particularly those related to modular curves and more general
Shimura varieties, it is necessary to count rational points on sets
that are definable in $\R_{\an,\exp}$ but \emph{not} in $\R_\an$. We
briefly recall the most standard example coming from the universal
covers of modular curves.

Recall that the universal covering map
$j:\H\to\SL_2(\Z)\backslash \H\simeq\C$ can be factored as a
composition $j(\tau)=J(q)$ where $q=e^{2\pi i\tau}$ and
$J(q):D_\circ(1)\to\C$ is meromorphic in a neighborhood of $q=0$. We
extend $j$ and $J$ as functions of several variables coordinate-wise,
and denote by $\Omega\subset\H^n$ the (coordinate-wise) standard
fundamental domain for $j$. In \cite{pila:andre-oort} a key step is
the application of the Pila-Wilkie theorem to sets of the form
$Y:=\Omega\cap j^{-1}(X)$ where $X\subset\C^n$ is an algebraic
variety. This set is not (in general) subanalytic: it is of the form
$Y=\Omega\cap \log A$ where $\log(\cdot)=(2\pi i)^{-1}\log_e(\cdot)$
and $A:=D(1/2)\cap J^{-1}(X)$. Note that since $J$ is meromorphic in
the closure of $D(1/2)$ the set $A$ is subanalytic, so that $Y$ is a
log-set.

\subsubsection{An interpolation result for logarithms of complex cells}

Let $\cC^{\vdelta/2}$ be a complex cell and let $n:=\dim\cC$ and $m$
denote the number of fibers of type $D$. Let $\vf:=(f_0,\ldots,f_n)$
be a tuple of non-vanishing holomorphic functions on
$\cC^{\vdelta/2}$. We write $\vx_{0..n}=\log\vf_{0..n}$ and denote
$X:=\vx(\cC)$.

Recall that $\pi_1(\cC)\simeq\Z^{n-m}$, so we may take
$\valpha(\vf_j)\in\Z^{n-m}$ for $j=0,\ldots,n$. Then there exist $m+1$
vectors $\vgamma^0,\ldots,\vgamma^m\in\Z^{n+1}$ linearly independent
over $\Z$ and satisfying $\sum_{j=0}^n \vgamma^k_j\valpha(\vf_j)=0$
for every $k=0,\ldots,m$. We write
$\vg_k = \prod_j \vf_j^{\vgamma^k_j}$. Then $\vg_0,\ldots,\vg_m$ are
non-vanishing holomorphic functions on $\cC^{\vdelta/2}$ with trivial
associated monomials. Then the logarithms
$\vy_{0..m}:=\log\vg_{0..m}$, which are $\Z$-linear combinations of
the $\vx$-variables, are holomorphic univalued functions in
$\cC^{\vdelta/2}$.

\begin{Prop}\label{prop:bp-logs}
  Let $k,H\in\N$ and $\delta>0$. Suppose that $\cC^\vdelta$
  satisfies the conditions of Lemma~\ref{lem:bp-upper-complex} with
  $\delta$ and $\mu\sim k^{m+1}$. If $k=\Omega_\ell(1)$ and
  \begin{equation}\label{eq:bp-logs-cond}
    \log\delta < -O_\vf(k^{-1/m}\log H)
  \end{equation}
  then $X(\Q,H)$ is contained in an algebraic hypersurface
  $\{P(\vy)=0\}$ of degree at most $k$ in $\C^{n+1}$.

  If $\cC^\vdelta,\vf$ vary over a definable family such that the type of
  $\cC$ and the associated monomials of $\vf_{0..n}$ are fixed then
  the asymptotic constants can be taken uniformly over the family.
\end{Prop}
\begin{proof}
  Since $\vy$ are fixed $\Z$-linear combinations of $\vx$ we have
  $H(\vy)=O_\vf(H(\vx))$. Thus it will be enough to prove the
  statement for $Y(\Q,H)$ instead of $X(\Q,H)$ where $Y:=\vy(\cC)$.
  
  By the monomialization lemma (Lemma~\ref{lem:monomial}) the diameter
  of $\vy_j(\cC^\vdelta)$ is bounded in $\C$ for $j=0,\ldots,m$ by
  some quantity $M=O_\vf(1)$ which can be taken uniform over definable
  families. We are thus in position to apply
  Lemma~\ref{lem:bp-upper-complex} with
  \begin{equation}
    \Lambda(k)=\{\valpha\in\N^{m+1}:|\valpha|\le k\}, \qquad \mu(k)\sim k^{m+1}.
  \end{equation}
  Let $\vp\subset\cC$ be a $\mu$-tuple of points satisfying
  $H(\vy(\vp_j))\le H$ for $j=1,\ldots,\mu$. Then by
  Lemmas~\ref{lem:bp-upper-complex} and~\ref{lem:bp-lower} we have
  \begin{equation}
     H^{-\mu k} \le |\Delta^{\Lambda(k)}(\vy;\vp)| \le M^\mu \mu^{O_\ell(\mu)} \delta^{E_m\mu^{1+1/m}-O_\ell(\mu)}
  \end{equation}
  unless $\Delta^{\Lambda(k)}(\vx;\vp)=0$. In the former case we have
  \begin{equation}
    \log\delta > \frac{-O_\vf(k\log H)}{k^{1+1/m}-O_\ell(1)}=-O_\vf(k^{-1/m}\log H).
  \end{equation}
  Therefore, for $\delta$ satisfying~\eqref{eq:bp-logs-cond} we have
  $\Delta^{\Lambda(k)}(\vy;\vp)=0$ for any $\vp$ as above, and
  $Y(\Q,H)$ is indeed contained in a hypersurface of degree at most
  $k$ in the $\vy$-variables.
\end{proof}

\begin{Rem}[Logarithms in families]\label{rem:logs-in-family}
  We remark that if $g_\lambda:\cC_\lambda^{1/2}\to\C\setminus\{0\}$
  is a definable family of functions with trivial associated monomials
  then $\log g_\lambda:\cC_\lambda\to\C$ is not necessarily definable
  (in $\R_\an$) as illustrated by the example $g_\lambda\equiv\lambda$
  for $\lambda\in(0,1]$. However, if we define
  $\tilde g_\lambda=g_\lambda/\norm{g_\lambda}_{\cC_\lambda}$ then
  $\log\tilde g_\lambda$ is definable. Indeed by the monomialization
  lemma (Lemma~\ref{lem:monomial}) we know that
  $w_\lambda=\log \tilde g_\lambda(\cC)\subset D(M)$ for some
  uniformly bounded $M$. Then the graph of a (univalued) branch of
  $\log\tilde g_\lambda$ is definable, being one of the components of
  the definable set $e^{w_\lambda}=\tilde g_\lambda(\vz)$ with
  $w\in D(M)$ and $\vz\in\cC_\lambda$ (note that the exponential here
  is \emph{restricted} to a compact set).
\end{Rem}

\subsubsection{Proof of Proposition~\ref{prop:log-set-pw}}
\label{sec:log-set-pw-proof}

We apply Corollary~\ref{cor:cpt-subanalytic} to the subanalytic set
given by the total space $X\subset[0,1]^K$ of the family $\{X_\lambda\}$,
where we order the parameter variables before the fiber variables. We
obtain real cellular maps $f_j:\cC^{1/2}_j\to\cP_K^{1/2}$ such that
$f_j(\R_+\cC_j^{1/2})\subset X$ and $\cup_j f_j(\R_+\cC_j)=X$. We may
also require by Remark~\ref{rem:cpt-semi-extra} that each $f_j$ is
compatible with the fiber variables $\vx_{1..\ell}$. It will be enough
to prove the statement for each $f_j(\R_+\cC_j)$ separately, so fix
$\cC=\cC_j$ and $f=f_j$ and assume $X=f(\R_+\cC)$.

We view $\cC$ as a subanalytic family of cells $\cC_\lambda$ of length
$\ell$. If the type of $\cC_\lambda$ has dimension strictly less then
$n$ then we can take $Y=X$, so assume it has dimension $n$. Below we
let $\vf_{0..n}$ denote the $f$-pullback of some fixed $n+1$-tuple of
the variables $\vx_{1..\ell}$. Note that the associated monomials
$\valpha(\vf_j)$ of $\vf_j$ on $\cC_\lambda$ are independent of
$\lambda$ (it is obtained from the associated monomial of $\vf_j$ on
$\cC$ by eliminating the $\lambda$ variables).

Let $\cR=\{f_{\lambda,\theta}:\cC_{\lambda,\theta}\to\cC_\lambda\}$
denote the refinement family consisting of the refinement cells of
$\cC_\lambda$ as in Lemma~\ref{lem:refinement-special}. We split $\cR$
into subfamilies $\cR(\cT)$ consisting of the cells
$\cC_{\lambda,\theta}$ of type $\cT$. According to
Remark~\ref{rem:refinement-vs-monom} the types $\cT$ are those
obtained from the type of $\cC_\lambda$ by possibly replacing some
$D_\circ,A$ fibers by $D$. Fix some $\cT$ and let $m$ denote the
number of fibers of type $D$. The associated monomial of
$f^*_{\lambda,\theta}\vf_j$ is constant over $\cR(\cT)$: it is
obtained from $\valpha(\vf_j)$ by eliminating those indices that
correspond to fibers that were replaced by $D$ in $\cT$.

We can thus apply Proposition~\ref{prop:bp-logs} to
$f_{\lambda,\theta}^*\vf_{0..n}$ to deduce that if the cell
$\cC_{\lambda,\theta}$ admits $\vdelta/2$ extension and if
$\delta = H^{-O_X(k^{-1/m})}$ then
\begin{equation}
  [\vx_{\lambda,\theta}(\cC_{\lambda,\theta})](\Q,H)\subset\{P(\vy_{\lambda,\theta})=0\}
\end{equation}
for some polynomial $P(\vy_{\lambda,\theta})$ of degree $k$, where
$\vy_{\lambda,\theta}$ are some \emph{fixed} $\Z$-linear combinations
of $\vx_{\lambda,\theta}:=\log (f^*_{\lambda,\theta}\vf_{0..n})$ which
are holomorphic in $\cC_{\lambda,\theta}^{\vdelta/2}$. Note that
$\vy_{\lambda,\theta}$ does \emph{not} necessarily depend definably on
the parameters. However, the normalized $\tilde\vy_{\lambda,\theta}$
defined by replacing $\vx_{\lambda,\theta}$ with their normalized
versions $\tilde\vx_{\lambda,\theta}$ is definable according to
Remark~\ref{rem:logs-in-family}. For each fixed value of the
parameters we have
$\tilde\vy_{\lambda,\theta}=\vy_{\lambda,\theta}+\const$, and in
particular the hypersurface $\{P(\vy_{\lambda,\theta})=0\}$ can be
rewritten as a hypersurface
$\{\tilde P(\tilde\vy_{\lambda,\theta})=0\}$ in the
$\tilde\vy_{\lambda,\theta}$-variables.

We now define a family $\tilde Y_{\lambda,\theta,c}$ with the
parameter $\theta$ of $\cR$ and the parameter $c$ encoding the
coefficients of an arbitrary non-zero polynomial $P_c$ of degree $k$,
\begin{equation}
  \tilde Y_{\lambda,\theta,c} := \R (f_{\lambda,\theta}(\cC_{\lambda,\theta})\cap \{P_c(\tilde\vy_{\lambda,\theta})=0\}).
\end{equation}
For any $\lambda$ we can choose, according to
Lemma~\ref{lem:refinement-special} a collection of
\begin{equation}
  N=\poly_\ell(k,\delta^{-1})=\poly_\ell(k)\cdot H^{O_X(k^{-1/n})}
\end{equation}
cells $\cC_{\lambda,\theta_j}$ which admit $\delta$-extensions,
satisfy the conditions of Lemma~\ref{lem:bp-upper-complex}, and cover
$\cC_\lambda$. For each of them there exists a polynomial
$P_{c_j}(\tilde\vy_{\lambda,\theta_j})$ whose zeros contain
$[\vx_{\lambda,\theta}(\cC_{\lambda,\theta})](\Q,H)$. The union over
$j=1,\ldots,N$ of $\log \tilde Y_{\lambda,\theta_j,c_j}$ thus contains
$(\log X_\lambda)(\Q,H)$.

Recall that $\tilde Y$ was constructed for some choice of $n+1$ of the
coordinates $\vx_{1..\ell}$. We now repeat this construction to obtain
$\tilde Y^S$ for each such choice $S$. We define a family
$\{\hat Y_{\lambda,\mu}\}$, with the parameter $\mu$ encoding the
parameters $(\theta_S,c_S)$ for each choice of $S$, and the
fiber given by
\begin{equation}
  \hat Y_{\lambda,\mu} = \bigcap_S \tilde Y^S_{\lambda,\theta_S,c_S}.
\end{equation}
As before, for every $\lambda$ there is a choice of
$N=\poly_\ell(k) H^{O_X(k^{-1/n})}$ parameters $\mu_j$ such that the
union of $(\log \hat Y_{\lambda,\mu_j})(\Q,H)$ contains
$(\log X_\lambda)(\Q,H)$.

If $\hat Y_{\lambda,\mu}$ has dimension $n$ then $X_\lambda$ satisfies
a polynomial equation of degree $k$ in $\log\vf_{0..n}$ for every
$n+1$ tuple of coordinates $\vf_{0..n}$ among $\vx_{1..\ell}$ and in
this case $\log X_\lambda$ is contained in an algebraic variety of
dimension $n$. Since $X_\lambda$ has pure dimension $n$ this implies
that $(\log X_\lambda)^\trans=\emptyset$. Thus if we define $Y$ by
removing from $\hat Y$ any fibers of dimension $n$, the conditions of
the proposition are satisfied.

\newpage

\section{List of notations and definitions}
\label{sec:notations}

The following table lists some of the main notations and definitions
used in this paper. We do not include the notations used in the two
appendices.  \vspace{1cm}

\begin{tabular}{|c|p{0.5\textwidth}|c|}
  \hline
  Notation & Meaning & Definition \\ \hline

  $\norm{\cdot}_U$ & Maximum norm on $U$ & \secref{sec:intro-algebraic-lemma}
  \\
  $\norm{\cdot}_r$ & Smooth $r$-norm & Equation~\eqref{eq:r-norm-def}
  \\
  $(A,K)-mild$ & Mild parametrization & Definition~\ref{def:mild}
  \\
  $H(\cdot)$ & Height of rational point & \secref{sec:intro-rational}
  \\
  $X(\Q,H)$ & Points of height $H$ in $X$ &
  \\
  $X^\alg,X^\trans$ & Algebraic and transcendental parts of $X$ &
                                                                  \secref{sec:intro-pw}
  \\
  Log-set & & Definition~\ref{def:log-set}
  \\
  $D,D_\circ,A,*$ & Basic fibers types &
                                                        Equation~\eqref{eq:basic-fibers}
  \\
  $\cF^\delta,\cF^\hrho$ & Extensions of basic fibers &
                                                        Equations~\eqref{eq:fiber-delta-ext}
                                                        and~\eqref{eq:rho-ext-def}
  \\
  $\cX\odot\cF$ & Fibered dot-product & Equation~\eqref{eq:odot-def}
  \\
  $\cO(U),\cO_b(U)$ & Holomorphic (bounded) functions on $U$ &
                                                             \secref{sec:cells-general-setting}
  \\
  $\cC$ & Complex cell & Definition~\ref{def:cells}
  \\
  $\cC^\vdelta,\cC^{\he\vrho}$ & Extensions of cells &
                                                       Definition~\ref{def:cell-ext}
  \\
  $\R\cC,\R_+\cC$ & (Positive) real part of a cell &
                                                     Definition~\ref{def:real-cells}
  \\
  $f:\cC\to\hat\cC$ & (Prepared) cellular map & Definition~\ref{def:cell-maps}
  \\
  Compatible & & Definition~\ref{def:compatible}
  \\
  $\valpha(f)$ & Associated monomial of $f$ &
                                              Definition~\ref{def:assoc-monom}
  \\
  $\dist(\cdot,\cdot;X)$ & Distance in $X$ &
                                             \secref{sec:cell-topology-geometry}
  \\
  $\diam(A;X)$ & The diameter of $A$ in $X$ &
  \\
  $B(x,r;X)$ & Ball of radius $r$ around $x$ &
  \\
  $\cC_{\times\nu},R_\nu$ & $\nu$-cover of $\cC$ &
                                             Definition~\ref{def:nu-cover}
  \\
  $\cC_{\times\vnu,\vsigma},R_{\vnu,\vsigma}$ & Signed covers &
                                                             \secref{sec:nu-cover}
  \\
  $\vz^{[\valpha]}$ & Normalized monomial & \secref{sec:laurent}
  \\
  $\cP_\ell$ & Unit polydisc in $\C^\ell$ &
  \\
  $\cS(\cC)$ & Skeleton of $\cC$ & Definition~\ref{def:skeleton}
  \\
  $V_\Gamma(f)$ & Voorhoeve index of $f$ along $\Gamma$ &
                                                          \secref{sec:proof-monom}
  \\
  Monomial cell & & Definition\ref{def:monom-cell}
  \\
  $\cF_{i,q},\cF_{i,q+}$ & Clusters around $y_i$ &
                                                   Proposition~\ref{prop:clusters}
  \\
  Weierstrass cell & & Definition~\ref{def:weierstrass-cell}
  \\
  $(p,M)$ domination & Laurent domination property &
                                                     Definition~\ref{def:domination}
  \\
  $\Qua\cC$ & Positive quadrant of $\cC$ & \secref{sec:quadric-cells}
  \\
  $I_\vsigma$ & Inversion map on $\cC$ &
  \\
  $\sec_\ell(\e)$ & Polysector in $\cC^\ell$ &
                                               Equation~\eqref{eq:polysector}
  \\
  $B(\e)$ & Sectorial extension of cube & \secref{sec:sectorial-param}
  \\
  \hline                                         
\end{tabular}

\newpage

\bibliographystyle{plain}
\bibliography{nrefs}

\end{document}